\documentclass[english]{amsart} 

\usepackage{amsfonts,amsmath,amssymb,color,amscd,amsthm}
\usepackage[T1]{fontenc} 
\usepackage[english]{babel}
\usepackage{mathtools}\usepackage{gensymb}
 
\usepackage{pgf,tikz}
\usepackage{pgfplots} \usepgfplotslibrary{patchplots}
\usetikzlibrary{cd,arrows,positioning}
\tikzset{>=stealth}
\tikzcdset{arrow style=tikz}
\tikzset{link/.style={column sep=1.8cm,row sep=0.16cm}}
\tikzset{map/.style={row sep=0em, column sep=0em}}

\usepackage{soul}

\usepackage{tikz}
\usetikzlibrary{shapes,arrows}

\usepackage[curve,all]{xy}
\xyoption{matrix}
 \xyoption{curve}
 \xyoption{color}
 \xyoption{line}
\xyoption{arc}

\usepackage[shortlabels]{enumitem}
\setlist[1]{wide}
\setlist[2]{leftmargin=15mm}
\setlist[enumerate]{label=\rm{(\arabic*)}}
\setlist[enumerate,2]{label=\rm({\it\roman*}), }
\setlist[itemize]{label=\raisebox{0.25ex}{\tiny$\bullet$}}

\usepackage{thmtools}
\usepackage{thm-restate}
\usepackage[backref, colorlinks, linktocpage, citecolor = blue, linkcolor = blue]{hyperref}

\usepackage{cleveref}

\pgfplotsset{
        compat=1.11,
        My Line Style/.style={
            smooth,
            thick,
            samples=400,
        },
    }
    \definecolor{vlightgray}{rgb}{0.9, 0.9, 0.9}

\newcommand{\Polygone}[3]{\addplot3[patch,patch type=polygon, vertex count=#2,color=#3,faceted color=black]
coordinates {#1};}
\newcommand{\LLine}[1]{\addplot3[patch,patch type=line,color=black,faceted color=gray,style=dashed]
coordinates {#1};}
\newcommand{\PLine}[1]{\addplot3[patch,patch type=line,color=white,faceted color=black,style=thick]
coordinates {#1};}

\theoremstyle{plain}
\newtheorem{theoremA}{Theorem}

\newcommand{\tr}[1]{\vphantom{#1}^{t}\!#1}

\theoremstyle{plain}
\newtheorem{theorem}{Theorem}[subsection]
\newtheorem{lemma}[theorem]{Lemma}
\newtheorem{proposition}[theorem]{Proposition}
\newtheorem{corollary}[theorem]{Corollary}

\theoremstyle{definition}
\newtheorem{definition}[theorem]{Definition}

\theoremstyle{remark}
\newtheorem{remark}[theorem]{Remark}

\newcommand\iso{\stackrel{\simeq}{\longrightarrow}}

\newcommand{\incl}[1][r]{\ar@<-0.2pc>@{^(-}[#1] \ar@<+0.2pc>@{-}[#1]}

\renewcommand{\P}{\mathbb{P}}

\newcommand{\p}{\mathbb{P}}

\newcommand{\Q}{\mathbb{Q}}

\newcommand{\A}{\mathbb{A}}

\newcommand{\C}{\mathbb{C}}
\newcommand{\F}{\mathbb{F}}
\newcommand{\N}{\mathbb{N}}

\newcommand{\Z}{\mathbb{Z}}

\newcommand{\G}{\mathbb{G}}

\renewcommand{\O}{\mathcal{O}}

\newcommand{\im}{\mathbf{i}}

\DeclareMathOperator{\Aut}{Aut}
\DeclareMathOperator{\Chow}{Chow}

\DeclareMathOperator{\Autz}{Aut^{\circ}}
\DeclareMathOperator{\PGL}{PGL}

\DeclareMathOperator{\rk}{rk}

\DeclareMathOperator{\Pic}{Pic}

\DeclareMathOperator{\Bir}{Bir}

\DeclareMathOperator{\Ker}{Ker}

\title{Connected algebraic groups acting on Fano fibrations over $\p^1$}
\date{\today}

\author{J\'er\'emy Blanc}
\author{Enrica Floris}

\begin{document}

\thanks{Both authors acknowledge support by the Swiss National Science
Foundation Grant ''Algebraic subgroups of the Cremona groups`` 200021\_159921.
Enrica Floris  is supported by the ANR Project FIBALGA ANR-18-CE40-0003-01.
}

\tikzstyle{decision} = [diamond, draw, 
    text width=3.5em, text centered,  inner sep=0pt]
    \tikzstyle{sblock} = [rectangle, draw,  
    text width=3em, text centered, rounded corners, node distance=14em, minimum height=3em, minimum width=5em]
\tikzstyle{block} = [rectangle, draw,  
   text width=10em, text centered, rounded corners, minimum height=3em, minimum width=12em]
    \tikzstyle{mblock} = [rectangle, draw,  
    text width=18em, text centered,  rounded corners, minimum height=4em, minimum width=12em] 
\tikzstyle{bblock} = [rectangle, draw,  
    text width=20em, text centered, rounded corners, minimum height=4em, minimum width=12em]    
\tikzstyle{varX} = [rectangle, draw, node distance=20em,  minimum height=3em]  
\tikzstyle{var} = [rectangle, draw, node distance=14em,  minimum height=3em]  
\tikzstyle{line} = [draw, -latex']
\tikzstyle{y} = [circle, draw, minimum height=3em]
\tikzstyle{n} = [circle, draw, node distance=14em,
    minimum height=2em]

\begin{abstract}
Let $X/\p^1$ be a Mori fibre space with general fibre of Picard rank at least two. We prove that there is a proper closed subset $S\subsetneq X$, invariant by the connected component of the identity ${\rm Aut}^{\circ}(X)$ of the automorphism group of $X$, which is moreover the orbit of a section $s$ and whose intersection with a fibre is an orbit of the subgroup of ${\rm Aut}^{\circ}(X)$ acting trivially on $\mathbb{P}^1$.

Such result is a tool to describe equivariant birational maps from $X/\p^1$ to other Mori fibre spaces and therefore finds its applications in the study of connected algebraic subgroups of ${\rm Aut}^{\circ}(X)$. 
This represents a first reduction step towards  a possible classification of maximal connected algebraic subgroups of the Cremona group of rank $4$.
\end{abstract}  

\maketitle
\setcounter{tocdepth}{1}
\tableofcontents
\section{Introduction}

In this text, we work over the field of complex numbers. By a classical result, every maximal connected algebraic subgroup of $\Bir(\p^2)$ is conjugate to either the connected group $\Aut(\p^2)$, or $\Autz(\p^1\times \p^1)$, or $\Aut(\F_n)$ for some $n\ge 2$. This result, essentially due to Enriques \cite{Enr1893}, can be now be seen easily using modern tools, by finding a smooth projective rational surface where the subgroup acts, then running a minimal model program, which in the case of surfaces is a sequence of contractions of $(-1)$-curves, and which for rational varieties gives as an outcome a \emph{Mori fibration} $X\to B$. A Mori fibration is a fibration with $\rho(X/B)=1$ and whose fibres are Fano varieties. Therefore, if $X$ is a rational surface the only possibilities are that either $X=\p^2$ and $B$ is a point, or $X$ is a Hirzebruch surface and $B=\p^1$. The fact that $\Aut(\p^2)$, $\Autz(\p^1\times \p^1)$ and $\Aut(\F_n)$, $n\ge 2$ are maximal connected algebraic subgroups of $\Bir(\p^2)$ (via a birational map to $\p^2$) is then a direct consequence of the fact that
those groups act on the respective surfaces without 
 fixed points. The study of (maximal) connected algebraic subgroups of $\Bir(X)$ for a non-rational surface $X$ can be done with the same strategy, see \cite{Fong}.

In dimension $3$, the classification of the maximal connected algebraic subgroups of the Cremona group was started by Enriques and Fano in \cite{EF98} and achieved by Umemura in a series of four papers \cite{Ume80,Ume82a,Ume82b,Ume85}.
It was recovered in \cite{BFT1,BFT2} using the minimal model program and studying the possible Mori fibrations and their automorphisms groups. In dimension 3, in many cases, the automorphism group of a Mori fibre space is very small.
Hence, the maximal connected algebraic subgroups of $\Bir(\p^3)$ correspond to $\Autz(X)$ for \emph{some} Mori fibrations $X\to B$. These are some very natural ones which are in some sense ``symmetric enough'' as they have a group $\Autz(X)$ large enough. It is interesting to determine which are the Mori fibrations realising the maximal subgroups as automorphism groups.
 In particular, if $B$ is a curve and $X$ is a rational threefold with terminal singularities such that the group $\Autz(X)$ is a maximal connected subgroup of $\Bir(X)$ (or equivalently of $\Bir(\p^3)$ via a birational map $X\dasharrow \p^3$), then either $X\to B$ is a $\p^2$-bundle or a Mori fibration with general fibres isomorphic to $\p^1\times \p^1$ but a generic fibre which is a smooth quadric of Picard rank $1$ (see  \cite[Theorems D and E]{BFT2}). Moreover, in this latter case, there are plenty of examples of maximal algebraic groups (essentially parametrised by classes of hyperelliptic curves) and each is conjugate to the group of automorphisms of infinitely many Mori fibrations $X\to \p^1$ whose generic fibre is a smooth quadric of Picard rank $1$, but is not conjugate to a subgroup of automorphisms of any Mori fibration $Y\to B$ with $\dim B\not=1$.

For the moment, the study of maximal connected algebraic subgroups of $\Bir(\p^n)$ for $n\ge 4$ (or more generally of $\Bir(X)$ for  some rationally connected varieties $X$ of dimension $\ge 4$) seems out of reach in its full generality, due to the incredibly large number of possible cases.

In this text, we focus  on the case of Mori fibre spaces $X\to \p^1$, with $X$ a terminal $\Q$-factorial variety and where the general fibres are smooth Fano varieties of Picard rank $\ge 2$. If $X$ has dimension $3$, this corresponds to the quadric fibrations described above, whose general fibres are then isomorphic to $\p^1\times \p^1$. If $X$ has dimension $4$, the general fibre is a Fano variety of dimension 3. There are $88$ deformation families of smooth Fano threefolds of Picard rank $\ge 2$ \cite{MoriMukai1,MoriMukai2} and among these, exactly $9$ occur as general fibres of klt Mori fibre spaces $X\to \p^1$ \cite[Theorem~1.4]{CFST16}. If $X$ has dimension $\ge 5$, the possible classes for the general fibres are not fully classified (see \cite{CFST16,CFST18} for partial results).

In the study of connected algebraic groups acting on projective varieties, it is natural to look for invariant subsets, as these can be used to construct equivariant birational maps from one Mori fibre space to another. In particular, if $\pi\colon X\to B$ is a Mori fibre space and $\Autz(X)$ acts transitively on $X$, every $\Autz(X)$-equivariant birational map from $X$ to any other Mori fibre space is an isomorphism. This explains  the importance played by the next  result in the study of connected algebraic groups acting on Mori fibres space over $\p^1$, which is our main result:
\begin{theoremA}\label{Thm:ExistenceSection}
 Let $\pi\colon X\to\p^1$ be a $\Q$-factorial terminal Mori fibre space such that a general fibre $F$ satisfies $\rho(F)\geq 2$. Then, the action of \[\Autz(X)_{\p^1}=\{g\in \Autz(X)\mid \pi g=\pi\}\] on a general fibre is not transitive. Moreover, there is a section $s\subset X$ of $\pi$ such that the following holds:
 \begin{enumerate}
 \item The set
 \[\mathcal{S}=\Autz(X)\cdot s=\Autz(X)_{\p^1}\cdot s=(\Autz(X)_{\p^1})^{\circ}\cdot s\] is a proper closed subset of $X$;
 \item
 For each $b\in \p^1$, the fibre $\pi^{-1}(b)\cap \mathcal{S}$ of $\pi|_{\mathcal{S}}\colon \mathcal{S}\to \p^1$ is equal to \[\pi^{-1}(b)\cap \mathcal{S}=(\Autz(X)_{\p^1})^{\circ}\cdot p,\] where $p\in s$ is the point such that $\pi(p)=b$.
 \end{enumerate}
 \end{theoremA}
 
 The proof of Theorem~\ref{Thm:ExistenceSection} is done in Section~\ref{Sec:ProofExistenceS}, by studying sets of sections of the Mori fibre space $X\to \p^1$, applying an equivariant version of Bend and Break (Proposition~\ref{bandb}), and looking at actions of the different subgroups  $\Autz(X)_{\p^1}$ and $(\Autz(X)_{\p^1})^{\circ}$ of  $\Autz(X)$ on the set of minimal sections.

One motivation for studying the group $\Autz(X)_{\p^1}$ comes from the following two observations, proven in Section~\ref{Sec:Verticalgroup}:
\begin{restatable}{propositionA}{proptorus}\label{Prop:Torus}
Let $\pi\colon X\to\p^1$ be a Mori fibre space such that a general fibre $F$ satisfies $\rho(F)\geq 2$. 
Assume that \[\Autz(X)_{\p^1}=\{g\in \Autz(X)\mid \pi g=\pi\}\] is either finite or a torus. Then $\Autz(X)$ is a torus 
of dimension $r\in \{\dim (\Autz(X)_{\p^1})$, $\dim (\Autz(X)_{\p^1})+1\}$. Moreover, if $r\ge 1$, there is a smooth projective variety $C$, a trivial $\p^r$-bundle $Y\to C$ and a birational map $\psi\colon X\dasharrow Y$ such that $\psi\Autz(X)\psi^{-1}\subsetneq \Autz(Y)$.
\end{restatable}
\begin{restatable}{propositionA}{propsmallorbit}\label{Prop:Smallorbit}
Let $\pi\colon X\to B$ be a Mori fibre space such that \[\Autz(X)_{B}=\{g\in \Autz(X)\mid \pi \circ g=\pi\}\] is a linear group of positive dimension and that no orbit of $\Autz(X)_{B}$ is dense in a fibre of $\pi$.
Then, $k=\max\{\dim((\Autz(X)_{B})^\circ\cdot x)\mid x\in X\}>0$ and there is a Mori fibre space $Y\to C$ with $\dim C\ge \dim X-k>\dim B$, and an $\Autz(X)$-equivariant birational map $X\dasharrow Y$.
\end{restatable}

As an application, we obtain the following result on Mori fibre spaces of dimension $4$. 
\begin{theoremA}\label{Thm:Dimension4}
 Let $\pi\colon X\to\p^1$ be a $\Q$-factorial terminal Mori fibre space such that a general fibre $F$ is a smooth threefold of Picard rank $\ge 2$, and such that $\Autz(X)$ is not trivial. Then, one of the following holds:
 
  \begin{enumerate}
 \item\label{Dim41}
There is a Mori fibre space $\pi_B\colon Y\to B$ with general fibres isomorphic to either $\p^1$, or $\p^3$, or a smooth quadric $Q\subset \p^4$ and 
an $\Autz(X)$-equivariant birational map $\varphi\colon X\dasharrow Y$. 
 \item\label{Dim42}
 A general fibre $F$ of $\pi$ is isomorphic to one of the following two smooth Fano threefolds of Picard rank $\ge 2$ with $\Autz(F)\simeq \PGL_2(\C)$:
 \begin{enumerate}
 \item\label{Dim43}
 The blow-up of the quadric $Q\subset \p^4$ given by $x_0x_4 - 4x_1x_3 + 3x_2^2=0$ along the image of the Veronese embedding of degree $4$ of $\p^1$.
 \item\label{Dim44}
 The threefold 
\[\left\{(x,y,z)\in (\p^2)^3 \left| 
 \sum_{i=0}^2 x_i y_i= \sum_{i=0}^2 x_i z_i=\sum_{i=0}^2 y_i z_i=0\right\}\right..\]
 \end{enumerate}
 \end{enumerate}
 \end{theoremA}
 
 Theorem~\ref{Thm:Dimension4} is proved as follows. In Section~\ref{Sec:SymmSmoothFano} (see Table~\ref{3folds.MFS}) we recall  \cite[Table~1]{CFST16}, 
 which lists all the smooth threefolds $F$ with Picard rank at least $ 2$ that are fibres of klt Mori fibre spaces. 
 These varieties are ``symmetric'', in the sense that there is a finite group $G\subseteq \Aut(F)$ such that $\Pic(F)^{G}$ has rank $1$ \cite[Theorem~1.2]{ProGFano2}. 
 We then show (Proposition~\ref{prop:AutFiniteTable}) that among the $9$ families listed, only the following four have infinite automorphism group: 
 $(\p^1)^3$, a smooth divisor of bidegree $(1,1)$ in $\p^2\times \p^2$ (isomorphic to $\P(T_{\p^2})$, see Lemma~\ref{T4Aut}), the blow-ups of a smooth quadric along a 
 Veronese curve of degree $4$ and the blow-up of $\P(T_{\p^2})\subset \p^2\times \p^2$ along a curve of bidegree $(2,2)$ whose projection on both factors is an embedding. 
 We moreover determine in which cases the automorphism group is not a torus.
 
In Section~\ref{Sec:Symmbirmap}, we describe some symmetric birational maps from $(\p^1)^3$ or $\P(T_{\p^2})\subset \p^2\times \p^2$ (which blow-up curves balanced with respect to the gradings) and use then these maps in Section~\ref{Sec:MfsP13TP2}, together with Theorem~\ref{Thm:ExistenceSection}, to get some $\Autz(X)$-equivariant birational maps from Mori fibre spaces $X\to \p^1$ having general fibres being isomorphic to $(\p^1)^3$ or $\P(T_{\p^2})$.

The proof of Theorem~\ref{Thm:Dimension4} is then given at the end of Section~\ref{Sec:MfsP13TP2}.

 The authors thank Andrea Fanelli, Ronan Terpereau, Andreas H\"oring, Vladimir Lazi\'c and Christopher Hacon for helpful discussions during the preparation of this text.

\section{Preliminaries}
\subsection{Mori fibre spaces and algebraic groups acting on them}
We work over the complex numbers. We refer to \cite{KM98} for the basic notions in birational geometry and minimal model program. We recall that a \emph{fibration} is a surjective morphism with connected fibres.
\begin{definition}
Let $f \colon X\rightarrow Y $ be a dominant projective morphism of normal varieties. Then $f$ is called a \emph{Mori fibre space} if the following conditions are satisfied:
\begin{itemize}
\item $f$ has connected fibres, with $\dim Y< \dim X$ ;
\item $X$ is terminal $\Q$-factorial with at most terminal singularities; 
\item the relative Picard number of $f$ is one and $-K_X$ is $f$-ample (i.e.~there is an element $[C]\in NS(X)$ with $-K_X\cdot C>0$, and such that each curve contracted by $f$ is numerically equivalent to an element of $\mathbb{R}_{>0}\cdot [C]$).
\end{itemize}
\end{definition}

We recall first the statement of the \textit{Blanchard's lemma}, which will be of fundamental importance for us. This result is due to Blanchard \cite{Bla56} in the setting of complex geometry, and the proof has been adapted to the setting of algebraic geometry.

\begin{lemma} \emph{\cite[Proposition~4.2.1]{BSU13}} \label{blanchard}
Let $f\colon X \to Y$ be a proper morphism between varieties such that $f_*(\O_X)=\O_Y$. If a connected algebraic group $G$ acts regularly on $X$, then there exists a unique regular action of $G$ on $Y$ such that $f$ is $G$-equivariant.
\end{lemma}

We also recall the following classical fact, which follows from \cite[Proposition~2, page~8]{BrionMonastir} or from \cite[Lemma~3.7]{Newstead}.
\begin{lemma}\label{lemm:DimMaxopen}
Let $G$ be an algebraic group acting regularly on a projective variety~$X$. Let $n=\max\{\dim(G\cdot x)\mid x\in X\}$ be the maximal dimension of an orbit of $G$. Then, the set $\{x\in X\mid \dim (G\cdot x)<n\}$ is a closed subset of $X$. In particular, the union of orbits of dimension $n$ is a dense open $G$-invariant subset of $X$.
\end{lemma}

The following theorem is a relative version of the so-called relative Base point free theorem 
\cite[Theorem~3-1-1]{KMM} and follows from \cite[Theorem~2.1]{Fujino11b}. As we will use it many times in this note, we recall the statement.
\begin{theorem}\label{fujino}Let $X$ be a variety with terminal singularities, let $\pi\colon X\to S$ be a proper surjective morphism of normal varieties,
and $D$ a $\pi$-nef Cartier divisor on~$X$. Assume that
  $rD-K_X$ is nef and big over $S$ for some positive integer $r$.

Then $D$ is relatively semiample. More precisely, there  exists a positive integer $m_0$ such  that  for  every $m\geq m_0$ the  natural  
homomorphism $\pi^*\pi_*\mathcal O_X(mD)\to\mathcal O_X(mD)$ is surjective.  
\end{theorem}
\begin{proof}
We apply \cite[Theorem~2.1]{Fujino11b} to the pair $(X,0)$. As the pair is terminal, the second hypothesis of \cite[Theorem~2.1]{Fujino11b} is verified.
\end{proof}
\subsection{Rational Maps between Mori fibrations}
The following lemma is known to experts (see \cite{HXu}), we recall the proof here for the reader's convenience. 
\begin{lemma}\label{lem:cpt}
Let $\pi_U'\colon Y_U\to U$ be a smooth projective fibration over a quasi-projective variety $U$ such that $\rho(Y_U/ U)=1$, $K_{Y_U}$ is $\pi_U'$-antiample
and there is a connected group $G$ acting on $Y_U$. 
Let $B\supseteq U$ be a $G$-equivariant compactification. 
Then there is a $G$-equivariant compactification $Y\supseteq Y_U$ and a morphism $\pi'\colon Y\to B$ such that $\pi'\vert_{Y_U}=\pi'_U$, and $Y\to B$ is a Mori fibre space.
\end{lemma}
\begin{proof}
Let $\overline Y$ be a $G$-equivariant compactification of $Y_U$ such that there is a morphism $\overline Y\to B$  and let $\eta\colon\widehat Y\to \overline Y$ be a $G$-equivariant resolution of singularities which is a composition of blow ups whose centers are contained in the singular locus of $\overline Y$.
In particular $\eta$ is an isomorphism over $Y_U$.
We run a $K_{\widehat Y}$-MMP over $B$ with scaling of a relatively ample divisor.
This MMP terminates by \cite[Theorem~2.3]{Fujino11a} (applied with $\Delta=0$) with a Mori fibre space over $B$ that we denote by  $\pi'\colon Y\to T$. Since $\rho(Y_U/ U)=1$  the MMP induces an isomorphism on $Y_U$ and $B=T$.
\end{proof}

\begin{lemma}\label{relBPF}
Let $X$ be a terminal variety,  $\pi\colon X\to B$ be a fibration  and let $\varepsilon \colon X'\to X$ be a $\Autz(X)$-equivariant birational  morphism.

Assume that there are a non-empty open subset $U\subseteq B$, an integer $a$ and a divisor $L$ on $X'$ such that  for every $b\in U$
\begin{enumerate}
\item\label{nfbig} the divisor $((a-1)K_{X'}+L)\vert_{X'_b}$ is nef and big
\item\label{nfnotbig} the divisor $(aK_{X'}+L)\vert_{X'_b}$ is nef, not numerically zero and not big.
\end{enumerate} 
where $X'_b\subset X'$ is the fibre over $b$.
Then there is a commutative diagram
\[
\xymatrix{
X'\ar[d]_{\varepsilon}\ar@{-->}[r]^{\psi}&W\ar[d]^{f}\\
X\ar[d]_{\pi}&S\ar[ld]\\
B&
}
\]
with $\Autz(X)$-equivariant arrows, such that $\psi$ is birational and $\dim S >\dim B$. Moreover, for each $m>0$ big enough, the map $X'_b\to S_b$ is given by $\lvert m(aK_{X'}+L)\rvert$ for each $b\in U$ $($where $S_b\subset S$ denotes the fibre over $b)$.
\end{lemma}
\begin{proof}
As the group $\Autz(X)$ acts on the fibres satisfying conditions \ref{nfbig}-\ref{nfnotbig}, we can assume that $U$ is $\Autz(X)$-invariant.
Let $X_U=\pi^{-1}U$ and $X'_U=(\pi\circ\varepsilon)^{-1}U$.
The divisor $((a-1)K_{X'}+L)\vert_{X'_U}$ is relatively nef and big by  \ref{nfbig} and $(aK_{X'}+L)\vert_{X'_U}$ is nef by \ref{nfnotbig}.
Therefore
by Theorem~\ref{fujino} the divisor $(aK_{X'}+L)\vert_{X'_U}$ is relatively semiample.
Therefore there is a diagram
\[
\xymatrix{
X'_U\ar[d]\ar[rd]^{h}&\\
X_U\ar[d]&S_U\ar[ld]\\
U&
}
\]
where $h$ is the morphism induced by $\lvert m(aK_{X'}+L)\vert_{X'_U}\rvert$ for some $m$ big enough.
As the divisor is not relatively big nor numerically zero by \ref{nfnotbig},
$\dim U<\dim S_U<\dim X'_U$.

Then we consider an $\Autz(X)$-equivariant compactification $S$ of $S_U$ such that there is a morphism $S\to U$
and an $\Autz(X)$-equivariant compactification $W$ of $X'_U$ such that there are morphisms $W\to X$ and $W\to S$ and we proved our claim.
\end{proof}

The following lemma is an immediate consequence of \cite[Theorem~2.3]{Fujino11a} but we add the proof for the sake of completeness.
\begin{lemma}\label{relMfs}Let $W$, $S$ be quasi-projective varieties such that $W$ is terminal.
Let $f\colon W\to S$ be a fibration  whose general fibre is uniruled. Then there is a factorisation
\[
\xymatrix{
W\ar[ddr]_{f}\ar@{-->}[rr]^{\phi}&&Y\ar[d]^{\pi'}\\
&&T\ar[dl]\\
&S&
}
\]
where $\phi$ is a birational map and $\pi'\colon Y\to T$ is a Mori fibre space.
Moreover, all the maps appearing in the diagram are $\Autz(X)$-equivariant.
\end{lemma}
\begin{proof}
We run a $K_W$-MMP over $S$ with scaling of a relatively ample divisor.
This MMP terminates by \cite[Theorem~2.3]{Fujino11a} (applied with $\Delta=0$) with a Mori fibre space over $S$ that we denote by  $\pi'\colon Y\to T$.
\end{proof}

\begin{lemma}\label{lem:monodr}
Let $\pi\colon X\to \p^1$ be a Mori fibre space such that a general fibre satisfies $\rho(F)\geq 2$.
 Then $\pi$ has at least two singular fibres. In particular, the action of $\Autz(X)$ on $\p^1$ given by the Blanchard's lemma fixes at least two points.
\end{lemma}
\begin{proof}
Let $U\subseteq\p^1$ be the set $\{b\in \p^1\vert\; \pi^{-1}b \;\;{\rm is}\;\;{\rm terminal}\}$. 
 By \cite[Theorems 2.2 and 2.5]{CFST16} and \cite{KolMor} for every $b\in U$ the restriction $r\colon N^1(X)_{\Q}\to N^1(F)_{\Q}^{\pi_1(U)}$ is a surjective map.
 We notice that since $\rho(X)=2$ we have $N^1(X)_{\Q}\cong\Q^2$.
 Moreover, the class of $F$ is in the kernel of $r$, therefore there is a surjective map $N^1(X)_{\Q}/\Q[F]\cong \Q\to N^1(F)_{\Q}^{\pi_1(U)}$.
 If $\rho(F)\geq 2$, then the fundamental group of $U$ must be non-trivial.
 Therefore the locus of non-terminal fibres contains at least two points.
 The last sentence follows from the fact that $\Autz(X)$ is a connected group preserving the non-terminal locus.
\end{proof}
\subsection{Finite morphisms and Fano manifolds}
\begin{lemma}\label{qfinite2finite}
 Let $f\colon W\to Z$ be a quasi-finite surjective morphism such that all its fibres have the same cardinality, $Z$ is normal projective and $W$ is quasi-projective.
 Then $W$ is projective and $f$ is finite.
\end{lemma}
\begin{proof}
 By the existence of compactifications and resolution of indeterminacies we can factor $f$ as $\bar f\circ\iota$ where $\iota\colon W\to\overline W$
 is an open immersion,  $\overline W$ is projective and $\bar f\colon\overline W\to Z$ is projective and generically finite.
 We consider then the Stein factorisation $\bar f=\eta\circ g$ of  $\bar f$, where $\eta\colon\overline W\to Z'$ has connected fibres and $g\colon Z'\to Z$ is finite. As $\bar f$ is quasi-finite, $\eta$ is birational.
Then the morphism $\eta\circ\iota\colon W\to Z'$ is birational and all its fibres are finite; it is thus an open embedding by the Zariski main theorem \cite[Lemma~37.38.1]{stacks-project}. 
Therefore, we can take $Z'=\overline W$ and then view $W$ as an open subset of $Z'$, and $f$ as the restriction of the finite morphism $g\colon Z'\to Z$.

 Since all the fibres of $f$ have the same cardinality, by semicontinuity of the cardinality of the fibres, all the fibres of $g$ have the same cardinality.
 Since $f$ is surjective, we get $W=\overline W$ which finishes the proof.
\end{proof}

We recall the following lemma for the readers' convenience 
\begin{lemma}\label{etaleFano}
Let $Z$ be a Fano manifold. Then any finite \'etale map $f\colon W\to Z$ is an isomorphism. 
\end{lemma}
\begin{proof}

Since $f$ is \'etale, the variety $W$ is a Fano manifold.
Then by the Kawamata-Viehweg vanishing theorem $\chi(W)=1=\chi(Z)$.
On the other hand, if $f$ is finite \'etale, then  $\chi(W)=\deg f\chi(Z)$. Therefore $\deg f=1$ and $f$ is an isomorphism.
\end{proof}

\section{Existence of an invariant horizontal closed subspace}\label{Sec:ProofExistenceS}

 \subsection{Defining some sets of sections and a bend and break result}

\begin{lemma}\label{lemma1}
Let $X$ be a $\mathbb{Q}$-factorial variety and let $\pi\colon X\to B$ be a Mori fibration over a smooth irreducible curve $B$.  
 Then, the set \[\mathcal K=\{- K_X\cdot s\vert\; s\text{ is a section of }\pi\}\] is a non-empty discrete subset of $\mathbb{Q}$ which is bounded from below. In particular, it admits a minimum.
\end{lemma}
\begin{proof}
The fact that $K$ is non-empty follows from~\cite[Theorem~1.1]{GHS}.
Let $H$ be an ample divisor on $X$ and let $F$ be a fibre. Since $\rho(X/B)=2$, the divisor $-K_X$ is numerically equivalent to $\alpha H+\beta F$ for some $\alpha,\beta\in \mathbb{R}$. As $\pi$ is a Mori fibration, the restriction of $-K_X$ to a general fibre is ample, so $\alpha> 0$.
For each section $s$, we have $H\cdot s> 0$, so   $-K_X\cdot s =\alpha H\cdot s+\beta \ge \beta$. This shows that $\mathcal K$ is bounded from below.  Since $X$ is $\Q$-factorial, there exists $r\in\N$ such that $rK_X$ is Cartier.
 Therefore $\mathcal K\subseteq \frac{1}{r}\Z$ is discrete.
\end{proof}
\begin{definition}
Let $X$ be a $\mathbb{Q}$-factorial varieties and let $\pi\colon X\to B$ be a Mori fibration over a smooth irreducible curve $B$.  We say that a section $s\subset X$ is a \emph{minimal section} if $-K_X\cdot s\le -K_X\cdot s'$ for each section $s'\subset X$.
We say that a section $s\subset X$ is an \emph{aut-minimal section} if $-K_X\cdot s\le -K_X\cdot s'$ for each section $s'\subset \overline{\Autz(X)\cdot s}$.
 Lemma~\ref{lemma1} shows that minimal and aut-minimal sections always exist.
\end{definition}
\begin{remark}
If $\pi\colon X\to B$ is a $\p^1$-bundle over a smooth irreducible curve $B$, for each section $s\subset X$, the adjunction formula gives $s^2=-s\cdot K_X-2+2g(B)$, so Lemma~\ref{lemma1} generalises the classical fact that $s^2$ is bounded from below, and minimal sections correspond here to the sections of minimal self-intersection.
\end{remark}

We present now a Bend and Break result. The proof follows \cite[Proposition~3.2]{Deb01}.

\begin{proposition}\label{bandb}
 Let $X$ be a projective variety together with a fibration $\pi\colon X\to\p^1$, let $s\subset X$ be a section of $\pi$ and $x\in s$.
Suppose that there is an irreducible curve $\Gamma\subseteq \Autz(X)$ such that $g(x)=x$ and $g(s)\not=s$ for a general $g\in \Gamma$.
Then the $1$-cycle $s$ is numerically equivalent to a non-integral effective rational $1$-cycle passing through $x$ and contained in $\overline{\Autz(X)\cdot s}$. In particular, $s$ is not an aut-minimal section.
\end{proposition}

\begin{proof}
 Let $\nu\colon C\to\Gamma$ be the normalisation of $\Gamma$. By the Blanchard's Lemma~\ref{blanchard} there is a morphism $\nu'\colon C\to \Aut(\p^1)=\PGL_2(\C)$ such that $\pi\circ \nu(g)=\nu'(g)\circ \pi$ for each $g\in C$.

Let $\varphi\colon \p^1\to s$ be the morphism such that $\pi\circ \varphi=\mathrm{id}_{\p^1}$. We now prove that the morphism
 \[ \begin{array}{rccc}
 F\colon & \p^1\times C&\to & X\times C\\
  & (p,g) &\mapsto &(\nu(g)(\varphi(p)),g)
 \end{array}\]
 is finite. This is implied by the fact that for each $g\in \Gamma$, the morphism $\psi_g\colon \p^1\to X, t\mapsto \nu(g)(\varphi(p))$ is  injective. This last claim follows from the fact that $\pi\circ \psi_g=\nu'(g)\in \Aut(\p^1)$, as $\pi\circ \psi_g(p)=(\pi\circ \nu(g))(\varphi(p))=(\nu'(g)\circ \pi)(\varphi(p))=\nu'(g)(p)$ for each $p\in \p^1$.
 
 As $F$ is finite, $F(\p^1\times C)$ has dimension $2$. We then follow the proof of \cite[Proposition~3.2]{Deb01}.
 Let $\overline C$ be a smooth compactification of $C$. Let $S$ be the normalisation in $\C(\p^1\times C)$ of the closure in $X\times\overline C$ of the image of $F$, with finite canonical morphism $\overline F\colon S\to X\times\overline C$.
 Since $\p^1\times C$ is normal, by uniqueness of the normalisation we have $\overline F^{-1}(X\times C)=\p^1\times C$. We obtain the commutative diagram
 \[
 \xymatrix{
  \p^1\times C\ar[dd]\ar[r]&S\ar@/_2pc/[dd]_{\kappa}\ar[d]^{\overline F}\ar[r]^{e}&X\\
 &X\times\overline C\ar[ur]_{p_1}\ar[d]^{p_2}&\\
 C\ar[r]&\overline C&}
 \]
 No component of a fibre of $\kappa$ is contracted by $e$ because it would otherwise be contracted by $\overline F$
 and $\overline F\colon S\to \overline F(S)$ is finite as it is a normalisation.
The morphism $\kappa\colon S\to\overline C$ is flat as $\overline C$ is a smooth curve \cite[III, prop.9.7]{Har77}.
Therefore each fibre is 1-dimensional, with no embedded component, of genus 0 \cite[III, cor.9.10]{Har77}. An integral fibre is a rational curve and a singular one is a tree of rational curves.

Denoting by $p_0=\pi(x)\in \p^1$ the image of $x$, we have $\varphi(p_0)=x$. As every element of $\Gamma$ fixes $x$, the image by $F$ of $\{p_0\}\times C$ is $\{x\}\times C$. In particular, the closure $C_0$ of $\{p_0\}\times C$  is contracted to $x$ by $e$. 

We now prove that there is a non-integral fibre of $\kappa$. This part is different from  \cite[Proposition~3.2]{Deb01}.
 Since $C_0$ is contracted by $e$ and $e(S)$ is a surface, we have $C_0^2<0$.
 We then consider the morphism $\tau\colon S\to \overline{C}\times \p^1$ that is given by  the two morphisms $\kappa\colon S\to \overline C$ and $\pi\circ e\colon S\to\p^1$. Observe that $\tau$ is birational: for an element $c\in C$, the preimage in $S$ of the curve $\p^1\times \{c\}\subseteq \p^1\times C$  is sent to a section of $\pi$ by $F$. As $\tau$ is a birational morphism and as $C_0^2<0$, the morphism $\tau$ contracts some irreducible components contained in the fibres of $\overline C\setminus C$, which are therefore not all integral.

Let $z_0\in \overline C$ be a point such that $\kappa^*z_0$ is a non integral fibre. The curve $s$ is equal to  $e_*\kappa^*z_1$, where some $z_1\in C$ is sent to the identity in $\Gamma\subset \Autz(X)_{\p^1}$.
 Then $e_*\pi^*z_0$ is an effective $1$-cycle  numerically equivalent to $s$ and passes through $x$ as $C_0$ is a section of $\kappa$ that is contracted to $x$. Moreover, $e_*\pi^*z_0$ is not integral as $\pi^*z_0$ is not integral and $e$ does not contract any irreducible component of a fibre.
 
It remains to see that $s$ is not a minimal section. We write $e_*\pi^*z_0=\sum_{i=1}^r \ell_i$ where the $\ell_1,\ldots,\ell_r$ are irreducible and reduced curves on $X$. As $e_*\pi^*z_0$  is numerically equivalent to a section, exactly one of the $\ell_i$, say $\ell_1$, is a section, and $\ell_2,\ldots,\ell_r$ are contained in fibres. As $-K_X$ is ample on the fibres, we have $-K_X\cdot \ell_i>0$ for $i\ge 2$. This gives
\[-K_X\cdot s=-K_X\cdot \ell_1+\sum_{i=1}^r -K_X\cdot \ell_i>-K_X\cdot \ell_1\]
and implies that $s$ is not a minimal section.
\end{proof}

\subsection{Transitivity on the fibres}
\begin{lemma}\label{lemmaactiontrans}
 Let $\pi\colon X\to\p^1$ be a Mori fibre space such that the action of \[\Autz(X)_{\p^1}=\{g\in \Autz(X)\mid \pi g=\pi\}\] on a general fibre is transitive. 
 
 Then, there is a dense open subset $U\subseteq \p^1$, a smooth Fano variety $F$ of Picard rank $1$ and an isomorphism $\theta\colon U\times F\iso \pi^{-1}(U)$ such that $\pi\circ \theta$ is the first projection $U\times F\to U$.
\end{lemma}
\begin{proof}
Let $s\subset X$ be a minimal section (which exists by Lemma~\ref{lemma1}) and  let $\varphi\colon \p^1\to s$ be the morphism such that $\pi\circ \varphi=\mathrm{id}_{\p^1}$.  
Let $G=\Autz(X)_{\p^1}$ and $G^{\circ}$ the connected component of the identity. As the finite group $G/G^{\circ}$ acts on the set of orbits of $G^{\circ}$, the action of $G^{\circ}$ on a general fibre is transitive.

Let $H=\{g\in G^{\circ}\mid g(s)=s\}\subseteq G^{\circ}$. The quotient $V=G^{\circ}/H$ is homogeneous for the action of $G^{\circ}$ and is thus smooth. Let $x\in X$ be such that $\{x\}=s\cap F$. We obtain a surjective $G^{\circ}$-equivariant morphism 
\[\begin{array}{rrcl}
\Phi\colon& V&\to& F\\
&  [g]&\mapsto& g(x),
 \end{array}\]
 where $[g]\in V=G^{\circ}/H$ denotes the class of $g\in G^{\circ}$.
 
 We prove that $\dim V=\dim F$. Indeed, otherwise $\dim V>\dim F$ and there is an irreducible curve $\Gamma\subseteq G^{\circ}$ such that $g(x)=x$ and 
$g\circ\varphi\not=\varphi$ for a general $g\in \Gamma$. The fact that  $g\in G^{\circ}$ and $g\circ\varphi\not=\varphi$ implies that $g(s)\neq s$, impossible by Lemma~\ref{bandb}.

We now consider the Stein factorisation of $\Phi$ is given by $V\overset{\tilde \Phi}{\rightarrow}\widetilde F\overset{\nu}{\rightarrow}F$.
Since $\dim V=\dim F$, the morphism $\tilde{\Phi}$ is birational; it is moreover $G^{\circ}$-equivariant by Blanchard's lemma (Lemma~\ref{blanchard}). Hence,  $\nu$ is also $G^{\circ}$-equivariant.

As $\nu$ is finite and $G^{\circ}$-equivariant and as $G^{\circ}$ acts transitively on $F$, it is \'etale (by the generic smoothness). Lemma~\ref{etaleFano} implies that $\nu$ is an isomorphism.
We just proved that for a general fibre $F$, for any $y\in F$ there is a unique $g\in G^{\circ}$ such that $y\in g(s)$ (or equivalently such that $y=g(\varphi(\pi(y)))$).
Therefore there is an open set $U\subseteq\p^1$ such that the morphism 
 $$
 \begin{array}{rrcl}
 \theta\colon & V\times U&\to&\pi^{-1}(U)\\
&  ([g],t)&\mapsto&g(\varphi(t)).
 \end{array}
 $$
 is bijective. Restricting $U$ we may assume that $\pi^{-1}(U)$ is smooth, and thus by the Main theorem of Zariski, the above morphism is an  isomorphism.

This proves that $\pi$ is a trivial $V$-bundle over $U$, so in particular $V$ is isomorphic to a general fibre $F$. It remains to see that the Picard rank of $V$ is equal to $1$. Suppose for a contradiction that there exist two prime divisors $D_1$ and $D_2$ on $F$ having classes in $\mathrm{NS}(F)$ which are $\Q$-independent (which exist as soon as $\rho(F)\ge 2$)  
 and $\mathcal D_i$ the Zariski closure of $\theta(D_i\times U)$ in $X$.
 Then the classes  $\mathcal D_1$ and $\mathcal D_2$ are not numerically equivalent over the base, even after some multiple, so $\rho(X/\p^1)\ge 2$, 
 contradicting that $\pi$ is a Mori fibre space. 
\end{proof}

\subsection{Existence of invariant subsets and the proof of Theorem~\ref{Thm:ExistenceSection}}
 
\begin{lemma}\label{Lem:orbithoriz2}
Let $\pi\colon X\to\p^1$ be a Mori fibre space and let \[G=\Autz(X)_{\p^1}=\{g\in \Autz(X)\mid \pi g=\pi\}\] and let $G^{\circ}$ the connected component of the identity. Then, the following hold
 \begin{enumerate}
 \item \label{dimensionOrbits}
 For each aut-minimal section $s\subset X$ and for any two points $p,q\in s$, we have \[\dim G\cdot p=\dim G^{\circ} \cdot p=\dim G^{\circ} \cdot q= \dim G\cdot q.\]
 \item\label{orbitnotclosedBendBreak}
For each $p\in s$ for each point $q\in \overline{G^{\circ}\cdot p}\setminus G^{\circ}\cdot p$ there is an aut-minimal section $s'\subset X$ that contains $q$.
 \item\label{ClosedImageOrbits}
   For each aut-minimal section $s\subset X$ such that $G\cdot p$ is closed for some $p\in s$,  the stabilisers $(G^\circ)_q$ and $(G^\circ)_s$ are equal for all $q\in s$ and the quotient $G^{\circ}/(G^\circ)_s$ is a Fano variety. Moreover,   denoting by $\varphi\colon \p^1\to s$  the morphism such that $\pi\circ \varphi=\mathrm{id}_{\p^1}$, the morphism
\[\begin{array}{rccc}
\kappa\colon &G^{\circ}/(G^{\circ})_s\times \p^1 &\to &X\\
& ([g],b) & \mapsto & g(\varphi(b)).\end{array}\]
is a closed embedding.
  \end{enumerate}
\end{lemma}
\begin{proof}Lemma~\ref{lemma1} implies that there exists a minimal section of $X$, which is therefore also aut-minimal.

\ref{dimensionOrbits}: As the finite group $G/G^{\circ}$ acts on the set of orbits of $G^{\circ}$, we have $\dim G\cdot p=\dim G^{\circ} \cdot p$ for each $p\in s$. It then suffices to show that $\dim G^{\circ}\cdot p= \dim G^{\circ}\cdot q$ for any two points $p,q\in s$. Up to exchanging $p$ and $q$ we may assume that $\dim G^{\circ}\cdot p< \dim G^{\circ}\cdot q$,  in order to derive a contradiction. Then the stabilisers satisfy $\dim (G^{\circ})_{p}>\dim (G^{\circ})_{q}$. As the stabiliser $(G^{\circ})_s\subseteq G^{\circ}$ of $s$ is contained in both $(G^{\circ})_{p}$ and $(G^{\circ})_{q}$, we obtain $\dim (G^{\circ})_{p}/(G^{\circ})_s>1$. There is thus an irreducible curve $\Gamma\subseteq G^{\circ}\subseteq G=\Autz(X)_{\p^1}$, such that $g(p)=p$ and $g(s)\not=s$ for a general $g\in \Gamma$.
This contradicts the minimality of $s$ by Lemma~\ref{bandb}. 

\ref{orbitnotclosedBendBreak}: Assume that $G^{\circ}\cdot p$ is not closed and let $q\in \overline{G^{\circ}\cdot p}\setminus G^{\circ}\cdot p$.
Set $a=\pi(q)=\pi(p)\in \p^1$ and consider the morphism \[\begin{array}{rccc}
\overline\kappa\colon &G^{\circ}\times \p^1 &\to &X\\
& (g,b) & \mapsto & g(\varphi(b)),\end{array}\] where $\varphi\colon \p^1\to s$  is the morphism such that $\pi\circ \varphi=\mathrm{id}_{\p^1}$. There exists an irreducible curve $C\subset G^{\circ}\times \p^1$ such that $q$ belongs to the closure of $\overline{\kappa}(C)$ and such that $(\mathrm{id},a)\in C$. The closure $\Gamma\subseteq G^{\circ}$  of the projection of $C$ in $G^{\circ}$ is an irreducible curve  such that $\mathrm{id}\in \Gamma$ and such that $q$ belongs to the closure of $\overline{\kappa}(\Gamma\times \p^1)$.
 
Let $\overline\Gamma$ be a compactification of a normalisation of  $\Gamma$.
The morphism $\overline{\kappa}|_{\Gamma\times \p^1}\colon \Gamma\times \p^1\to X$ yields a rational map $\theta\colon \overline\Gamma\times \p^1\dasharrow X$.
Let $S$ be the Zariski closure of the image of $\theta$ (or equivalently of $\overline{\kappa}(\Gamma\times \p^1)$). Then $S$ has dimension 2 
and $q\in S$. We take a  resolution of the indeterminacies of  $\theta$ 
\[
 \xymatrix{
&\hat S\ar[ld]_\nu\ar[rd]^\mu&\\
 \overline\Gamma\times \p^1\ar@{-->}[rr]^{\theta}&&S
}\]
and find $ z\in \overline\Gamma\setminus \Gamma$ such that $q\in \mu_*\nu^*(z\times \p^1)$. As $\mathrm{id}\in \Gamma$, we obtain $\mu_*\nu^*({ \mathrm{id}}\times \p^1)=s$. Hence, $\mu_*\nu^*(z\times \p^1)$ is a cycle numerically equivalent to $s$, that we can write as 
\[\mu_*\nu^*(z\times \p^1)\equiv s'+l\]
where $s'$ is a section and $l$ is an effective $1$-cycle contained in fibres of $\pi$. As $-K_X\cdot l\ge 0$ and $s'\subseteq \overline{\Autz(X)\cdot s}$, the minimality of $s$ implies that $l=0$ and that $s'\equiv s$.
The section $s'$ is aut-minimal as $\overline{\Autz(X)\cdot s'}\subseteq\overline{\Autz(X)\cdot s}$. Moreover $q$ belongs to $\mu_*\nu^*(z\times \p^1)=s'$, so $q\in s'$. This achieves the proof of \ref{orbitnotclosedBendBreak}.
 
\ref{ClosedImageOrbits}:
By \ref{dimensionOrbits} for every $p\in s$ we have $\dim (G^{\circ})_p=\dim (G^{\circ})_s$.
Hence, the morphism 
$\tau\colon G^{\circ}/(G^{\circ})_s\to G^{\circ}/(G^{\circ})_p$ is a quasi finite morphism.
The morphism $\tau$ being $G^{\circ}$-equivariant, all of its fibres have the same cardinality.
As the orbit $G^{\circ} \cdot p$ is closed in $X$, it is projective, hence the variety $G^{\circ}/(G^{\circ})_p$ is projective, 
and thus Fano as it is homogeneous for the action of the linear connected group $G^{\circ}$ by \cite[Corollary 2.1.7]{AGV}. 
By Lemma~\ref{qfinite2finite} the variety $G^{\circ}/(G^{\circ})_s$ is projective. Therefore $\tau$ is \'etale and by Lemma~\ref{etaleFano}
it is an isomorphism. This gives $(G^{\circ})_p= (G^{\circ})_s$ for every $p\in s$.

The morphism \[\begin{array}{rccc}
\kappa\colon &G^{\circ}/(G^{\circ})_s\times \p^1 &\to &X\\
& ([g],b) & \mapsto & g(\varphi(b)).\end{array}\] is closed as $G^{\circ}/(G^{\circ})_s$ is projective. 
It is moreover an isomorphism onto its image as for a fixed $b\in\p^1$ it induces the isomorphism $G^{\circ}/(G^{\circ})_{\varphi(b)}\to G^{\circ}\cdot\varphi(b)$.
 \end{proof}

\begin{lemma}\label{Lemm:AutzAndAutzP1}
Let $\pi\colon X\to\p^1$ be a Mori fibre space such that a general fibre $F$ has Picard $\rho(F)\ge 2$, and let $s\subset X$ be an aut-minimal section such that $\Autz(X)_{\p^1}\cdot p$ is closed for some $p\in s$. 

If $\Autz(X)_{\p^1}\cdot s\not= \Autz(X)\cdot s$, then there is an aut-minimal section $s'\subset X$  such that 
\[\dim {\Autz(X)\cdot s'}<\dim {\Autz(X)\cdot s}.\]
\end{lemma}
\begin{proof}
Let us write $G=\Autz(X)_{\p^1}$ and denote by $G^{\circ}$ the connected component of the identity.
By assumption, $G\cdot p$ is closed for some $p\in s$,  so we can apply Lemma~\ref{Lem:orbithoriz2}\ref{ClosedImageOrbits}. This implies that the   the stabilisers $(G^\circ)_q$ and $(G^\circ)_s$ are equal for each $q\in s$ and that 
\[\begin{array}{rccc}
\kappa\colon &G^{\circ}/(G^{\circ})_s\times \p^1 &\to &X\\
& ([g],b) & \mapsto & g(\varphi(b)).\end{array}\]
is a closed embedding, where $\varphi\colon \p^1\to s$ is the morphism such that $\pi\circ \varphi=\mathrm{id}_{\p^1}$. In particular, $G^{\circ}\cdot q=G^{\circ}/(G^{\circ})_s\cdot q$ is closed in $X$ for each $q\in s$.

Since $G\cdot s\not= \Autz(X)\cdot s$, we have $G\subsetneq \Autz(X)$. As the action of $\Autz(X)$ on $\p^1$ fixes at least two points (Lemma~\ref{lem:monodr}), we may change coordinates such that the action of $\Autz(X)$ on $\p^1$ is exactly the group, isomorphic to $\C^*$, fixing $[1:0]$ and $[0:1]$. We denote by $\alpha\colon \Autz(X)\to \C^*$ the corresponding surjective group homomorphism.

We consider the $\Autz(X)$-equivariant morphism 
\[\begin{array}{rccc}
\Phi\colon &\Autz(X)/(G^{\circ})_s\times \p^1 &\to &X\\
& ([g],b) & \mapsto & g(\varphi(b)),\end{array}\]
whose restriction to $G^{\circ}/(G^{\circ})_s$ gives $\kappa$. The image \[Z=\Phi(\Autz(X)/(G^{\circ})_s\times \p^1)=\Autz(X)\cdot s\] strictly contains $G\cdot s$, which is closed in $X$. Since $Z$ is irreducible (because $\Autz(X)/(G^{\circ})_s\times \p^1$ is irreducible), we find $\dim Z>\dim G\cdot s=\dim (G^{\circ}/(G^{\circ})_s\times \p^1)$. As $\Autz(X)/G\simeq \C^*$ has dimension $1$, we obtain 
\[\dim(\Autz(X)/(G^{\circ})_s\times \p^1)=\dim Z=\dim (G\cdot s)+1.\]
In particular, $\Phi\colon \Autz(X)/(G^{\circ})_s\times \p^1\to Z$ is generically finite.

We now prove that every fibre of $\Phi$ is finite. Suppose by contradiction that some irreducible curve $\Gamma\subseteq \Autz(X)/(G^{\circ})_s\times \p^1$ 
is sent onto a point $q_0\in Z$. For each $g\in \Autz(X)$, the curve $g\cdot \Gamma$ is contracted onto $g(q_0)$, so 
the closure of $\{g\cdot \Gamma\mid g\in \Autz(X)\}$ is a subvariety $F\subseteq \Autz(X)/(G^{\circ})_s\times \p^1$ that is sent by 
$\Phi$ onto a subvariety $F'=\Phi(F)$ of $Z$ of dimension $\dim F'<\dim F$. Since $\Phi$ is generically finite, 
$\dim F <\dim ( \Autz(X)/(G^{\circ})_s\times \p^1)$. As $F$ is invariant by $\Autz(X)$, it has to be equal to $\Autz(X)/(G^{\circ})_s\times \{b_0\}$ for some $b_0\in \p^1$. 
This implies that $\Gamma=\Gamma'\times \{b_0\}$ for some curve $\Gamma'\subseteq \Autz(X)/(G^{\circ})_s$. Replacing $\Gamma$ with $g\cdot \Gamma$ for some $g\in \Autz(X)$, we may assume that $q_0=\varphi(b_0)$. We then obtain $\Gamma'\cdot q_0=q_0$. As $s$ is an aut-minimal section, Proposition~\ref{bandb} implies that $\Gamma'\cdot s=s$. As $(G^{\circ})_{q_0}=(G^{\circ})_s$, the group $G_{q_0}/(G^{\circ})_s$ is finite, so $\Gamma'$ is not contained in $G$, which means that $\alpha(\Gamma')=\C^*$. For each element $g\in \Autz(X)$, there exists then $h\in \Gamma'$ such that $\alpha(g)=\alpha(h)$, and thus $g=g_0h$ for some $g_0\in G$. This gives $g(s)=g_0h(s)=g_0(s)\subseteq G\cdot s$. This being true for each $g\in \Autz(X)$, it gives $Z=\Autz(X)\cdot s=G\cdot s$, giving the desired contradiction.

We denote by $d$ the degree of the quasi-finite morphism $\Phi\colon \Autz(X)/(G^{\circ})_s\times \p^1\to Z$, by $\overline{Z}$ the closure of $Z$ in $X$ and by $R\subsetneq \overline{Z}$ the subset
\[R=\{x\in \overline{Z}\mid \Phi^{-1}(x)\text{ contains less than $d$ points}\}.\]
As $Z$ is $\Autz(X)$-invariant and $\Phi$ is $\Autz(X)$-equivariant, the set $R$ is $\Autz(X)$-invariant.

We prove that $R$ is closed in $\overline Z$. By the Hironaka resolution of singularities, there is a projective variety $W$
together with a morphism $\overline\Phi\colon W\to \overline Z$ and an open immersion $\iota\colon\Autz(X)/(G^{\circ})_s\times \p^1\to W$
such that $\overline\Phi\circ\iota=\Phi$.
Let $\overline\Phi=\sigma\circ\mu$ where $\mu$ has connected fibres, and is therefore birational, and $\sigma$ is finite.
By the Zariski main theorem \cite[Lemma~37.38]{stacks-project}, $\mu\circ\iota$ is an open immersion, therefore we can assume that $\overline\Phi$ is finite. The set $R$ is closed as it is the union of two closed sets: \[R=\overline Z\setminus Z\cup\{x\in \overline{Z}\mid \overline\Phi^{-1}(x)\text{ contains less than $d$ points}\}.\]

We now prove that $\pi(R)=\p^1$. Suppose for contradiction that $\pi(R)\not=\p^1$, which implies that $\pi(R)\subseteq \{[0:1],[1:0]\}$.  Writing $Z_1=Z\cap \pi^{-1}([1:1])$, the preimage $W_1=\Phi^{-1}(Z_1)$ is projective (Lemma~\ref{qfinite2finite}). It then suffices to show that the morphism $W_1\to \C^*$ obtained by composing $\alpha$ with the first projection $\Autz(X)/(G^{\circ})_s\times \p^1\to \Autz(X)/(G^{\circ})_s$ is surjective to obtain the desired contradiction. For each $\mu\in \C^*$, we take $g\in \Autz(X)$ such that $\alpha(g)=\mu$ and  choose the point $p\in s$ such that $g(\varphi(p))\in \pi^{-1}([1:1])$.

If $R$ contains a section $s''$, we have $\overline{\Autz(X)\cdot s''}\subseteq R\subsetneq\overline{\Autz(X)\cdot s}=Z$ and choose a aut-minimal section  $s'$ contained in $\overline{\Autz(X)\cdot s''}$. Otherwise, we choose a point $x\in R\cap \pi^{-1}([1:0])$ and a section $s''\subset X$ through $x$ and choose $s''$ such that $-K_X\cdot s''$ is minimal, among all sections through $x$ (this is possible by Lemma~\ref{lemma1}). If $ \dim {\Autz(X)\cdot s''}<\dim {\Autz(X)\cdot s}$, we choose a aut-minimal section  $s'$ contained in $\overline{\Autz(X)\cdot s''}$. It remains to assume that $ \dim {\Autz(X)\cdot s''}\ge \dim {\Autz(X)\cdot s})$ and to derive a contradiction. As $\Autz(X)\cdot x\subseteq R\cap \pi^{-1}([1:0])$, we have $\dim\Autz(X)\cdot x<\dim R<\dim Z=\dim {\Autz(X)\cdot s}$. This, together with $ \dim {\Autz(X)\cdot s''}\ge\dim Z$ implies the existence of an irreducible curve $C\subseteq\Autz(X)$ such that $C\cdot x=x$ and $C\cdot s''\not=s''$, contradicting Proposition~\ref{bandb}.

\end{proof}
The following proposition directly implies Theorem~\ref{Thm:ExistenceSection}.

\begin{proposition}\label{prop:inv}
 Let $\pi\colon X\to\p^1$ be a Mori fibre space such that a general fibre $F$ satisfies $\rho(F)\geq 2$ and let  $s\subset X$ be an aut-minimal section such that $\Autz(X)\cdot s$ is of minimal dimension $($i.e.~$\dim(\Autz(X)\cdot s)\le \dim(\Autz(X)\cdot s')$ for each aut-minimal section $s')$. Then, the following holds:
 \begin{enumerate}
 \item
 $\mathcal{S}=\Autz(X)\cdot s=\Autz(X)_{\p^1}\cdot s=(\Autz(X)_{\p^1})^{\circ}\cdot s$ is a proper closed subset of $X$;
 \item
 For each $b\in \p^1$, the fibre $\pi^{-1}(b)\cap \mathcal{S}$ of $\mathcal{S}\to \p^1$ is equal to $(\Autz(X)_{\p^1})^{\circ}\cdot p$, where $p\in s$ is such that $\pi(p)=b$.
 \end{enumerate}
\end{proposition}
\begin{proof}
We write $G=\Autz(X)_{\p^1}$ and denote by $G^{\circ}$ the connected component of the identity. If a point $p\in s$ were such that $G^{\circ}\cdot p$ were not closed, then for each point $q\in \overline{G^{\circ}\cdot p}\setminus G^{\circ}\cdot p$ there would be an aut-minimal section $s'\subset X$ that contains $q$ (Lemma~\ref{Lem:orbithoriz2}\ref{orbitnotclosedBendBreak}), giving thus $\dim (G\cdot s')<\dim(G\cdot s)$, impossible.

The choice of $s$ implies thus that  $G^{\circ}\cdot p$ is closed for each $p\in s$, and thus that $G\cdot p$ is closed too. Lemma~\ref{Lem:orbithoriz2}\ref{ClosedImageOrbits} then implies that $G^{\circ}\cdot s=(\Autz(X)_{\p^1})^{\circ}\cdot s$ is a closed subset of $X$ and that for each point $b\in \p^1$, the fibre $\pi^{-1}(b)\cap \mathcal{S}$ of $\mathcal{S}\to \p^1$ is equal to $G^\circ\cdot p=(\Autz(X)_{\p^1})^{\circ}\cdot p$, where $p\in s$ is the point such that $\pi(p)=b$. The fact that $G^{\circ}\cdot s$ is a proper subset of $X$ then follows from the fact that the general fibre has $F$ satisfies $\rho(F)\geq 2$ and from Lemma~\ref{lemmaactiontrans}.

Lemma~\ref{Lemm:AutzAndAutzP1} then implies that $\Autz(X)_{\p^1}\cdot s= \Autz(X)\cdot s$. 
\end{proof}

\section{The group $\Autz(X)_{\p^1}$}\label{Sec:Verticalgroup}

\subsection{The case of tori}
We start this section by proving Propositions~\ref{Prop:Torus} and~\ref{Prop:Smallorbit}, stated in the introduction. We recall the statements for the sake of readability.

\proptorus*
\begin{proof}Let us write $G=\Autz(X)$ to simplify the notation. As the result is empty if $G$ is trivial, we may assume that $\dim G=r\ge 1$. 
The group $G$ acts on $\p^1$ by Blanchard's Lemma~\ref{blanchard}.  This gives rise to an exact sequence
\[1 \to \Autz(X)_{\p^1}\to G\to H\to 1\]
where $H\subseteq \Aut(\p^1)$. As $\rho(F)\ge 2$ for a general fibre $F$, the number of singular fibres is at least $2$ so $H$ is a torus of dimension $0$ or $1$ 
(Lemma~\ref{lem:monodr}). Hence, we find $r\in \{\dim (\Autz(X)_{\p^1})$, $\dim (\Autz(X)_{\p^1})+1\}$. If $ \Autz(X)_{\p^1}$ is finite, 
then $\dim H =\dim G =r$, and \cite[IV.11.14, Corollary~1]{Bor91} provides the existence of a torus $T\subseteq G$ of dimension $r$. Hence, $G$ is a torus of dimension $r$.
 If $ \Autz(X)_{\p^1}$ is a torus of positive dimension, it is contained in a maximal torus $T$ of $G$, which contains a subtorus $T'\subseteq T$ such that $T'\to H$ is an isogeny (again by \cite[IV.11.14, Corollary~1]{Bor91}). In particular, $T/ \Autz(X)_{\p^1}$ is isomorphic to $H$, so $\dim T=\dim G$ and $G=T$ is again a torus.

We have proved that $G$ is isomorphic to $(\G_m)^r=(\C^*)^r$. There is a $G$-invariant open subset of $X$ isomorphic to $ (\G_m)^r\times U$, where $U$ is a smooth affine variety, such that  $G$ acts trivially on $U$ and acts on $(\G_m)^r$ by multiplication (see \cite[Proposition~2.5.1]{BFT2} or  \cite[Theorem~3]{Pop16}). Choosing a smooth projective variety $C$ birational to $U$, and embedding $ (\G_m)^r$ into $\p^r$, we obtain a birational map $\psi\colon X\dasharrow \p^r\times C$ such that $\psi G\psi^{-1}\subseteq \Autz(\p^r\times C)$. It remains to observe that this last inclusion is strict, as $\Aut(\p^r)$ embeds into $\Autz(\p^r\times C)$.
\end{proof}
\subsection{Orbits of small dimension}
\propsmallorbit*
\begin{proof}
Let us fix a polarization $H$ on $X$. By  \cite[Theorem~3.21.3]{KolRC}, for all positive integrs $h,\ell\ge 1$, there is a  projective variety $\Chow_{h,\ell}(X)$
which parametrises the proper algebraic cycles of dimension $h$ and degree $\ell$ with respect to $H$.
For more details on Chow varieties and the construction of these, we refer to \cite{KolRC} and \cite{Rydh}. For all $h,\ell\ge 1$, the group $\Autz(X)$ acts biregularly on $\Chow_{h,\ell}(X)$.

Let us write $G=(\Autz(X)_{B})^\circ$, so that  $k=\max\{\dim(G\cdot x)\mid x\in X\}$. By hypothesis, $G$ is of positive dimension, so $k>0$. 
The union of all orbits of dimension $k$ is a dense open $G$-invariant subset $U\subseteq X$ (Lemma~\ref{lemm:DimMaxopen}).

We denote by $R$ the set of closures of $k$-dimensional orbits of the linear connected algebraic group $G=(\Autz(X)_{B})^\circ$. The set $R$ is an uncountable subset of $\bigcup\limits_{\ell\ge 1} \Chow_{k,\ell}(X)$ that is pointwise fixed by the action of the group $G$ and left invariant by $\Autz(X)$, since $G$ is a normal subgroup of $\Autz(X)$. There is thus an integer $\ell\ge 1$ such that $R\cap \Chow_{k,\ell}(X)$ is uncountable.

For each integer $\ell\ge 1$, we decompose the closure of $R\,\cap\, \Chow_{k,\ell}(X)$ in $\Chow_{k,\ell}(X)$  into finitely many irreducible components $R_{\ell,j}$ and consider the irreducible varieties $Z_{\ell,j}=\{(x,[t])\in X\times R_{\ell,j}\vert\;x\in t\}$, which have dimension equal to $\dim R_{\ell,j}+k$. The morphism
\[\bigcup_{\ell,j}Z_{\ell,j}\to X\]
given by the first projection is surjective, since $\overline{U}=X$. On the other hand, there are countably many pairs $(\ell,j)$, therefore there exists a pair $(\ell,j)$ such that the first projection $\beta\colon Z_{\ell,j}\to X$ is surjective, which implies that $R_{\ell,j}$ has dimension at least $\dim X-k$. We now prove that the morphism $\beta$ is  generically injective. We denote by $p\colon Z_{\ell,j}\to R_{\ell,j}$ the second projection. There is an open set $V\subseteq R_{\ell,j}$ such that a point in $V$ corresponds to a unique cycle in $X$.
The morphism $\beta$ is $G$-equivariant, therefore it sends fibres of $p$ inside closures of orbits of the action of $G$ in $X$. We prove that $\beta$ is injective on $p^{-1}V\cap \beta^{-1}U$. For this, we take two points $y,y'\in p^{-1}V\cap \beta^{-1}U$ having the same image  $x=\beta(y)=\beta(y')\in X$. We write $y=(x,[t])$ and $y'=(x,[t'])$ where $t,t'\in R_{\ell,j}$. Then, $\beta (p^{-1}t)$, $\beta (p^{-1}t')$ are contained in the closure of the same orbit as they have the point $x$ in common. As the orbit of $x$ has dimension $k$,  
we have  $\beta (p^{-1}t)=\beta (p^{-1}t')$ implying $t=t'$.

Since the morphism $\beta$ is generically injective, it is birational.

As $\Autz(X)$  acts on $S'=R_{\ell,j}$ and on $Z'=Z_{\ell,j}$, there is an action of $\Autz(X)$ on the normalisations $S$ and $Z$ of $S'$ and $Z'$. We obtain two  $\Autz(X)$-equivariant morphisms $Z\to X$ and $Z\to S$; the morphism $Z\to X$ is birational and a general fibre of $Z\to S$ is a unirational variety of dimension $k$, so $\dim S=\dim Z-k=\dim X-k>\dim B$, since  we assumed $k<\dim X-\dim B$.
Applying an $\Autz(X)$-equivariant resolution of singularities and an $\Autz(X)$-equivariant resolution of indeterminacies, we can assume that $S$ and $Z$ are smooth. We run a $K_{Z}$-MMP over $S$ with scaling of a relatively ample divisor.
Since $k>0$ and because the fibres of $Z\to S$ are unirational, this MMP terminates by Lemma~\ref{relMfs}
with a Mori fibre space over $S$  that we denote by  $\pi'\colon Y\to C\to S$. We have also $\dim C\ge \dim S=\dim X-k>\dim B$.
\end{proof}

\section{Symmetric smooth Fano threefolds}\label{Sec:SymmSmoothFano}

\subsection{The list of symmetric smooth Fano threefolds}
In Table~\ref{3folds.MFS} we recall the list \cite[Table~1]{CFST16} of all the smooth threefolds $F$ with Picard rank at least~$2$ that are fibres of klt Mori fibre spaces. These are varieties which are ``symmetric'', in the sense that there is a finite group $G\subseteq \Aut(F)$ such that $\Pic(F)^{G}$ has rank $1$ \cite[Theorem~1.2]{ProGFano2}. The numeration that we use is the one of \cite[Table~1]{CFST16}. It is almost the same  numeration as in \cite[Theorem~1.2]{ProGFano2}: Cases \hyperlink{T1a}{$1a$} and \hyperlink{T1b}{$1b$} correspond to (1.2.1) in \cite[Theorem~1.2]{ProGFano2} (which is also subdivided into two cases), and Cases $\hyperlink{T2}{2}, \hyperlink{T3}{3},\ldots,\hyperlink{T8}{8}$ correspond respectively to $(1.2.2),(1.2.3),\ldots,(1.2.8)$ in \cite[Theorem~1.2]{ProGFano2}. 

\begin{table}[ht]\label{3folds.MFS}
\begin{center}

\begin{tabular}{*{4}{c} | p{8.5cm}}
 &&$\rho(F)$ & $-K_F^3$ & Description of $F$\\
\hline
$\hypertarget{T1a}{1a}$ & (6a) & 2 & 12 & A divisor of bidegree $(2, 2)$ in $\p^2 \times \p^2$.\\
$\hypertarget{T1b}{1b}$ & (6b) & 2 & 12 & A $2:1$ cover of a smooth divisor $W$ of bidegree $(1, 1)$ in $\p^2 \times \p^2$ branched along a member of $\lvert-K_W\rvert$.\\ 
$\hypertarget{T2}{2}$ & (12) &2 & 20 & The blow-up of $\p^3$ along a curve of degree $6$ and genus $3$ which is an 
intersection of cubics.\\
$\hypertarget{T3}{3}$ & (28) &2 & 28 & The blow-up of a smooth quadric $Q \subset \p^4$ along a smooth rational curve
of degree 4 which spans~$\p^4$. \\
$\hypertarget{T4}{4}$ & (32) &2 & 48 & A divisor of bidegree $(1, 1)$ in $\p^2 \times \p^2$.\\
$\hypertarget{T5}{5}$ & (1) &3 & 12 & A double cover of $\p^1 \times \p^1 \times \p^1$ branched along a member of  $\lvert-K_{\p^1 \times \p^1 \times \p^1}\rvert$.\\
$\hypertarget{T6}{6}$ &(13) &3 & 30 & The blow-up of a smooth divisor of bidegree $(1, 1)$ in $\p^2 \times \p^2$ along a curve $C$ of bidegree $(2, 2)$, such that $C \hookrightarrow W \hookrightarrow \p^2 \times \p^2 \to \p^2$ is an embedding for both projections $\p^2 \times \p^2 \to \p^2$.
\\
$\hypertarget{T7}{7}$ & (27) &3 & 48 & $\p^1 \times \p^1 \times \p^1$.\\
$\hypertarget{T8}{8}$ & (1) &4 & 24 & A smooth divisor of multidegree $(1, 1, 1, 1)$ in $\p^1 \times \p^1 \times \p^1 \times \p^1$.\\
\end{tabular}
\caption{Smooth Fano varieties being general fibres of klt Mori fibre spaces. The first column indicates the numeration of \cite{CFST16}, which is also the numeration of \cite{ProGFano2},  and the second column indicates the numeration of \cite{MoriMukai1}. }
\end{center}
\end{table}
\subsection{Automorphisms of symmetric smooth Fano threefolds}
Because of Proposition~\ref{Prop:Torus}, the interesting threefolds $F$ in the list are those such that $\Autz(F)$ is not a torus. We will show that $\Autz(F)$ is trivial in all cases except \hyperlink{T3}{$3$}, \hyperlink{T4}{$4$}, \hyperlink{T6}{$6$} or \hyperlink{T7}{$7$} (Proposition~\ref{prop:AutFiniteTable} below). The cases  \hyperlink{T4}{$4$} and \hyperlink{T7}{$7$} consist of one isomorphism class, with  automorphism group not being a torus; we moreover prove that in the families \hyperlink{T3}{$3$} and \hyperlink{T6}{$6$}, there is only one isomorphism class where $\Autz(F)$ is not a torus (Section~\ref{SubSec:PosDim}).

To prove Proposition~\ref{prop:AutFiniteTable}, we will need the following result whose proof is adapted from the proof of  \cite[Proposition~1.2]{BeauPrym}.
\begin{lemma}\label{Lem:Discriminant}
Let $S$ be a smooth threefold given either by
\[\hypertarget{S1}{(i)}\;\;\left\{([x_0:x_1:x_2],(u,v))\in \p^2\times \A^2\left|
\begin{bmatrix} x_0& x_1& x_2\end{bmatrix}\cdot M(u,v)\cdot \begin{bmatrix} x_0\\ x_1\\ x_2\end{bmatrix}=0 \right\}\right.\]
for some symmetric matrix $M\in \mathrm{Mat}_{3\times 3}(\C[u,v])$ or by 
\[\hypertarget{S2}{(ii)}\;\;\left\{([x_0:x_1],[y_0:y_1],(u,v))\in \p^1\times \p^1 \times \A^2\left|
\begin{bmatrix} x_0& x_1\end{bmatrix}\cdot M(u,v)\cdot \begin{bmatrix} y_0\\ y_1\end{bmatrix}=0 \right\}\right.\]
for some matrix  $M\in \mathrm{Mat}_{2\times 2}(\C[u,v])$.
Let $\Delta\subseteq \A^2$ be the zero locus of $\det M$ and let $\pi\colon S\to \A^2$ be the projection on the last factor. The following hold:
\begin{enumerate}
\item\label{Deltanodal}
$\Delta$ is a reduced curve of $\A^2$, smooth if $S \subseteq \p^1\times \p^1\times \A^2$ is as in \hyperlink{S2}{$(ii)$} and with only ordinary double points if $S\subseteq\p^2\times \A^2$ is as in \hyperlink{S1}{$(i)$};
\item\label{Deltathreecases}
For each $p\in \A^2$, the fibre $\pi^{-1}(p)$ is isomorphic to a conic in $\p^2$ which is  
\[\left\{\begin{array}{ll} \text{not reduced $($double line}) & \text{ if }p \text{ is a singular point of }\Delta;\\
\text{reduced and singular $($two distinct lines}) & \text{ if }p \text{ is a smooth point of }\Delta;\\
\text{smooth $($isomorphic to }\p^1) & \text{ if }p \text{ does not belong to }\Delta.\end{array}\right.\]
\end{enumerate}
\end{lemma}
\begin{proof} We take a point $p\in \A^2$ and consider the fibre $\pi^{-1}(p)\subset X$, and the matrix $M(p)$ associated to $p$. If $p\not\in \Delta$, then $M(p)$ is invertible, so $\pi^{-1}(p)\subset X$ is a smooth conic in $\p^2$ (respectively a smooth curve of bidegree $(1,1)$ in $\p^1\times \p^1$). This gives the third case of \ref{Deltathreecases}. We then assume that $p\in \Delta$, in which case $M(p)$ is not invertible, so $ \pi^{-1}(p)$ is not smooth. We prove that either $p$ is a smooth point of $\Delta$, or an ordinary double point of $\Delta$, and that these two cases give the first two cases of \ref{Deltathreecases}. For this, one can change coordinates on $\A^2$ and assume that $p=(0,0)$. We choose $r=5$ (respectively $r=3$) and $\alpha_0,\ldots,\alpha_r\in \C[u,v]$ so that either
\[M=\begin{pmatrix} \alpha_0 & \frac{1}{2}\alpha_1 & \frac{1}{2}\alpha_3\\  \frac{1}{2}\alpha_1 &\alpha_2& \frac{1}{2}\alpha_4
 \\ \frac{1}{2}\alpha_3& \frac{1}{2}\alpha_4& \alpha_5
\end{pmatrix}\text{ or }\begin{pmatrix} \alpha_0 & \alpha_1 \\ \alpha_2&\alpha_3
\end{pmatrix}\]
and the equation of $S$ is either
 \[\begin{array}{llll}
\hyperlink{S1}{(i)}\;\;\alpha_0 x_0^2+\alpha_1 x_0x_1+\alpha_2 x_1^2+\alpha_3 x_0x_2+\alpha_4 x_1x_2+\alpha_5 x_2^2&=&0&\text{ or }\\
\hyperlink{S2}{(ii)}\;\alpha_0 x_0y_0+\alpha_1 x_0y_1+\alpha_2 x_1y_0+\alpha_3 x_1y_1&=&0.\end{array}\]

We first assume that $M(p)$ is the zero matrix (i.e.~$\alpha_i(p)=0$ for $i=0,\ldots,r$), and derive a contradiction. Each symmetric matrix $R\in \mathrm{Mat}_{3\times 3}(\C)$  (respectively each matrix $R\in \mathrm{Mat}_{2\times 2}(\C)$) defines a closed subset $C_R$ in $\p^2$ (respectively $\p^1\times\p^1$) by $
\begin{bsmallmatrix} x_0& x_1& x_2\end{bsmallmatrix}\cdot R\cdot \begin{bsmallmatrix} x_0\\ x_1\\ x_2\end{bsmallmatrix}=0$ (or $
\begin{bsmallmatrix} x_0& x_1\end{bsmallmatrix}\cdot R\cdot {\small\begin{bsmallmatrix} y_0\\ y_1\end{bsmallmatrix}}=0$). The corresponding subsets $C_{\frac{\partial M}{\partial u}(p)}$ and $C_{\frac{\partial M}{\partial v}(p)}$ for $\frac{\partial M}{\partial u}(p)$ and $\frac{\partial M}{\partial v}(p)$  have non-empty intersection; we can then change coordinates in $\p^2$ (respectively $\p^1\times \p^1$) and assume that $[1:0:0]$  (respectively $([1:0],[1:0])$) belongs to $C_{\frac{\partial M}{\partial v}(p)}\cap C_{\frac{\partial M}{\partial u}(p)}$. 
This implies that $\frac{\partial \alpha_0}{\partial u}(p)=\frac{\partial \alpha_0}{\partial v}(p)=0$. Hence, the point $([1:0:0],p)$ (respectively $([1:0],[1:0],p)$) is a singular point of~$S$, contradicting the smoothness assumption.

We now assume that $\Ker M(p)$ has dimension $1$. Changing coordinates in $\p^2$ (respectively $\p^1\times \p^1$), we may assume that $M(p)$ is diagonal, with $\alpha_0(p)=0$ and $\alpha_2(p)=\alpha_5(p)=1$ (respectively $\alpha_3(p)=1$). Hence, $\pi^{-1}(p)$ is given by $x_1^2+x_2^2=0$ (respectively $x_1y_1=0$), which is isomorphic to the union of two lines in~$\p^2$. The fact that $S$ is smooth at $([1:0:0],p)$ (respectively $([1:0],[1:0],p)$) implies that $(\frac{\partial \alpha_0}{\partial u}(p),\frac{\partial \alpha_0}{\partial v}(p))\not=(0,0)$. The form of the diagonal matrix $M(p)$ implies that $\det M -\alpha_0$ has multiplicity $2$ 
 along $p$, which implies that $\det M$ has multiplicity $1$ at $p$, so $p$ is a smooth point of $\Delta$.

The remaining case is when $\Ker M(p)$ has dimension $2$. Since $M(p)\not=0$, this case only occurs for $\p^2\times \A^2$. We may change coordinates and assume that $M(p)$ is diagonal with $\alpha_5(p)=1$ and $\alpha_i(p)=0$ for $i=1,\ldots,4$. Hence $\pi^{-1}(p)$ is defined by $x_2^2=0$, which is a double line.  

It remains to prove that $p$ is an ordinary double point of $\Delta$.
For each $i\in \{0,1,2\}$ we denote by $a_i=u\frac{\partial \alpha_i}{\partial u}(p)+v\frac{\partial \alpha_i}{\partial v}(p)$ the linear part of $\alpha_i$, so that $\alpha_i-a_i$ has multiplicity at least $ 2$ at $p$. As $S$ is smooth along $\pi^{-1}(p)$, the polynomial $a_0 x_0^2+a_1 x_0x_1+a_2 x_1^2$ has multiplicity $1$ at any point of $\pi^{-1}(p)$ (note that $\alpha_3 x_0x_2+\alpha_4 x_1x_2+\alpha_5 x_2^2$ has multiplicity $2$ along any point of $\pi^{-1}(p)$ as $x_2$ and $\alpha_3,\alpha_4$ vanish on it). The smoothness of the point $([1:0:0],p)$ is equivalent to to the condition $a_0\not=0$. After a linear change of coordinates of $\A^2$, we may assume that $a_0=u$. We then replace $x_0$ by $x_0+\xi x_1$ for some $\xi\in \C$. Under this change of coordinates $a_1$ becomes $a_1+2a_0\xi$, and we may assume that $a_1=\epsilon v$  for some $\epsilon\in \C$. We then write $a_2=\lambda u+\mu v$ for some $\lambda,\mu\in \C$. For each $\theta\in \C$, the point $([\theta:1:0],p)$ is smooth, so we get $$0\not=\theta^2a_0+\theta a_1+a_2=(\theta^2+\lambda)u+(\theta \epsilon +\mu)v.$$ Choosing $\theta$ such that $\theta^2=-\lambda$, we get $\theta\epsilon+\mu\not=0$ and $-\theta\epsilon+\mu\not=0$, so $0\not=(\theta\epsilon+\mu)(-\theta\epsilon+\mu)=\lambda\epsilon^2+\mu^2$.

The polynomial $\det(M)-\alpha_0\alpha_2+\frac{1}{4}\alpha_1^2$, has multiplicity at least $ 3$ at the origin, so $p$ is a singular point of~$\Delta$ and it remains to see that $a_0a_2-\frac{1}{4}a_1^2$ is not a square. This is because  $a_0a_2-\frac{1}{4}a_1^2=u(\lambda u+\mu v)-\frac{1}{4}\epsilon^2v^2$, whose discriminant is $\lambda\epsilon^2+\mu^2$.
\end{proof}
\begin{proposition}\label{prop:AutFiniteTable}

Let $F$ be a smooth Fano threefold being in the list of Table~$\ref{3folds.MFS}$ $($or equivalently of \cite[Table~1]{CFST16}$)$. If $\Autz(F)$ is not trivial, then $F$ belongs to the families \hyperlink{T3}{$3$}, \hyperlink{T4}{$4$}, \hyperlink{T6}{$6$} or \hyperlink{T7}{$7$} $($respectively $(28)$, $(13)$, $(32)$, $(27))$ in the notation of \cite{CFST16} $($respectively of \cite{MoriMukai1}$)$.

\end{proposition}
\begin{proof}
We study the list case-by-case and prove that $\Aut(F)$ is finite in cases 
\hypertarget{T1a}{$1a$},
\hypertarget{T1b}{$1b$},
\hypertarget{T2}{$2$},
\hypertarget{T5}{$5$} and
\hypertarget{T8}{$8$}. 
We use the notation  of \cite{CFST16}, \textit{i.e.}~the first column.

\hyperlink{T1a}{$1a:$} In case \hyperlink{T1a}{$1a$}, $F$ is a hypersurface of bidegree $(2, 2)$ in $\p^2 \times \p^2$, that we can view as
\[F=\left\{([x_0:x_1:x_2],[y_0:y_1:y_2])\in \p^2\times \p^2\left|
\begin{bmatrix} y_0& y_1& y_2\end{bmatrix}\cdot M(x)\cdot \begin{bsmallmatrix} y_0\\ y_1\\ y_2\end{bsmallmatrix}=0 \right\}\right.\]
where $M$ is a symmetric $3\times 3$-matrix whose coefficients are homogeneous polynomials of degree $2$ in $\C[x_0,x_1,x_2]$. We consider the first projection $\pi\colon F\to  \p^2$, whose fibres are conics.  As $F$ is smooth, the curve  $\Delta\subset \p^2$ given by the polynomial $\det(M)$, which parametrises the singular fibres, is reduced and has only ordinary double points (it follows from Lemma~\ref{Lem:Discriminant} applied to affine charts of $\p^2$).

By Blanchard's Lemma~\ref{blanchard}, the group $\Autz(F)$ acts on $\p^2$ via  a connected subgroup $H\subseteq \Aut(\p^2)\simeq \PGL_3(\C)$. We will prove that $H$ is trivial, which implies that $\Autz(F)$ is trivial, as we can make the same argument with the other projection. The group $H$ preserves the reduced curve $\Delta\subset \p^2$ of degree $6$. Suppose first that an irreducible component $C$ of $\Delta$ is not a line. The action of $H$ on $C$ gives an injective group homomorphism $H\hookrightarrow \Aut(C)$. If $C$ is not rational, then $H$ is trivial, as $\Aut(C)$ does not contain any connected linear algebraic group of positive dimension, so we may assume that $C$ is rational. Every singular point of $C$ is an ordinary node, so if there are at least two singular points, $\Autz(C)$ is trivial. Hence we may assume that $\deg(C)\in \{2,3\}$. Then $C'=\overline{\Delta\setminus C}$ intersects $C$ in $\deg(C)\cdot \deg(C')\ge 8$ points, all fixed by $H$, which implies again that $H$ is trivial. The remaining case is when $\Delta$ is a union of $6$ lines, no $3$ of which have a common point. One may change coordinates so that four of the lines are given by $xyz(x+y+z)=0$, which implies that $H$ is trivial, as it has to leave each of the four lines invariant.

\hyperlink{T1b}{$1b:$} In case \hyperlink{T1b}{$1b$}, $F$ is a $2:1$ cover of a smooth divisor $W$ of bidegree $(1, 1)$ in $\p^2 \times \p^2$ branched along a member $D$ of $\lvert-K_W\rvert$. We consider the composition $p_i\colon F\to W\to \p^2$ of the covering with the projection on the $i$-th factor, and observe that fibres are conics. By Blanchard's Lemma~\ref{blanchard},  $\Autz(F)$ acts on $\p^2$, making the morphism equivariant. Hence, the group $\Autz(F)$ acts on $W$, via a connected algebraic subgroup $H\subseteq \Autz(W)$ that preserves $D$. As $W$ is smooth, the divisor $D$ is also smooth, and satisfies $K_D=0$. The Kodaira dimension of $D$ being non-negative, $\Autz(D)$ does not contain any connected linear algebraic group of positive dimension, so the action of $H$ on $D$ is trivial, and thus $H$ is trivial, as the set of the fixed points of a non-trivial automorphism of $\p^2$ is a finite union of lines and isolated points.

\hyperlink{T2}{$2:$} In case \hyperlink{T2}{$2$}, the blow-up $F\to \p^3$ is  $\Autz(F)$-equivariant (Lemma~\ref{blanchard}). One may thus see $\Autz(F)$ as the subgroup of $\Aut(\p^3)=\PGL_4(\C)$ preserving the blown-up curve $\Gamma\subset \p^3$  of degree $6$ and genus $3$. As this curve is not contained in a plane, the action on $\Gamma$ gives an injective group homomorphism $\Autz(F)\hookrightarrow \Autz(\Gamma)$. This shows that $\Autz(F)$ is trivial as $\Aut(\Gamma)$ is finite.

\hyperlink{T5}{$5:$} In case \hyperlink{T5}{$5$}, $F$ is a double cover of $\p^1 \times \p^1 \times \p^1$ whose branch locus is a divisor $D$ of
tridegree $(2, 2, 2)$. We consider the morphism $F\to \p^1\times \p^1$ obtained by projecting onto the first two coordinates. As the fibres are connected, by Blanchard's Lemma~\ref{blanchard} the group $\Autz(F)$ acts on $\p^1\times \p^1$. Doing the same with the other projections on two factors, one obtains a group homomorphism $\Autz(F)\to \Autz(\p^1 \times \p^1 \times \p^1)\simeq \PGL_2(\C)^3$ whose image is a connected group $H$ that preserves the divisor $D$. As $F$ is smooth, the divisor $D$ is also smooth. Moreover, $K_D\sim 0$ so the Kodaira dimension of $F$ is non-negative, which implies that $\Autz(F)$ does not contain any connected linear algebraic group of positive dimension. The action of $H$ on $F$ is then trivial, which implies that $H$ is trivial.

\hyperlink{T8}{$8:$} In case \hyperlink{T8}{$8$}, $F$ is a smooth divisor of multidegree $(1, 1, 1, 1)$ in $\p^1 \times \p^1 \times \p^1 \times \p^1= (\p^1)^4$. 
We can view $F$ as the set
\[\left\{([x_0:x_1],[y_0:y_1],[u_0:u_1],[v_0:v_1])\in (\p^1)^4\left|
\begin{bmatrix} x_0& x_1\end{bmatrix}\cdot M(u,v)\cdot \begin{bmatrix} y_0\\ y_1\end{bmatrix}=0 \right\}\right.\]
where $M$ is a symmetric $2\times 2$-matrix whose coefficients are homogeneous polynomials of bidegree $(1,1)$ in $\C[u_0,u_1][v_0,v_1]$. 
We consider the projection $\pi\colon F\to\p^1\times \p^1$ onto the last two factors. By Lemma~\ref{Lem:Discriminant} applied to affine charts of $\p^1 \times \p^1$,
as $F$ is smooth, the curve  $\Delta\subset \p^2$ defined by the polynomial $\det(M)$, parametrising the singular fibres, is smooth. By Blanchard's Lemma~\ref{blanchard}, the group $\Autz(F)$ acts on $\p^1\times \p^1$ via  a connected subgroup $H\subseteq \Autz(\p^1\times \p^1)\simeq \PGL_2(\C)\times \PGL_2(\C)$. We will prove that $H$ is trivial, which implies that $\Autz(F)$ is trivial, as we can make the same argument with the projection onto the first two factors. The group $H$ preserves the smooth curve $\Delta\subset \p^1\times \p^1$ which is of bidegree $(2,2)$ and is thus of genus $1$. As $\Aut(\Delta)$ does not contain any connected linear algebraic group of positive dimension, the group $H$ acts trivially on $\Delta$, and is thus the trivial group, as the set of fixed points of every non-trivial element of $\Autz(\p^1\times \p^1)$ consists of unions of isolated points and fibres of the projections.\end{proof}

\subsection{Explicit descriptions of the groups of automorphisms of positive dimension}\label{SubSec:PosDim}
According to Proposition~\ref{prop:AutFiniteTable}, the only  smooth Fano threefolds $F$ in the list of Table~$\ref{3folds.MFS}$ for which $\Autz(F)$ is not trivial belong to the families \hyperlink{T3}{$3$}, \hyperlink{T4}{$4$}, \hyperlink{T6}{$6$} or \hyperlink{T7}{$7$}. We now describe $\Autz(F)$ in these cases, and prove that in each family there are finitely many isomorphism classes for which $\Autz(F)$ is not a torus.

In Case \hyperlink{T7}{$7$}, $F=\p^1\times \p^1\times \p^1$, so $\Autz(F)=\PGL_2(\C)\times \PGL_2(\C)\times \PGL_2(\C)$ (this is classical and follows from Blanchard's Lemma~\ref{blanchard} applied to the projections $F\to \p^1$). We then consider the cases \hyperlink{T3}{$3$}, \hyperlink{T4}{$4$} and \hyperlink{T6}{$6$} in Lemmata~\ref{T3Aut}, \ref{T4Aut} and \ref{T6Aut} respectively.
\begin{lemma}\label{T3Aut}
Let $F$ be the blow-up of a smooth quadric $Q \subset \p^4$ along a smooth rational curve $C$
of degree $4$ which spans~$\p^4$ $($Case~\hyperlink{T3}{$3$} of Table~$\ref{3folds.MFS})$.
\begin{enumerate}
\item
$\Autz(F)$ is either trivial, or isomorphic to $\G_a$, $\G_m$ or $\PGL_2(\C)$, all cases being possible.
\item
If $\Autz(F)\simeq \PGL_2(\C)$, then up to a change of coordinates, $Q$ is given by $x_0x_4 - 4x_1x_3 + 3x_2^2=0$ and $C$ is the image of the Veronese embedding of degree $4$ of $\p^1$.
\item
If $\Autz(F)\simeq \G_a$, then up to a change of coordinates, $Q$ is given by $x_0x_4 - 4x_1x_3 + 3x_2^2+x_0x_2 - x_1^2=0$ and $C$ is the image of the Veronese embedding of degree $4$ of $\p^1$. Moreover, there is a unique point $p\in F$ fixed by $\Autz(F)$.
\end{enumerate}
\end{lemma}
\begin{proof}By Blanchard's Lemma~\ref{blanchard}, the morphism $\eta\colon F\to Q$ is $\Autz(F)$-equivariant, so $\Autz(F)$ is conjugate via $\eta$ to the connected component $H^{\circ}$ of the group
 \[H=\{g\in \Aut(Q)\mid g(C)=C\}=\{g\in \Aut(\p^4)\mid g(C)=C, g(Q)=Q\}\]containing the identity.
Changing coordinates on $\p^4$, we may assume that $C$ is the image of the Veronese embedding
\[\tau\colon \p^1\to \p^4, [u:v]\mapsto [u^4:u^3v:u^2v^2:uv^3:v^3].\]
In particular, $H$ is contained in $\hat{H}=\{g\in \Aut(\p^4)\mid g(C)=C\}$, which is isomorphic to $\Aut(C)\simeq \Aut(\p^1)\simeq\PGL_2(\C)$.
We then choose the following basis of the vector space of polynomials of degree $2$ vanishing on $C$:
\[\begin{array}{rclrcl}
f_0&=&x_0x_4 - 4x_1x_3 + 3x_2^2,& f_3&=&x_1x_3 - x_2^2\\
 f_1&=&x_0x_2 - x_1^2,&f_4&=&x_1x_4 - x_2x_3,\\
  f_2&=&x_0x_3 - x_1x_2,& f_5&=&x_2x_4 - x_3^2.\end{array}\]

One then verifies that $f_0$ is invariant by $\hat{H}\simeq \PGL_2(\C)$.

$(i)$ Suppose first that some torus $\G_m$ is contained in $H$. Conjugating by an element of $\hat{H}$, the torus acts on the image of the Veronese embedding via $\tau$ as $[u:v]\mapsto [u:\xi v]$, $\xi\in \C^*$ and thus acts on $\p^4$ as $[x_0:\cdots:x_4]\mapsto [x_0:\xi x_1:\xi^2 x_2:\xi^3 x_3 :\xi^4x_4]$.  Replacing $x_i$ with $\xi^i x_i$ in $f_0,\ldots,f_5$ yields $\xi^4f_0, \xi^2 f_1, \xi^3f_2,\xi^4 f_3,\xi^5 f_4,\xi^6 f_6$. Hence, the equation of $Q$ is either given by $f_0+\kappa f_3=0$ for some $\kappa\in \C$, or is given by $f_i$ for some $i\in \{1,\ldots,6\}$. This latter case is impossible as $Q$ is smooth, so we are in the former case. If $\kappa=0$, then $H\simeq \PGL_2(\C)$, as $f_0$ is invariant by $\hat{H}$. We may thus assume that $\kappa\not=0$ and prove that $H=\G_m\rtimes \langle \sigma\rangle$, where $\sigma$ is the involution $\sigma\colon [x_0:x_1:\cdots:x_4:x_5]\mapsto [x_5:x_4:\cdots:x_1:x_0]$, which preserves~$Q$. As $\PGL_2(\C)$ acts $2$-transitively on $\p^1$, the group $\G_m\rtimes \langle \sigma\rangle$, which is the group of elements preserving $\{[1:0],[0:1]\}$, is maximal in $\hat{H}$. Hence, it remains to see that $H\not=\hat{H}$ when $\kappa\not=0$. We simply consider the automorphism of $\p^1$ given by $[u:v]\mapsto [u:v+u]$, which induces the automorphism $ \nu\colon [x_0:\cdots:x_4]\mapsto [x_0:x_1+x_0:x_2+2x_1+x_0:x_3+3x_2+3x_1+x_0:x_4+4x_3+6x_2+4x_1+x_0]$ on $\p^4$, and observe that $\nu$ does not preserve $Q$.

$(ii)$ Suppose now that no torus $\G_m$ is contained in $H$. This implies that either $H$ is finite (in which case $\Autz(F)$ is trivial) or $H^{\circ} \simeq \G_a$. Conjugating by an element of $\hat{H}$, we may assume that $H$ corresponds to $[u:v]\mapsto [u:v+\xi u]$, $\xi\in \C$, and thus contains the element $\nu$ given above. We then compute that $ f_0\circ \nu,\ldots,f_6\circ \nu$ are equal to 
$f_0,f_1,2f_1+f_2,f_1+f_2+f_3,f_0+2f_1+3f_2+6f_3+f_4,f_0+f_1+2f_2+6f_3+2f_4+f_5$
respectively.  This implies that the equation of $Q$ is a linear combination of $f_0$ and $f_1$. As $\{f_1=0\}$ is singular and no torus is contained in $H$, the equation is of the form $f_0+\lambda f_1$ for some $\lambda\in \C^*$. Conjugating by $[u:v]\mapsto [u:\xi v]$ for some $\xi\in \C^*$ we obtain the equation $f_0+f_1$. We now prove that $H^{\circ}$ is indeed isomorphic to $\G_a$ for this equation. Indeed, otherwise $H^{\circ}$ would be bigger and thus would contain the torus $[u:v]\mapsto [u:\xi v]$, $\xi\in \C^*$, which is impossible because $f_0+f_1$ is not invariant under this torus. It remains to show that there is only one fixed point for $\Autz(F)$.
As $q=[0:\cdots:0:1]$ is the only point of $\p^4$ fixed by $\nu$, all fixed points of $\Autz(F)$ are contained in the preimage of $q$ in $F$, which is  a curve $e$ isomorphic to $\p^1$.
It remains then to see that the action of $\Autz(F)$ on $e$ is not trivial. Note that the tangent hyperplane of $Q$ at $q$ is given by $x_0=0$. The tangent line of $C$ at $q$ is given by $x_0=x_1=x_2=0$. The quadrics $f_0,f_4,f_5$ then generate the conormal bundle of $C$ at $q$ and $f_1,f_2,f_3$ are singular at $q$. Since $f_4\circ \nu=f_0$, $f_4\circ \nu=f_0+2f_1+3f_2+6f_3+f_4$ and $f_5\circ \nu=f_0+f_1+2f_2+6f_3+2f_4+f_5$, the action of $\Autz(F)$ on $e$ is not trivial.
\end{proof}

A smooth hypersurface of bidegree $(1,1)$ in $ \p^2\times \p^2$ is isomorphic to $\P(T_{\p^2})$, via any 
of the two projections. In the next lemma we recall the proof of this classical fact for the reader's convenience.

\begin{lemma}\label{T4Aut}
Let $F\subset  \p^2\times \p^2 $ be a smooth hypersurface of bidegree $(1,1)$ $($Case~\hyperlink{T4}{$4$} of Table~$\ref{3folds.MFS})$. Then, the following hold:
\begin{enumerate}
\item\label{F11standard}
Changing coordinates on $\p^2\times \p^2$, the threefold $F$ is given by
\[F=\left\{([x_0:x_1:x_2],[y_0:y_1:y_2])\in \p^2\times \p^2 \left| \textstyle \sum\limits_{i=0}^2 x_i y_i=0\right\}\right..\]
\item\label{PTP2PGL3}
The group $\PGL_3(\C)$ acts faithfully on $F$ via
\[\begin{array}{ccc}
\PGL_{3}(\C) \times F &\longrightarrow & F\\
\left(A, \left(\begin{bsmallmatrix} x_0\\ x_1\\ x_2\end{bsmallmatrix},\begin{bsmallmatrix} y_0\\ y_1\\ y_2\end{bsmallmatrix}\right)\right)&\mapsto& \left(A\cdot \begin{bsmallmatrix} x_0\\ x_1\\ x_2\end{bsmallmatrix}, \tr{A}^{-1}\cdot \begin{bsmallmatrix} y_0\\ y_1\\ y_2\end{bsmallmatrix}\right)\end{array}.\]
Moreover, this actions provides an isomorphism $\PGL_3(\C)\simeq \Autz(F)$.\end{enumerate}
\end{lemma}
\begin{proof}\ref{F11standard}: The variety $F$ is given by 
\[\left\{([x_0:x_1:x_2],[y_0:y_1:y_2])\in \p^2\times \p^2\left|
\begin{bsmallmatrix} x_0& x_1& x_2\end{bsmallmatrix}\cdot M\cdot \begin{bsmallmatrix} y_0\\ y_1\\ y_2\end{bsmallmatrix}=0 \right\}\right.\]
for some matrix $M\in \mathrm{Mat}_{3\times 3}(\C)$. 
After a change of coordinates on $\p^2\times \p^2$ of the form $(A,B)$ with $A,B\in \mathrm{GL}_3(\C)$, we can assume that $M$ is diagonal with all entries equal to either $0$ or $1$. Indeed, in the new coordinates, $F$ is given by the matrix 
$\tr{A}^{-1} M B^{-1}$.
As $F$ is smooth, this implies that $M$ is the identity matrix, so 
\[F=\left\{([x_0:x_1:x_2],[y_0:y_1:y_2])\in \p^2\times \p^2 \left| \textstyle \sum_{i=0}^2 x_i y_i=0\right\}\right..\]

\ref{PTP2PGL3}: The group $\PGL_3(\C)$ acts faithfully on $F$ via
\[\begin{array}{ccc}
\PGL_{3}(\C) \times F &\longrightarrow & F\\
\left(A, \left(\begin{bsmallmatrix} x_0\\ x_1\\ x_2\end{bsmallmatrix},\begin{bsmallmatrix} y_0\\ y_1\\ y_2\end{bsmallmatrix}\right)\right)&\mapsto& \left(A\cdot \begin{bsmallmatrix} x_0\\ x_1\\ x_2\end{bsmallmatrix}, \tr{A}^{-1}\cdot \begin{bsmallmatrix} y_0\\ y_1\\ y_2\end{bsmallmatrix}\right)\end{array}.\]
In particular, the group $\Autz(F)$ contains $\PGL_3(\C)$. It remains to see that each element of $\Autz(F)$ is of this form. To see this, we first use Blanchard's Lemma~\ref{blanchard} for the two projections $F\to \p^2$, which are $\p^1$-bundles, and obtain that each element of $\Autz(F)$ is of the form $(x,y)\mapsto (Ax,By)$ for some $A,B\in \PGL_3(\C)$. Applying the element $(x,y)\mapsto (\tr{B}x,B^{-1}y)$, we may assume that $B$ is the identity. It remains to see that $A$ is the identity too. Denoting by $\pi\colon F\to \p^2$ the second projection, the automorphism $(x,y)\mapsto (Ax,y)$ leaves invariant every fibre of $\pi$. In particular, it preserves $\pi^{-1}([1:0:0])$, $\pi^{-1}([0:1:0])$ and $\pi^{-1}([0:0:1])$, so $A$ is diagonal. It moreover preserves $\pi^{-1}([1:1:1])$, whose image under the first projection is the line $x_0+x_1+x_2=0$, and thus $A$ is the identity.
\end{proof}

Case~\hyperlink{T6}{$6$} of Table~$\ref{3folds.MFS}$ is presented in \cite{ProGFano2} as an intersection of three hypersurfaces of tridegree $(0,1,1)$, $(1,0,1)$ and $(1,1,0)$ in $\p^2\times \p^2\times \p^2$. In \cite[Case 1.2.6, page 422]{ProGFano2}, it is explained that varieties of that form are isomorphic to smooth varieties as in Case~\hyperlink{T6}{$6$}. Lemma~\ref{Lemm:Iso26} below gives an explicit way to see the converse. Lemma~\ref{T6Aut} then describes the group of automorphisms.

\begin{lemma}\label{Lemm:Iso26}
Let $T\subset \p^2\times \p^2$ be a smooth hypersurface of bidegree $(1,1)$. Let $C\subset T$ be a smooth curve of bidegree $(2,2)$, such that the projection to any of the two $\p^2$ gives an embedding $C\hookrightarrow \p^2$. Denoting by $\eta\colon F\to T$ the blow-up of $C$ and by $E\subset F$ the exceptional divisor, the following hold:
\begin{enumerate}
\item\label{PTP2CFano}
The threefold $F$ is a smooth Fano threefold of Picard rank $3$.
\item\label{PTP2Cbignef}
The divisor $D=-\frac{3}{2}K_F-\frac{1}{2}E$ is ample.
\item\label{PTP2CBirMorph}
The linear system $\lvert D+K_F\rvert$ gives a morphism $\kappa\colon F\to \p^2$. Moreover, the morphism $\kappa\times\eta$ gives a closed embedding $F\to \p^2\times \p^2\times \p^2$, that sends $F$ onto a the intersection of three hypersurfaces of tridegree $(0,1,1)$, $(1,0,1)$ and $(1,1,0)$.
\end{enumerate}\end{lemma}
\begin{proof}
After changing coordinates in $ \p^2\times \p^2$, we may assume that $C$ is the image of the morphism $\tau\colon \p^1\iso C$, $[u:v]\mapsto ([u^2:uv:v^2], [u^2:uv:v^2])$.

We have $K_T=-2H$, where $H\subset T$ is the intersection of $T$ with a hypersurface of $\p^2\times \p^2$ of bidegree $(1,1)$.
Since $K_F=\eta^{*}(K_T)+E$, we obtain $D+K_F=-\frac{1}{2}K_F-\frac{1}{2}E=-\frac{1}{2}\eta^{*}(K_T)-E=\eta^{*}(H)-E$. Therefore the linear system $\lvert D+K_F\rvert$ is the linear system of strict transforms of hypersurfaces of bidegree $(1,1)$ through~$C$. 
 The vector space of polynomials of  bidegree $(1,1)$ vanishing along $C$ is of dimension $4$, generated by \[ x_1y_2 - x_2y_1,\;\;x_2y_0-x_0y_2,\;\;x_0y_1 - x_1y_0 ,\;\;x_1y_1 - x_2y_0.\]
The equation of $T$ is a linear combination of the four above polynomials. We observe  that it is linearly independent of the first three. Indeed, for any $(a_0,a_1,a_2)\in \C^3\setminus \{0\}$, the hypersurface of $\p^2\times \p^2$ given by $a_0(x_1y_2 - x_2y_1)+a_1(x_2y_0-x_0y_2)+a_2(x_0y_1 - x_1y_0 )$ is singular at $([a_0:a_1:a_2],[a_0:a_1:a_2])\in \p^2\times \p^2$, as the derivative with respect to $x_i$ or $y_i$ is zero for all $i$. This implies that the intersection of $T$ with the diagonal of $\p^2\times \p^2$ is equal to $C$, that the linear system of hypersurfaces of $T$ of degree $(1,1)$ through $C$ is of dimension $3$, and that the rational map $\kappa\colon F\dasharrow \p^2$ induced by $\lvert D+K_F\rvert$ is equal to $\kappa=\theta\circ \eta$, where $\theta\colon \p^2\times \p^2\dasharrow \p^2$ is given by
\[\theta\big( ([x_0:x_1:x_2],[y_0:y_1:y_2])\big)= [x_1y_2 - x_2y_1:x_2y_0-x_0y_2:x_0y_1 - x_1y_0].\]
We consider the variety 
\[W=\{(x,y,z)\in \p^2\times \p^2\times \p^2\mid x_0z_2+x_1z_1+x_2z_0=0,\;\; y_1z_2+y_1z_1+y_2z_0=0\}\]
and observe that $\epsilon \colon W\to \p^2\times \p^2, (x,y,z)\mapsto (x,y)$ is the blow-up of the diagonal $\Delta\subset\p^2\times\p^2$. This can for instance be seen on the local charts $U_{ij}\subset \p^2\times \p^2$ given by $x_iy_j\not=0$ with $i,j\in \{0,1,2\}$. The inverse of the blow-up is  $(x,y)\mapsto (x,y,\theta(x,y))$. Hence, $\epsilon \colon \epsilon^{-1}(T)\to T$ is the blow-up of $T\cap \Delta=C$, with inverse $(x,y)\mapsto (x,y,\kappa(x,y))$.  It follows that the map $\kappa\times\eta\colon F\dasharrow \p^2\times \p^2\times \p^2$ is an isomorphism onto its image $F\iso\epsilon^{-1}(T)$. This proves~\ref{PTP2CBirMorph}, which in turn implies that $D+K_F=\eta^*(H)-E$ is semiample but not big. 

To see that $-K_F=2\eta^*(H)-E$ is ample, one observes that for each curve  $\Gamma\subset F$ not contracted by $\eta$, one has $\Gamma\cdot (-K_F)\ge  \Gamma\cdot \eta^*(H)>0$ since $\eta^*(H)-E$ is nef and that $\Gamma\cdot  (-K_F)\ge \Gamma\cdot (\eta^*(H)-E)>0$ for curves $\Gamma$ contracted by $\eta$. Moreover, $(-K_F)^3=30>0$, so we conclude by the Nakai-Moishezon criterion. 

We can also see that $F$ is in \cite[n$\degree$13 of Table~3]{MoriMukai1}. This gives \ref{PTP2CFano} and then implies that  $D$ is ample, since $D+K_F$ is semiample.
\end{proof}

\begin{lemma}\label{T6Aut}
Let $F$ be the blow-up of a smooth divisor $T\subset \p^2\times \p^2$ of bidegree $(1, 1)$ along a 
curve $C$ of bidegree $(2, 2)$, such that $C \hookrightarrow T \hookrightarrow \p^2 \times \p^2 
\to \p^2$ is an embedding for both projections $\p^2 \times \p^2 \to \p^2$, assume that $F$ is smooth $($Case~\hyperlink{T6}{$6$} of Table~$\ref{3folds.MFS})$. Then, the following hold:
\begin{enumerate}
\item\label{ThreeCasesAutzF6}
$\Autz(F)$ is either isomorphic to $\G_m$, or to $\G_a$ or to $\PGL_2(\C)$, all cases being possible.
\item\label{F6PGL2}
If $\Autz(F)\simeq \PGL_2(\C)$, then $F$ is isomorphic to 
\[F_0=\left\{([x_0:x_1:x_2],[y_0:y_1:y_2],[z_0:z_1:z_2])\in (\p^2)^3 \left| \begin{array}{lll}
 \sum_{i=0}^2 x_i y_i&=&0\\
 \sum_{i=0}^2 x_i z_i&=&0\\
 \sum_{i=0}^2 y_i z_i&=&0\end{array}\right\}\right..\]
The projection $F_0\to \p^2\times \p^2$ on two different factors yields a birational morphism $F_0\to T$, where
$T=\left\{([x_0:x_1:x_2],[y_0:y_1:y_2])\in \p^2\times \p^2 \left| \textstyle \sum_{i=0}^2 x_i y_i=0\right\}\right.,$
which is the blow-up of 
$\Gamma=\left\{([x_0:x_1:x_2],[x_0:x_1:x_2])\in \p^2\times \p^2 \left| \textstyle \sum_{i=0}^2 x_i^2=0\right\}\right.$.

For each matrix $M\in \PGL_3(\C)$ such that $\tr{M}\cdot M=\mathrm{id}$, we have an automorphism of $F_0$ given by
\[\begin{array}{ccc}
F_0 &\longrightarrow & F_0\\
 \left(\begin{bsmallmatrix} x_0\\ x_1\\ x_2\end{bsmallmatrix},\begin{bsmallmatrix} y_0\\ y_1\\ y_2\end{bsmallmatrix},\begin{bsmallmatrix} z_0\\ z_1\\ z_2\end{bsmallmatrix}\right)&\mapsto& \left(M\cdot \begin{bsmallmatrix} x_0\\ x_1\\ x_2\end{bsmallmatrix}, M\cdot \begin{bsmallmatrix} y_0\\ y_1\\ y_2\end{bsmallmatrix},M\cdot \begin{bsmallmatrix} z_0\\ z_1\\ z_2\end{bsmallmatrix}\right)\end{array}.\]
and every automorphism of $F_0$ is of this form. This gives an isomorphism $\Autz(F_0)\simeq \mathrm{PO}_3(\C)\simeq \PGL_2(\C)$.
\item\label{F6Ga}If $\Autz(F)\simeq \G_a$, there is only one isomorphism class for $F$. This latter is isomorphic to the intersection of three hypersurfaces of $\p^2\times \p^2\times \p^2$ of tridegree $(0,1,1)$, $(1,0,1)$ and $(1,1,0)$ such that the group $\Autz(F)$ acts on each of the three copies of $\p^2$ by fixing exactly one point.
\end{enumerate}
\end{lemma}
\begin{proof}
By definition, we have a birational morphism $\eta\colon  F\to T$, where $T\subset \p^2\times \p^2$ is a smooth divisor of bidegree $(1,1)$. This morphism is the blow-up of a curve $C$ of bidegree $(2,2)$.  
By Blanchard's Lemma~\ref{blanchard}, the morphism $\eta\colon F\to T$ is $\Autz(F)$-equivariant, so $\Autz(F)$ is conjugate via $\eta$ to the group $H=\{g\in \Autz(T)\mid g(C)=C\}$.
Changing coordinates, by Lemma~\ref{T4Aut}\ref{F11standard} we may assume that $T=\left\{(x,y)\in \p^2\times \p^2 \left| \textstyle \sum_{i=0}^2 x_i y_i=0\right\}\right.$. Moreover, the embeddings of $C$ into $\p^2$ given by each of the two projections $\pi_1,\pi_2\colon \p^2\times \p^2\to \p^2$ induce isomorphisms of $C$ with two conics $C_1\subset \p^2$ and $C_2\subset \p^2$. We may apply an automorphism of $T$ of the form $(\tr{A}^{-1},A)$ (cf. Lemma~\ref{T4Aut}\ref{PTP2PGL3})  and assume that $C_1$ is defined by $\sum_{i=0}^2 x_i^2=0$.

By Lemma~\ref{T4Aut}, $\PGL_3(\C)\simeq \Autz(T)$ acts on  $T$ by $(x,y)\mapsto (Ax, \tr{A}^{-1} y)$. The subgroup of $\PGL_3(\C)$ that preserves the conic $C_1\subset\p^2$ is the projective orthogonal group $\mathrm{PO}_3(\C)=\{M\in \PGL_3(\C)\mid M=\tr{M}\}\simeq \PGL_2(\C)$. Hence, the group $H$ is contained in the subgroup $\hat{H}\subset \Autz(T)$ given by $\hat{H}=\{(x,y)\mapsto (Mx,My)\mid M\in \mathrm{PO}_3(\C)\}\simeq \mathrm{PO}_3(\C)\simeq \PGL_2(\C)$. The closed curve 
$\Gamma=\{([x_0:x_1:x_2],[x_0:x_1:x_2])\in \p^2\times \p^2\mid \sum_{i=0}^2 x_i^2=0\}\subset T$ given by the diagonal embedding of $C_1$ into $T$ is invariant by $\hat{H}$, so one obtains $H=\hat{H}$ if $C=\Gamma$. In any case, $C$ is contained in the surface $(\pi_1)^{-1}(C_1)\subset T$, which is isomorphic to $\p^1\times \p^1$, via 
\[\begin{array}{cccc}
\tau\colon & \p^1\times \p^1& \iso & (\pi_1)^{-1}(C_1)\\
& ([a:b],[c:d])&\mapsto & ([a^2-b^2:\im( a^2+b^2) : 2ab],[ac - bd: \im (ac+bd):ad+bc]).
\end{array}\]
Moreover, the isomorphism $\tau$ sends the diagonal of $\p^1\times \p^1$ onto $\Gamma$. As $\hat{H}\simeq \PGL_2(\C)$ acts on $(\pi_1)^{-1}(C_1)\simeq \p^1\times \p^1$ via a faithful action on the first coordinate (corresponding to the action of $\mathrm{PO}_3(\C)$ on $C_1$) and preserves the diagonal, the action on $\p^1\times \p^1$ is the diagonal action (for a suitable isomorphism $\hat{H}\simeq \PGL_2(\C)$). In particular, $\Gamma$ is the unique curve of $(\pi_1)^{-1}(C_1)$ that is invariant by $\hat{H}$.

The curve $C$ is the image by $\tau$ of a curve $C'\subset\p^1\times \p^1$ of bidegree $(1,1)$. If $C$ is not equal to $\Gamma$, it intersects $\Gamma$ in two or one point. In the first case, we may apply an element of $\hat{H}$ and assume that the two points are the image by $\tau$ of $([0:1],[0:1])$ and $([1:0],[1:0])$, which implies that $C'$ is given by $ad+\xi bc$ for some $\xi\in \C\setminus \{-1\}$. Hence, $H$ is isomorphic to $\C^*$, acting as $([a:b],[c:d])\mapsto ([\lambda a:b],[\lambda c:d])$.
In the second case, we may assume that the point is $([0:1],[0:1])$. Hence, $C'$ is given by $ad-bc+\xi ac$ for some $\xi\in \C^*$. Applying an element of the form $([a:b],[c:d])\mapsto ([\lambda a:b],[\lambda c:d])$, we may assume that $\xi=1$. The group is then isomorphic to $\G_a$, via 
$([a:b],[c:d])\mapsto ([ a:b+\mu a],[c:d+\mu c])$. This achieves the proof of~\ref{ThreeCasesAutzF6}.

\ref{F6PGL2}: If $\Autz(F)\simeq \PGL_2(\C)$, then in the above description, $C$ is given by $\Gamma$. We write as above \[F_0=\left\{(x,y,z)\in (\p^2)^3 \left|  \sum\nolimits_{i=0}^2 x_i y_i=\sum\nolimits_{i=0}^2 x_i z_i=\sum\nolimits_{i=0}^2 y_i z_i=0\right\}\right.,\]
and consider the rational map $\tau\colon T\dasharrow F_0$ given by $(x,y)\mapsto (x,y,[x_1y_2 - x_2y_1:x_2y_0-x_0y_2:x_0y_1 - x_1y_0 ])$, which is $\mathrm{PO}_3(\C)$-equivariant, with an action on $F_0$ given by $(x,y,z)\mapsto (Mx,My,Mz)$ (follows from the fact that $\tau$ corresponds to the cross-product). Lemma~\ref{Lemm:Iso26} implies that $\tau\circ \eta\colon F\dasharrow F_0$ is an isomorphism and thus that the projection $F_0\to T$ is the blow-up of $C$.

\ref{F6Ga}: If $\Autz(F)\simeq \G_a$, then the curve $C$ is the image by $\tau$ of the curve of $\p^1\times \p^1$ given by $ad-bc+ac=0$. Hence, \[\begin{array}{rcl}
C&=&\{\tau([a:b],[a:b-a])\mid [a:b]\in \p^1\}\\
&=&\{([a^2-b^2:\im( a^2+b^2) : 2ab],[a^2+ab-b^2: \im (a^2+b^2-ab):2ab-a^2])\mid [a:b]\in \p^1\}.\end{array}\]
The action of $\Autz(F)\simeq \G_a$ on any of the two $\p^2$ preserves the conic and fixes a unique point of the conic, so has a unique point fixes on $\p^2$. This point is equal to $\tau([0:1],[0:1])$. Using the embedding of $F$ into $\p^2\times \p^2\times \p^2$ of  Lemma~\ref{Lemm:Iso26}, one gets an action of $\G_a$ on the third factor too, with a unique fixed point, indeed, the projection on the last two coordinates gives again a birational morphism which is the blow-up of a curve of bidegree $(2,2)$ (see \cite[Case 1.2.6, page 422]{ProGFano2}) and we can use the same argument as above.
\end{proof}
\begin{corollary}\label{Coro:PGL2}
There are exactly two smooth Fano threefolds $F$ which satisfy $\rho(F)\ge 2$,  $\Autz(F)\simeq \PGL_2(\C)$ and which can occur as general fibres of a Mori fibre space. These two threefolds are the following:
\begin{enumerate}[$(A)$]
\item\label{PGL2firstcase3}
The blow-up of the quadric $Q\subset \p^4$ given by $x_0x_4 - 4x_1x_3 + 3x_2^2=0$ along the image of the Veronese embedding of degree $4$ of $\p^1$.
\item\label{PGL2secondcase6}
The threefold 
\[\left\{(x,y,z)\in (\p^2)^3 \left| 
 \sum_{i=0}^2 x_i y_i= \sum_{i=0}^2 x_i z_i=\sum_{i=0}^2 y_i z_i=0\right\}\right..\]
\end{enumerate}
\end{corollary}
\begin{proof}
Let $F$ be a Fano threefold with $\rho(F)\ge 2$ and  $\Autz(F)\simeq \PGL_2(\C)$, which occur as a general fibre of a Mori fibre space. By \cite[Theorem~1.4]{CFST16}, the threefold needs to be in the list of Table~$\ref{3folds.MFS}$ $($or equivalently of \cite[Table~1]{CFST16}$)$. Since $\Autz(F)$ is not trivial, Proposition~\ref{prop:AutFiniteTable} implies that $F$ belongs to the families \hyperlink{T3}{$3$}, \hyperlink{T4}{$4$}, \hyperlink{T6}{$6$} or \hyperlink{T7}{$7$}. 

In Case \hyperlink{T3}{$3$}, Lemma~\ref{T3Aut} proves that $F$ is isomorphic to the threefold~\ref{PGL2firstcase3} above.

In Case \hyperlink{T4}{$4$}, it is impossible to have $\Autz(F)\simeq \PGL_2(\C)$ (Lemma~\ref{T4Aut}).

In Case \hyperlink{T36}{$6$}, Lemma~\ref{T6Aut} proves that $F$ is isomorphic to the threefold~\ref{PGL2secondcase6} above.

In Case \hyperlink{T7}{$7$}, we have $F\simeq (\p^1)^3$, contradicting $\Autz(F)\simeq \PGL_2(\C)$.
\end{proof}

\section{Symmetric birational maps from $(\p^1)^3$ or $\P(T_{\p^2})$}\label{Sec:Symmbirmap}
\subsection{Symmetric birational maps from $\p^1\times \p^1\times \p^1$}
In order to construct birational maps from a Mori fibre space $X\to B$ with general fibre a smooth Fano threefold $F$, we need to understand rational maps from $F$ to other varieties, which are symmetric enough.
Those will be typically induced by sublinear systems of $-mK_F$ for some  positive $m\in \Q$. In this section, we study the case of $F=\p^1\times \p^1\times \p^1$.

\begin{lemma}\label{Lem:CurveinP1P1P1link}
Let $C$ be a curve of tridegree $(1,1,1)$ in $F=\p^1\times \p^1\times \p^1$, let $\eta\colon \hat{F}\to F$ be the blow-up of $F$ along $C$, with exceptional divisor $E$. 
\begin{enumerate}
\item\label{P1P1P1CFano}
The threefold $\hat{F}$ is a smooth Fano threefold of Picard rank $4$.
\item\label{P1P1P1Cbignef}
The divisor $D=-\frac{3}{2}K_{\hat{F}}-\frac{1}{2}E$ is ample.
\item\label{P1P1P1CBirMorph}
The linear system $\lvert D+K_{\hat{F}}\rvert$ gives a birational morphism $\hat{F}\to \p^3$, which is  the contraction of the strict transforms of divisors of $F$ of tridegree $(0,1,1)$, $(1,0,1)$, $(1,1,0)$ through $C$, or equivalently the blow-up of three skew lines of $\p^3$.\end{enumerate}
\end{lemma}
\begin{proof}
We have $K_F=-2H$, where $H\subset F$ is a hypersurface of tridegree $(1,1,1)$.
Since $K_{\hat{F}}=\eta^{*}(K_F)+E$, we obtain $D+K_{\hat{F}}=-\frac{1}{2}K_{\hat{F}}-\frac{1}{2}E=-\frac{1}{2}\eta^{*}(K_F)-E=\eta^{*}(H)-E$.
Changing coordinates, we may assume that $C=\{([u:v],[u:v],[u:v])\mid [u:v]\in \p^1\}$. The divisors $H_1,H_2,H_3\subseteq F$  of tridegree $(0,1,1)$, $(1,0,1)$, $(1,1,0)$ through $C$ are then given by 
\[H_1=\{x_0y_1-x_1y_0=0\}, H_2=\{y_0z_1-y_1z_0=0\}, H_3=\{x_0z_1-x_1z_0=0\}.\]

The rational map $\tau\colon F\dasharrow \p^3$ induced by $\lvert -\frac{1}{2}K_{\hat{F}}-\frac{1}{2}E\rvert$ is then given by hypersurfaces of tridegree $(1,1,1)$ through $C$ and thus given by
\[ ([x_0:x_1],[y_0:y_1],[z_0:z_1])\mapsto [y_0 ( x_0z_1-x_1z_0 ):y_1 ( x_0z_1-x_1z_0 ):z_0 ( 
x_0y_1-x_1y_0 ):z_1 ( x_0y_1-x_1y_0 )].\]
Its inverse $\tau^{-1}\colon  \p^3 \dasharrow\p^1\times \p^1\times \p^1$ is given by
\[[w:x:y:z]\mapsto ([w-y:x-z],[w:x],[y:z]).\]
We observe that $\tau^{-1}$ contracts the smooth quadric surface $S=\{wz-xy=0\}\subset \p^3$ onto the curve $C$, and that $\tau$ contracts respectively $H_1,H_2,H_3$ onto the three skew lines $\ell_1,\ell_2,\ell_3\subset S\subset \p^3$ given by $\ell_1=\{y=z=0\}$, $\ell_2=\{w-y=x-z=0\}$, $\ell_3=\{w=x=0\}$. Denote by $\kappa\colon X\to \p^3$ the blow-up of $\ell_1,\ell_2,\ell_3$. For $i\in \{1,2,3\}$, we denote by $\pi_i\colon \p^1\times \p^1\times \p^1\to \p^1$  the $i$-th projection, and observe that $\pi_i\circ \tau^{-1}\colon  \p^3\dasharrow \p^1$ is the linear projection away from the line $\ell_i$. Hence, $\pi_i\circ \tau^{-1}\circ\kappa \colon X\to \p^1$ is a morphism. This being true for the three projections, the birational map $\tau^{-1}\circ\kappa \colon X\to \p^1\times\p^1 \times \p^1$ is a morphism. 
This birational morphism between two smooth threefolds contracts the strict transform of $S$, isomorphic to $\p^1\times \p^1$, onto the curve $C\simeq \p^1$, and is thus the blow-up of $C$. This achieves the proof of~\ref{P1P1P1CBirMorph}.

To prove~\ref{P1P1P1CFano}, one can compute the cone of effective curves and prove that it is polyhedral, like in Lemma~\ref{Lem:PointpinP1P1P1linkII} and check that $-K_{\hat{F}}$ is ample. 
Equivalently, one can see that $\hat{F}$ appears in the classifcation of Fano threefolds (see \cite[n$\degree$6 of Table~4]{MoriMukai1}).
As for~\ref{P1P1P1Cbignef},
we  first observe that~\ref{P1P1P1CBirMorph} implies that  $D+K_{\hat{F}}$ is big and nef, and \ref{P1P1P1CFano} implies that $-K_{\hat{F}}$ is ample, so $D$ ample.
\end{proof}
\begin{remark}In Lemma~\ref{Lem:CurveinP1P1P1link}, note that $\tau\colon \p^1\times \p^1\times \p^1\dasharrow \p^3$ is $\PGL_2(\C)$-invariant, where $\begin{pmatrix} a & b \\ c & d\end{pmatrix}\in \PGL_2(\C)$ acts on $\p^1\times \p^1\times \p^1$ and $\p^3$ as
$([x_0:x_1],[y_0:y_1],[z_0:z_1])\mapsto ([ax_0+bx_1:cx_0+dx_1],[ay_0+by_1:cy_0+dy_1],[az_0+bz_1:cz_0+dz_1])$ and 
$[w:x:y:z]\mapsto [aw+bx:cw+dx:ay+bz:cy+dz]$.\end{remark}

\begin{lemma}\label{Lem:PointpinP1P1P1linkII}
Let $p$ be a point of $F=\p^1\times \p^1\times \p^1$ and let $\ell_1,\ell_2,\ell_3\subset F$ be the three curves of tridegree $(1,0,0)$, $(0,1,0)$, $(0,0,1)$  passing through $p$. Let $\eta_1\colon F_1\to F$ be the blow-up of $F$ at $p$ and let $\eta_2\colon F_2\to F_1$
be the blow-up at the strict transforms $\tilde{\ell}_1,\tilde{\ell}_2,\tilde{\ell}_3$ of  $\ell_1,\ell_2,\ell_3$.
Denoting by $E_i\subset F_i$ the exceptional divisor of $\eta_i$ and writing again $E_1\subset F_2$ for the strict transform of $E_1\subset F_1$, the following hold:

\begin{enumerate}
\item\label{P1P1P1Dnefbig}
The divisor $D=-\frac{3}{2}K_{F_2}-E_1-\frac{1}{2}E_2$ is big and nef.
\item\label{P1P1P1semiample}
The linear system $\lvert D+K_{F_2}\rvert$ gives a birational morphism $\tau_2\colon F_2\to \p^3$, which is the contraction of the exceptional divisors of $\eta_2$ and the strict transforms of the divisors of $F$ of tridegree $(1,0,0)$, $(0,1,0)$, $(0,0,1)$ through $p$. It is also the blow-up of three non-collinear points of $\p^3$ followed by the blow-up of the strict transforms of the three lines through two of them.
\item\label{P1P1P1flopcontraction}
The birational map $\tau_1=\tau_2\circ \eta_2^{-1}\colon F_1\dasharrow \p^3$ is obtained by the flop of the curves $\tilde{\ell}_1,\tilde{\ell}_2,\tilde{\ell}_3$ followed by the contraction of the strict transforms of the divisors of $F$ of tridegree $(1,0,0)$, $(0,1,0)$, $(0,0,1)$ through $p$. 
\item\label{P1P1P1p1Ample}
The divisor $A_1=-(\eta_1)^*K_F-E_1$  is an ample divisor of $F_1$ and the union of the curves of $F_1$ having intersection $1$ with $A_1$ is $E_1\cup\tilde\ell_1\cup\tilde\ell_2\cup\tilde\ell_3$.
\end{enumerate}
\end{lemma}
\begin{proof}
We denote by $H_1,H_2,H_3\subset F$ the divisors of tridegree $(1,0,0)$, $(0,1,0)$, $(0,0,1)$ respectively through $p$. This gives, for all $i,j,k$ with $\{1,2,3\}=\{i,j,k\}$, that $\ell_k=H_i\cap H_j$, $H_i\cdot \ell_j=0$  and $H_i\cdot \ell_i=1$.  The cone of curves of $F$ is then generated by $\ell_1,\ell_2,\ell_3$, and one has $-K_{F}=2H_1+2H_2+2H_3$, so $-K_{F}\cdot \ell_i=2$ for each $i\in \{1,2,3\}$.

We denote by $e_1\subset E_1\subset F_1$ a line in $E_1\simeq \p^2$, by $\tilde{\ell}_i$ and $H_i$ the strict transforms of  $\ell_i$ and $H_i$ on $F_1$, giving $(\eta_1)^*(H_i)=H_i+E_1$, for $i=1,2,3$. For all $i,j\in \{1,2,3\}$ with $i\not=j$, one finds the following intersection numbers:
\[\begin{array}{|c|c|c|c|c|}
\hline
&  E_1 & H_i &H_j\\
\hline
e_1 & -1 &1 & 1 \\
\tilde{\ell}_i& 1& 0& -1 \\ \hline
\end{array}\]
 This implies that the cone of curves of $F_1$ is polyhedral, generated by $\tilde{\ell}_1,\tilde{\ell}_2,\tilde{\ell}_3,e_1$. Indeed, each irreducible curve $C$ of $F_1$ is either contained in $E_1\simeq \p^2$ and thus equivalent to a positive multiple of $e_1$, or is the strict transform of a curve of $F$, so equal to $\sum a_i \tilde{\ell}_i+be_1$ with $b\in \Z$ and $a_1,a_2,a_3\ge 0$, $a_1+a_2+a_3\ge 1$. If $C$ is not equal to $\tilde{\ell}_1$, $\tilde{\ell}_2$ or $\tilde{\ell}_3$, then it is not contained in $H_i$ and $H_j$ for two distinct $i,j\in \{1,2,3\}$. Choosing $k$ with $\{i,j,k\}=\{1,2,3\}$ we obtain $0\le H_i\cdot C=b-a_j-a_k$ and  $0\le H_j\cdot C=b-a_i-a_k$, which implies that $b\ge 1$. We moreover obtain $K_{F_1}=(\eta_1)^*(K_{F})+2E_1=-2\sum_{i=1}^3 H_i-4E_1.$
 
 We now use this to prove \ref{P1P1P1p1Ample}. Firstly, the divisor $A_1=-(\eta_1)^*K_F-E_1=2\sum_{i=1}^3 H_i+5 E_1$
 is ample as $A_1\cdot \tilde{\ell}_1=A_1\cdot \tilde{\ell}_2=A_1\cdot \tilde{\ell}_3=A_1\cdot e_1=1$. As every line in $E_1\simeq \p^2$ is equivalent to $e_1$, its intersection with $A_1$ is $1$. The union of curves having intersection $1$ with $A_1$  thus contains $E_1\cup\tilde\ell_1\cup\tilde\ell_2\cup\tilde\ell_3$. Conversely, an irreducible curve $C\subset F_1$ not contained in $E_1$ is numerically equivalent to $\sum a_i \tilde{\ell}_i+be_1$ with $a_1,a_2,a_3,b\ge 0$ and $a_1+a_2+a_3\ge 1$. Moreover, if it is not equal to $ \tilde{\ell}_1$, $ \tilde{\ell}_2$ or $\tilde{\ell}_3$, then $b\ge 1$, as we observed before, so $C\cdot A_1=a_1+a_2+a_3+b\ge 2$.  This achieves the proof of  \ref{P1P1P1p1Ample}.

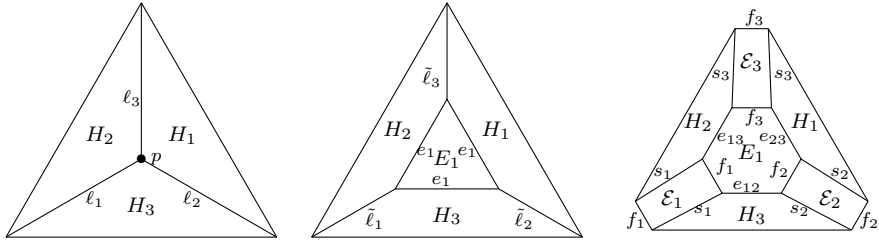
\begin{figure}[ht]
\begin{align*}\begin{tikzpicture}[scale=3.6,font=\footnotesize]
\coordinate (A) at (0:0) {};
\coordinate (B) at (0:1) {};
\coordinate (C) at (60:1) {};
\coordinate (H) at (barycentric cs:A=1,B=1,C=1) {};
\coordinate (H1) at (barycentric cs:A=0.15,B=0.5,C=0.5) {};
\coordinate (H2) at (barycentric cs:A=0.5,B=0.15,C=0.5) {};
\coordinate (H3) at (barycentric cs:A=0.5,B=0.5,C=0.15) {};
\draw (A) to  (C)
      (C) to  (B)
      (B) to (A);
\draw (A) to ["\scriptsize $\ell_1$",swap,yshift =1.2mm,xshift=0.5] (H)
      (B) to ["\scriptsize $\ell_2$",yshift =1.2mm,xshift=0.5]  (H)
      (H) to ["\scriptsize $\ell_3$",yshift =0mm,xshift=2] (C);
\node at (H1) {$H_1$};
\node at (H2) {$H_2$};
\node at (H3) {$H_3$};
\node at (H) {$\bullet$};
\node[right] at (H) {\scriptsize $p$};
\end{tikzpicture} 
&&
 \begin{tikzpicture}[scale=3.6,font=\footnotesize]
\coordinate (A) at (0:0) {};
\coordinate (B) at (0:1) {};
\coordinate (C) at (60:1) {};
\coordinate (R) at (barycentric cs:A=1,B=1,C=1) {};
\coordinate (E2) at (barycentric cs:A=0.35,B=1,C=0.35) {};
\coordinate (E3) at (barycentric cs:A=0.35,B=0.35,C=1) {};
\coordinate (E1) at (barycentric cs:A=1,B=0.35,C=0.35) {};
\coordinate (H1) at (barycentric cs:A=0.1,B=0.5,C=0.5) {};
\coordinate (H2) at (barycentric cs:A=0.5,B=0.1,C=0.5) {};
\coordinate (H3) at (barycentric cs:A=0.5,B=0.5,C=0.1) {};
\draw (A) to  (C)
      (C) to  (B)
      (B) to (A);
\draw (A) to ["\scriptsize $\tilde{\ell}_1$",swap,yshift =1mm,xshift=0.5] (E1)
      (B) to ["\scriptsize $\tilde{\ell}_2$",yshift =1mm,xshift=0.5]  (E2)
      (E3) to ["\scriptsize $\tilde{\ell}_3$",yshift =-0.5mm,xshift=0.75] (C)
      (E1) to ["\scriptsize $e_1$",yshift =-0.5mm,xshift=3] (E2)
      (E2) to ["\scriptsize $e_1$",yshift =0.5mm,xshift=3] (E3)
      (E3) to ["\scriptsize $e_1$",yshift =0.5mm,xshift=-3] (E1);
\node at (R) {$E_1$};
\node at (H1) {$H_1$};
\node at (H2) {$H_2$};
\node at (H3) {$H_3$};
\end{tikzpicture}
&&
 \begin{tikzpicture}[scale=4.2,font=\footnotesize]
\coordinate (A) at (0:0) {};
\coordinate (B) at (0:1) {};
\coordinate (C) at (60:1) {};
\coordinate (R) at (barycentric cs:A=1,B=1,C=1) {};
\coordinate (Z2) at (barycentric cs:A=0.6,B=3,C=0.2) {};
\coordinate (Z4) at (barycentric cs:A=0.2,B=3,C=0.6) {};
\coordinate (Z3) at (barycentric cs:A=0.2,B=0.6,C=3) {};
\coordinate (Z5) at (barycentric cs:A=0.6,B=0.2,C=3) {};
\coordinate (Z1) at (barycentric cs:A=3,B=0.6,C=0.2) {};
\coordinate (Z6) at (barycentric cs:A=3,B=0.2,C=0.6) {};
\coordinate (E2) at (barycentric cs:A=0.5,B=0.8,C=0.3) {};
\coordinate (E4) at (barycentric cs:A=0.3,B=0.8,C=0.5) {};
\coordinate (E3) at (barycentric cs:A=0.3,B=0.5,C=0.8) {};
\coordinate (E5) at (barycentric cs:A=0.5,B=0.3,C=0.8) {};
\coordinate (E1) at (barycentric cs:A=0.8,B=0.5,C=0.3) {};
\coordinate (E6) at (barycentric cs:A=0.8,B=0.3,C=0.5) {};
\coordinate (H1) at (barycentric cs:A=0.12,B=0.5,C=0.5) {};
\coordinate (H2) at (barycentric cs:A=0.5,B=0.12,C=0.5) {};
\coordinate (H3) at (barycentric cs:A=0.5,B=0.5,C=0.12) {};
\coordinate (S1) at (barycentric cs:A=2,B=0.5,C=0.5) {};
\coordinate (S2) at (barycentric cs:A=0.5,B=2,C=0.5) {};
\coordinate (S3) at (barycentric cs:A=0.5,B=0.5,C=2) {};
\draw (Z6) to  (Z5)
      (Z3) to  (Z4)
      (Z2) to (Z1);
\draw (E1) to ["\scriptsize $s_1$",yshift =1mm,xshift=-0.5] (Z1)
      (E6) to ["\scriptsize $s_1$",swap,yshift =-0.5mm,xshift=2.5] (Z6)
      (E2) to ["\scriptsize $s_2$",swap,yshift =1mm,xshift=0.5] (Z2)
      (E4) to ["\scriptsize $s_2$",yshift =-0.5mm,xshift=-2.5] (Z4)
      (E3) to ["\scriptsize $s_3$",swap,yshift =0.5mm,xshift=-1] (Z3)
      (E5) to ["\scriptsize $s_3$",yshift =0.5mm,xshift=1.3] (Z5)
      (E1) to ["\scriptsize $e_{12}$",yshift =-0.5mm,xshift=3.5] (E2)
      (E3) to ["\scriptsize $e_{23}$",swap,yshift =0.6mm,xshift=2] (E4)
      (E5) to ["\scriptsize $e_{13}$",yshift =0.6mm,xshift=-1.8] (E6)
      (E1) to ["\scriptsize $f_1$",swap,yshift =-0.5mm,xshift=-1.1] (E6)
      (Z1) to ["\scriptsize $f_1$",yshift =0.8mm,xshift=2.2] (Z6)
      (E2) to ["\scriptsize $f_2$",yshift =-0.7mm,xshift=1.4] (E4)
      (Z2) to ["\scriptsize $f_2$",swap,yshift =0.8mm,xshift=-1.8] (Z4)
      (Z3) to ["\scriptsize $f_3$",yshift =0.7mm,xshift=-3] (Z5)
      (E3) to ["\scriptsize $f_3$",yshift =0.5mm,xshift=-3] (E5);
\node at (R) {$E_1$};
\node at (H1) {$H_1$};
\node at (H2) {$H_2$};
\node at (H3) {$H_3$};
\node at (S1) {$\mathcal{E}_1$};
\node at (S2) {$\mathcal{E}_2$};
\node at (S3) {$\mathcal{E}_3$};
\end{tikzpicture}
\end{align*}
\caption{The divisors and curves on $F$, $F_1$ and $F_2$ respectively.} \label{FF1F2P1P1P1}
\end{figure}
The morphism $\eta_2\colon F_2\to F_1$ being the blow-up of $\tilde{\ell}_1,\tilde{\ell}_2,\tilde{\ell}_3$, 
we denote by $\mathcal{E}_i$ the irreducible component of $E_2$ lying over $\tilde{\ell}_i$, for $i=1,2,3$. 
The divisor $\mathcal{E}_i$ is isomorphic to $\p^1\times \p^1$ as it is isomorphic to the exceptional divisor of the  blow-up of $F$ along $\ell_i$. 
We write $s_i\subset \mathcal{E}_i$ for a section of $\mathcal{E}_i\to \tilde{\ell}_i$ of self-intersection $0$ and $f_i\subset \mathcal{E}_i$ a fibre of $\eta_2$; 
in particular $f_i\cdot \mathcal{E}_i=-1$.  We then denote again by $E_1,H_1,H_2,H_3$ the strict transforms of the same surfaces on $F_2$ and obtain $\eta_2^*(E_1)=E_1$ 
and $\eta_2^*(H_1)=H_1+\mathcal{E}_2+\mathcal{E}_3$.

Since $s_2\cdot \eta_2^*(H_1)=(\eta_2)_*(s_2)\cdot H_1=\tilde{\ell}_2\cdot H_1=-1$ and $s_2\cdot H_1=s_2\cdot \mathcal{E}_3=0$, we find $s_2\cdot \mathcal{E}_2=-1$. Similarly, one obtains $s_i\cdot \mathcal{E}_i=-1$ for each $i\in \{1,2,3\}$.

For all distinct $i,j\in \{1,2,3\}$, we also denote by $e_{ij}\subset F_2$ the strict transform of the line of $E_1\simeq \p^2$ 
that intersects $\tilde{\ell}_i$ and $\tilde{\ell}_j$ (see Figure~\ref{FF1F2P1P1P1}). 
This gives  
$E_1\cdot e_{ij}=\eta_2^*(E_1)\cdot e_{ij}=E_1\cdot (\eta_2)_*(e_{ij})=E_1\cdot e_1=-1$. We similarly obtain $E_1\cdot f_i=E_1\cdot f_j=0$. 

For all $i,j,k$ with $\{i,j,k\}=\{1,2,3\}$, one finds the following intersection numbers: 
\[\begin{array}{|c|c|c|c|c|c|}
\hline
& H_i & H_k& E_1 & \mathcal{E}_i &  \mathcal{E}_k\\
\hline
e_{ij}& 0 & -1&-1& 1 & 0\\
s_i&0&0&1&-1 & 0 \\
f_i& 0 & 1&0&-1 &  0\\
\hline
\end{array}\]
The fact that $e_{ij}\cdot H_k=-1$ can be computed as follows:   $1=e_1\cdot H_1=(\eta_2)_*e_{23}\cdot H_1=e_{23}\cdot \eta_2^*(H_1)=e_{23}\cdot (H_1+\mathcal{E}_2+\mathcal{E}_3)=e_{23}\cdot H_1+2$.

We now prove that the cone of effective curves of $F_2$ is polyhedral and generated by $s_1,s_2,s_3,f_1,f_2,f_3,e_{12},e_{13},e_{23}$, by proving that every irreducible curve $C$ is a non-negative linear combination of these.  If $C$ is contained in one of the surfaces $E_1,\mathcal{E}_i,H_i$, $i\in \{1,2,3\}$ this is true as our curves include all extremal rays of these del Pezzo surfaces. We may thus assume that $C$ is not contained in $E_1$ or in any of the $\mathcal{E}_i$ or $H_i$. The curve $C$ is the strict transform of $\eta_2(C)\equiv \sum_{i=1}^3 a_i \tilde{\ell}_i+be_1$ for some $a_1,a_2,a_3,b\ge 0$. This gives $C\equiv \sum_{i=1}^3 a_i s_i+be_{23}+\sum_{i=1}^3 c_i f_i$ for some $c_1,c_2,c_3\in \Z$. As $0\le H_1\cdot C=c_1-b$ we find that $c_1\ge 0$. Similarly, $0\le H_i\cdot C=c_i$ for $i=2,3$, proving the statement on the Mori cone of $F_2$.

We have $E_2=\sum_{i=1}^3\mathcal{E}_i$ and $D=-\frac{3}{2}K_{F_2}-E_1-\frac{1}{2}E_2$. Writing $\eta=\eta_1\circ\eta_2$, we get
\[\begin{array}{ccccc}
K_{F_2}&=&\eta^*(K_F)+2E_1+E_2&=&-2\sum\limits_{i=1}^3 H_i-4E_1-3E_2.\\
D&=&-\frac{3}{2}\eta^*(K_F)-4E_1-2E_2&=&3\sum\limits_{i=1}^3 H_i+5E_1+4E_2
\end{array}\] This  implies that \[D\cdot e_{ij}=0, D\cdot s_i=1\text{ and }D\cdot f_i=2.\]
 for all distinct $i,j\in \{1,2,3\}$.

The divisor $D$ is thus nef. As $-\frac{1}{2}K_F$ is big (it is very ample), the divisor $-\frac{1}{2}\eta^*(K_F)=\sum\limits_{i=1}^3 H_i+3E_1+2E_2$ is big. Since $D+\frac{1}{2}\eta^*(K_F)=2\sum\limits_{i=1}^3 H_i+2E_1+2E_2$ is effective, this implies that $D$ is big. This achieves the proof of ~\ref{P1P1P1Dnefbig}.

 To prove~\ref{P1P1P1semiample}, we change coordinates, assume that $p=([0:1],[0:1],[0:1])$ and take coordinates $([x_0:x_1],[y_0:y_1],[z_0:z_1])$ on $F$. This implies that $H_1$, $H_2$, $H_3$ are respectively given by $x_0=0$, $y_0=0$ and $z_0=0$.
We consider the linear system $\lvert D+K_{F_2}\rvert$ on $F_2$. Since 
 \[D+K_{F_2}=\sum\limits_{i=1}^3 H_i+E_1+E_2=-\frac{1}{2}\eta^*(K_F)-2E_1-E_2,\]
 it corresponds to the strict transform of hyperplane sections of $F=\p^1\times \p^1\times\p^1$ of tridegree $(1,1,1)$ having multiplicity $2$ at $p$ and passing through the three curves $\ell_1,\ell_2,\ell_3$. This last condition is in fact implied by the first. The linear system corresponds then to the toric birational map $\tau\colon \p^1\times \p^1\times \p^1\dasharrow \p^3$ given by 
\[ \tau\colon ([x_0:x_1],[y_0:y_1],[z_0:z_1])\mapsto [x_0y_0z_0:x_1y_0z_0:x_0y_1z_0:x_0y_0z_1],\]
whose inverse is given by
\[\tau^{-1}\colon [w:x:y:z]\mapsto ([w:x],[w:y],[w:z]),\]
and which restricts to an isomorphism $F\setminus (H_1\cup H_2\cup H_3)\iso \p^3 \setminus H_w$, where $H_w\subset \p^3$ is the hyperplane given by $w=0$.
Hence, the linear system $\lvert D+K_{F_2}\rvert$ yields the toric birational map $\tau_2=\tau\circ \eta\colon F_2\dasharrow \p^3$. 

We prove now that $\tau_2$ is a morphism, i.e.~that it is defined at every point $q\in F$. As $\tau$ is defined outside of $\ell_1\cup\ell_2\cup\ell_3$, we may assume that $\eta(q)\in \ell_1\cup\ell_2\cup\ell_3$. Using the action of $\mathrm{Sym}_3$ on $x,y,z$, we may assume that $\eta(q)\in \ell_1$. If $\eta(q)\not=p$, then $\eta(q)$ belongs to the image of the open embedding $\A^3\hookrightarrow F, (r,s,t)\mapsto ([1:r],[s:1],[t:1])$. The morphism $\eta$ corresponds on this chart to the blow-up of $s=t=0$, that is $\{((r,s,t),[u:v])\in \A^3\times \p^1\mid sv=tu\}$. Hence, $\tau_2$ is locally given by $((r,s,t),[u:v])\mapsto [st:rst:t:s]=[sv:rsv:v:u]=[tu:rtu:v:u]$ and is  then well-defined at every point. The remaining case is where $\eta(q)=p$, so $\eta_2(q)$ belongs to the surface $E_1\subset F_1$ isomorphic to $\p^2$. We study $\tau_1=\tau\circ \eta_1\colon F_1\dasharrow \p^3$ in a neighbourhood of $E_1$. For this, we take the open embedding $\A^3\hookrightarrow F, (r,s,t)\mapsto ([r:1],[s:1],[t:1])$, and obtain that $\eta_1$ is the blow-up of the origin of $\A^3$ in this chart, corresponding to $\{((r,s,t),[u:v:w])\in \A^3\times \p^2\mid su=rv,sw=tv,rw=tu\}$. The rational map $\eta_1$ is locally given by $((r,s,t),[u:v:w])\mapsto [rst:st:rt:rs]=[rvw:vw:uw:uv]$. The divisor $E_1$ corresponds to $(0,0,0)\times \p^2$, so $\eta_1$ is defined at every point of $E_1$ except the three toric points. These are exactly the points where $(\eta_2)^{-1}$ is not an isomorphism. Using the symmetry, we may assume that $\eta_2(q)=E_1\cap \tilde{\ell}_1$, corresponding to $((0,0,0),[0:0:1])\in \A^3\times \p^2$. We choose the open embedding $\A^3\hookrightarrow \A^3\times \p^2$, $(a,b,c)\mapsto ((ac,bc,c),[a:b:1])$, and see $\eta_2$ as the blow-up of the line $a=b=0$. This latter is given by $\{((a,b,c),[\alpha:\beta])\in \A^3\times \p^1\mid a\beta=b\alpha\}$ and  $\tau_2$ is given by $((a,b,c),[\alpha:\beta])\mapsto [abc:b:a:ab]=[a\beta c:\beta:\alpha :a\beta ],$ and is thus defined at every point.

We have now proven that $\tau_2=\lvert D+\frac{1}{2}K_F\rvert\colon F_2\to \p^3$ is a birational morphism. Since $\tau(H_1\setminus (\ell_2\cup \ell_3))=[0:1:0:0]$, the morphism $\tau_2$ contracts the surface $H_2\subset F_2$ onto $[0:1:0:0]$. Similarly, the surfaces $H_1,H_3\subset F_2$ are contracted onto $[0:0:1:0]$ and $[0:0:0:1]$. The above description of $\tau_2$ implies also that $\mathcal{E}_1$ is contracted to the curve $w=x=0$, so $\mathcal{E}_2$ and $\mathcal{E}_3$ are contracted onto $w=y=0$ and $w=z=0$. One can then either check in coordinates or use the universal property of blowing-ups, to see that $\tau_2$ is exactly the blow-up of $[0:1:0:0]$, $[0:0:1:0]$ and $[0:0:0:1]$, followed by the blow-up of the strict transform of the three lines through these points. This achieves the proof of \ref{P1P1P1semiample}.

It remains to prove~\ref{P1P1P1flopcontraction}. We have already shown that $\tau_1$ is obtained by blowing-up  the curves $\tilde{\ell}_1,\tilde{\ell}_2,\tilde{\ell}_3$, then contracting their divisors $\mathcal{E}_1,\mathcal{E}_2,\mathcal{E}_3\simeq \p^1\times \p^1$ ``in the other direction'' and then contracting the strict transforms of the divisors of $F$ of tridegree $(1,0,0)$, $(0,1,0)$, $(0,0,1)$ through $p$ onto $[0:1:0:0]$, $[0:0:1:0]$ and $[0:0:0:1]$. As  $\tilde{\ell}_1,\tilde{\ell}_2,\tilde{\ell}_3\subset F_1$ are extremal and have intersection $0$ with the canonical, the blow-up of them followed by the contracting of the divisors ``in the other direction''  simply consists of three Atiyah flops.
\end{proof}

\subsection{Symmetric birational maps from $\p(T_{ \p^2})$}
We now describe symmetric birational maps from a smooth hypersurface of $\p^2\times \p^2$ of bidegree $(1,1)$ $($Case~\hyperlink{T4}{$4$} of Table~$\ref{3folds.MFS})$, which is isomorphic to $\p(T_{ \p^2})$ (Lemma~\ref{T4Aut}).

The following lemma is similar to Lemma~\ref{Lem:CurveinP1P1P1link}.
\begin{lemma}\label{Lem:LineinPTP2link}
Let $C$ be a curve of bidegree $(1,1)$ in  $\p^2\times \p^2$ contained in a smooth hypersurface $F$ of bidegree $(1,1)$. Let $\eta\colon \hat{F}\to F$ be the blow-up of $F$ at $C$, with exceptional divisor $E$. 
\begin{enumerate}
\item\label{PTP2LFano}
The threefold $\hat{F}$ is a smooth Fano threefold of Picard rank $3$.
\item\label{PTP2Lbignef}
The divisor $D=-\frac{3}{2}K_{\hat{F}}-\frac{1}{2}E$ is ample.
\item\label{PTP2LBirMorph}
The linear system $\lvert D+K_{\hat{F}}\rvert$ gives a birational morphism $\hat{F}\to Q$, where $Q$ is a smooth quadric in $\p^4$, which is  the contraction of the strict transforms of divisors of $F$ of bidegree $(0,1)$, $(1,0)$ through $C$, or equivalently the blow-up of two skew lines of $Q$.\end{enumerate}
\end{lemma}
\begin{proof}
We have $K_F=-2H$, where $H\subset F$ is the intersection of $F$ with a hypersurface of $\p^2\times \p^2$ of bidegree $(1,1)$.
Since $K_{\hat{F}}=\eta^{*}(K_F)+E$, we obtain $D+K_{\hat{F}}=-\frac{1}{2}K_{\hat{F}}-\frac{1}{2}E=-\frac{1}{2}\eta^{*}(K_F)-E=\eta^{*}(H)-E$.
Changing coordinates, we may assume that \[F=\left\{([x_0:x_1:x_2],[y_0:y_1:y_2])\in \p^2\times \p^2 \left| \textstyle \sum\limits_{i=0}^2 x_i y_i=0\right\}\right..\]
 (Lemma~\ref{T4Aut}). We may then apply an element of $\PGL_3(\C)$ as in Lemma~\ref{T4Aut}\ref{PTP2PGL3} and assume that the projection of $C$ onto the first coordinate is given by $x_0=0$, so $C$ is given by 
\[C=\{([0:u:v],[\alpha u+\beta v:-v:u])\mid [u:v]\in \p^1\}\]
for some $\alpha,\beta\in \C$. Applying an automorphism of the form $([x_0:x_1:x_2],[y_0:y_1:y_2])\mapsto ([x_0:x_1-\beta x_0:x_2+\alpha x_0],[y_0+\beta y_1-\alpha y_2:y_1:y_2])$, we may assume that $\alpha=\beta=0$.
 
The divisors $H_1,H_2\subseteq F$  of bidegree $(0,1)$, $(1,0)$ through $C$ are then given respectively by 
\[H_1=\{y_0=0\}, H_2=\{x_0=0\}.\]

The linear system $\lvert -\frac{1}{2}K_{\hat{F}}-\frac{1}{2}E\rvert$ is the linear system of strict transforms of hypersurfaces of bidegree $(1,1)$ through $C$ and thus the rational map $\tau\colon F\dasharrow Q\subset \p^4$ induced by it is given by
\[ ([x_0:x_1:x_2],[y_0:y_1:y_2])\mapsto [x_0y_0:x_0y_1:x_1y_0:x_0y_2:x_2y_0].\]
Its image is $Q=\{[z_0:\cdots:z_4]\in \p^4\mid z_0^2+z_1z_2+z_3z_4=0\}$. The inverse $\tau^{-1}\colon Q \dasharrow F$ is given by
\[[z_0:\cdots:z_4]\mapsto ([z_0:z_2:z_4],[z_0:z_1:z_3]).\]
We observe that $\tau^{-1}$ contracts the smooth quadric surface $S=\{z_0=0, z_1z_2+z_3z_4=0\}\subset Q\subset\p^4$ onto the curve $C$, and that $\tau$ contracts respectively $H_1,H_2$ onto the two skew lines $\ell_1,\ell_2\subset S$ given by $\ell_1=\{z_0=z_2=z_4=0\}$ and $\ell_2=\{z_0=z_1=z_3=0\}$. Denote by $\kappa\colon X\to Q$ the blow-up of $\ell_1,\ell_2$. For $i\in \{1,2\}$, we denote by $\pi_i\colon \p^2\times \p^2\times \p^2$  the $i$-th projection, and observe that $\pi_i\circ \tau^{-1}\colon  \p^4\dasharrow \p^2$ is the linear projection away from the line $\ell_i$. Hence, $\pi_i\circ \tau^{-1}\circ\kappa \colon X\to \p^2$ is a morphism. This being true for the two projections, the birational map $\tau^{-1}\circ\kappa \colon X\to F$ is a morphism. 
This birational morphism between two smooth threefolds contracts the strict transform of $S$, isomorphic to $\p^1\times \p^1$ onto the curve $C\simeq \p^1$, and is thus the blow-up of $C$. This achieves the proof of~\ref{PTP2LBirMorph}.

To prove~\ref{PTP2LFano}, one can compute the cone of effective curves, prove that it is polyhedral like in Lemma~\ref{Lem:PointpinP1P1P1linkII} and check that $-K_{\hat{F}}$ is ample. Equivalently, $\hat{F}$ is a smooth Fano threefold appearing in the Mori-Mukai classification (see \cite[n$\degree$20 of Table~3]{MoriMukai1}). The proof of~\ref{PTP2Lbignef} can be done as follows: we  first observe that~\ref{PTP2LBirMorph} implies that  $D+K_{\hat{F}}$ is big and nef, and \ref{PTP2LFano} implies that $-K_{\hat{F}}$ is ample, so $D$ is ample.
\end{proof}

The following result, and its proof, are very similar to Lemma~\ref{Lem:PointpinP1P1P1linkII}. The main difference is the morphism induced by $\lvert D+K_{F_2}\rvert$: it is birational with image $\p^3$ in the case of $(\p^1)^3$ and it is not birational in the case of $F\subset \p^2\times \p^2$, as it is the restriction of $\p^2\times \p^2\dasharrow \p^1\times\p^1$ where each factor is a projection from a point.
\begin{lemma}\label{Lem:PointpinPTP2linkII}
Let $p=(p_1,p_2)$ be a point of $\p^2\times \p^2$,  contained in a smooth hypersurface $F$ of bidegree $(1,1)$. Let $\ell_1,\ell_2\subset F$ be the two curves of bidegree $(1,0)$, $(0,1)$  that pass through $p$. Let $\eta_1\colon F_1\to F$ be the blow-up of $F$ at $p$ and let $\eta_2\colon F_2\to F_1$
be the blow-up at the strict transforms $\tilde{\ell}_1,\tilde{\ell}_2$ of  $\ell_1,\ell_2$.
Denoting by $E_i\subset F_i$ the exceptional divisor of $\eta_i$ and writing again $E_1\subset F_2$ for the strict transform of $E_1\subset F_1$, the following hold:
\begin{enumerate}
\item\label{PTP2Dnefbig}
The divisor $D=-\frac{3}{2}K_{F_2}-E_1-\frac{1}{2}E_2$ is big and nef.
\item\label{PTP2semiample}
The linear system $\lvert D+K_{F_2}\rvert$ gives a morphism $\tau_2\colon F_2\to \p^1\times \p^1\subset \p^3$, with general fibres isomorphic to $\p^1$, 
corresponding to projections $\p^2\dasharrow \p^1$ away from $p_i$ on the two factors of $\p^2\times \p^2$.
\item\label{PTP2p1Ample}
The divisor $A_1=-(\eta_1)^*K_F-E_1$  is an ample divisor of $F_1$ and the union of the curves of $F_1$ having intersection $1$ with $A_1$ is $E_1\cup\tilde\ell_1\cup\tilde\ell_2\cup\tilde\ell_3$.
\end{enumerate}
\end{lemma}
\begin{proof}
Changing coordinates, we may assume that \[F=\left\{([x_0:x_1:x_2],[y_0:y_1:y_2])\in \p^2\times \p^2 \left| \textstyle \sum_{i=0}^2 x_i y_i=0\right\}\right..\]
 (Lemma~\ref{T4Aut}). We may then apply an element of $\PGL_3(\C)$ as in Lemma~\ref{T4Aut}\ref{PTP2PGL3} and assume that $p_1=[1:0:0]$ and $p_2=[0:1:0]$. This gives \begin{align*}\ell_1&=\{([u:0:v],[0:1:0]),[u:v]\in \p^1\},\\
 \ell_2&=\{([1:0:0],[0:u:v]),[a:b]\in \p^1\}.\end{align*}
 We denote by $H_1,H_2\subset F$ the divisors given by $x_1=0,y_0=0$ respectively.
We observe that the involution $\sigma\in \Aut(F)$ given by $([x_0:x_1:x_2],[y_0:y_1:y_2])\mapsto( [y_1:y_0:y_2],[x_1:x_0:x_2])$ exchanges $\ell_1$ and $\ell_2$, 
 exchanges $H_1$ and $H_2$ and fixes $p$.
 
The divisors $H_1$ and $H_2$ generate the cone of effective divisors of $F$ and satisfy $H_1\cap H_2=\ell_1\cup \ell_2$. For $\{i,j\}=\{1,2\}$, the  surface $H_i$ is isomorphic to $\F_1$, the projection on the $i$-th factor gives a $\p^1$-bundle on a line of $\p^2$, the curve $\ell_j\subset H_1$ being a fibre and the projection on the $j$-th factor gives the contraction $\F_1\to \p^2$ of the $(-1)$-curve $\ell_i\subset H_i$.  We also have $H_i\cdot \ell_j=0$ and $H_i\cdot \ell_i=1$ when $\{i,j\}=\{1,2\}$. The cone of curves of $F$ is then generated by $\ell_1,\ell_2$, and one has $-K_{F}=2H_1+2H_2$, so $-K_{F}\cdot \ell_i=2$ for each $i\in \{1,2\}$.

We denote by $e_1\subset E_1\subset F_1$ a line in $E_1\simeq \p^2$, by $\tilde{\ell}_i$ and $H_i$ the strict transforms in $F_1$ of  $\ell_i$ and $H_i$, giving $(\eta_1)^*(H_i)=H_i+E_1$, for $i=1,2$. For  $\{i,j\}=\{1,2\}$, one finds the following intersection numbers:
\[\begin{array}{|c|c|c|c|c|}
\hline
&  e_1 & \tilde{\ell}_i&\tilde{\ell}_j\vphantom{\Big|} \\
\hline
E_1 & -1 &1&1  \\
H_i& 1& 0 &-1\\ \hline
\end{array}\]

 This implies that the cone of curves on $F_1$ is polyhedral, generated by $\tilde{\ell}_1,\tilde{\ell}_2,e_1$. 
 Indeed, each irreducible curve $C$ of $F_1$ is either contained in $E_1\simeq \p^2$ and thus equivalent to a positive multiple of $e_1$, 
 or is the strict transform of a curve of $F$, so equal to $\sum a_i \tilde{\ell}_i+be_1$ with $b\in \Z$ and $a_1,a_2\ge 0$. 
 There is moreover $i\in \{1,2\}$ such that $a_i\ge 1$. If $C$ is not equal to $\tilde{\ell}_i$, then it is not contained in $H_i$, 
 for some right choice of $H_i$ in its pencil.
 We obtain $0\le H_i\cdot C=b-a_i$, which implies that $b\ge 1$. Moreover, we obtain $K_{F_1}=(\eta_1)^*(K_{F})+2E_1=-2H_1-2H_2-2E_1.$

\begin{figure}[ht]
\begin{tikzpicture}
\begin{axis}[hide x axis,hide y axis, hide z axis]
\PLine{(1.1,0,0) (1.1,0.7,0)};
\PLine{(-1.1,0,0) (-1.1,1,0)};
\Polygone{(0,0,1) (1,0,0) (1,0.2,-1) (-1,0,0)}{4}{vlightgray};
\Polygone{(0,0,1) (1,0,0) (-1,0,-1) (-1,0,0)}{4}{vlightgray};
\LLine{(1,0.2,-1) (-1,0,0) };
\node[inner sep=0.5pt] at (axis cs:-0.85,0,-0.7) {$H_1$};
\node[inner sep=0.5pt] at (axis cs:0.8,0.2,-0.7) {$H_2$};
\node[inner sep=0.5pt] at (axis cs:0,0,1) {$\bullet$};
\node[inner sep=0.5pt] at (axis cs:0,0,1.2) {$p$};
\node[inner sep=0.5pt] at (axis cs:-0.5,0,0.68) {$\ell_1$};
\node[inner sep=0.5pt] at (axis cs:0.5,0,0.68) {$\ell_2$};
\end{axis}
\end{tikzpicture}
 \begin{tikzpicture}
\begin{axis}[hide x axis,hide y axis, hide z axis]
\PLine{(1.1,0,0) (1.1,0.7,0)};
\PLine{(-1.1,0,0) (-1.1,1,0)};
\Polygone{(-1,0,0)(-0.3,0,0.6) (0.3,0,0.6) (1,0,0) (1,0.2,-1) }{5}{vlightgray};
\Polygone{(-1,0,0)(-0.3,0,0.6) (0.3,0,0.6) (1,0,0) (-1,0,-1) }{5}{vlightgray};
\Polygone{(-0.3,0,0.6) (0.3,0,0.6) (0,0,1.3) }{3}{gray};
\LLine{(1,0.2,-1) (-1,0,0) };
\node[inner sep=0.5pt] at (axis cs:-0.85,0,-0.7) {$H_1$};
\node[inner sep=0.5pt] at (axis cs:0.8,0.2,-0.7) {$H_2$};
\node[inner sep=0.5pt] at (axis cs:0,0,0.85) {$E_1$};
\node[inner sep=0.5pt] at (axis cs:-0.21,0,1) {$e_1$};
\node[inner sep=0.5pt] at (axis cs:0.2,0,1.1) {$e_1$};
\node[inner sep=0.5pt] at (axis cs:0,0,0.45) {$e_1$};
\node[inner sep=0.5pt] at (axis cs:-0.55,0,0.58) {$\tilde{\ell}_1$};
\node[inner sep=0.5pt] at (axis cs:0.55,0,0.58) {$\tilde{\ell}_2$};
\end{axis}
\end{tikzpicture}
\caption{The divisors and curves on $F$ and $F_1$ respectively.} \label{FF1PTP2}
\end{figure}
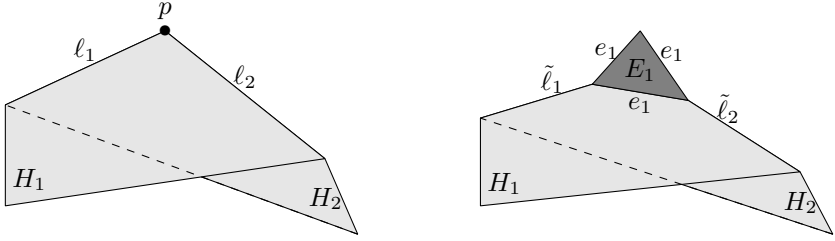

 We now use this to prove \ref{PTP2p1Ample}. Firstly, the divisor $A_1=-(\eta_1)^*K_F-E_1=2H_1+2H_2+3 E_1$
 is ample as $A_1\cdot \tilde{\ell}_1=A_1\cdot \tilde{\ell}_2=A_1\cdot e_1=1$. As every line in $E_1\simeq \p^2$ is equivalent to $e_1$, its intersection with $A_1$ is $1$. The union of curves having intersection $1$ with $A_1$  thus contains $E_1\cup\tilde\ell_1\cup\tilde\ell_2\cup\tilde\ell_3$. Conversely, an irreducible curve $C\subset F_1$ not contained in $E_1$ is numerically equivalent to $\sum a_i \tilde{\ell}_i+be_1$ with $a_1,a_2,b\ge 0$ and $a_1+a_2\ge 1$. Moreover, if it is not equal to $ \tilde{\ell}_1$ or $ \tilde{\ell}_2$, then $b\ge 1$, as we observed before, so $C\cdot A_1=a_1+a_2+b\ge 2$.  This achieves the proof of  \ref{PTP2p1Ample}.

The morphism $\eta_2\colon F_2\to F_1$ being the blow-up of $\tilde{\ell}_1,\tilde{\ell}_2$, 
we denote by $\mathcal{E}_i$ the irreducible component of $E_2$ lying over $\tilde{\ell}_i$, for $i=1,2$. We then denote again by $E_1,H_1,H_2$ the strict transforms of the same surfaces on $F_2$.
The morphism $\mathcal{E}_i\to \tilde{\ell}_i$ is a $\p^1$-bundle, and $H_1\cap \mathcal{E}_i$, $H_2\cap \mathcal{E}_i$ are two sections. 
We now prove that $\mathcal{E}_i$ is isomorphic to $\p^1\times \p^1$ and that $H_j\cap \mathcal{E}_i$ has self-intersection $1$ or $0$, if $j=i$ or $j\not=i$ respectively. 
To see this, we can look at the blow-up of $\ell_i\subset F$. Using the symmetry, we assume $i=1$; this allows to work on the open subset $U\subset F$ where $y_0=1$, 
isomorphic to $\p^1\times \A^2$, via $([u:v],(a,b))\mapsto ([u:-au-bv:v],[a:1:b])$. The curve $\ell_1$ is given in this chart by $a=b=0$, 
so the exceptional divisor is isomorphic to $\p^1\times \p^1$. The surface $H_1$ and $H_2$ are given by $au+bv=0$ and $a=0$ and thus their strict transform intersect 
the exceptional divisor along sections of self-intersection $1$ and $0$ respectively.

\begin{figure}[ht]
 \begin{tikzpicture}
\begin{axis}[view={20}{25},hide x axis,hide y axis, hide z axis]
\PLine{(1.1,0,0) (1.1,0.7,0)};
\PLine{(-1.1,0,0) (-1.1,1,0)};
% H2
\Polygone{(-1,0,0.2)(-0.3,0,0.2) (0.3,0,0.2) (1,0,0.8) (1,0.2,-1)}{5}{vlightgray};
\Polygone{(-0.3,0,0.2)(-0.3,0,0.8)  (0,0,1.5) (0.3,0,0.8) (0.3,0,0.2)}{5}{gray};

\Polygone{(-1,0,0.2)(-0.3,0,0.2)(-0.3,0,0.8)(-1,0,0.8) }{4}{lightgray}; %\mathcal{E}_1
\Polygone{(1,0,0.2)(0.3,0,0.2)(0.3,0,0.8)(1,0,0.8) }{4}{lightgray};%\mathcal{E}_2

%H1
\Polygone{(-1,0,0.8)(-0.3,0,0.2) (0.3,0,0.2) (1,0,0.2) (-1,0,-1) }{5}{vlightgray};
\LLine{(1,0.2,-1) (-1,0,0.2) }; %H2
\LLine{(0.3,0,0.2) (1,0,0.8) };
\LLine{(-0.3,0,0.2) (-1,0,0.2) }; % bottom of \mathcal{E}_1
\node[inner sep=0.5pt] at (axis cs:-0.85,0,-0.7) {$H_1$};
\node[inner sep=0.5pt] at (axis cs:0.8,0.2,-0.7) {$H_2$};
\node[inner sep=0.5pt] at (axis cs:0,0,0.85) {$E_1$};
\node[inner sep=0.5pt] at (axis cs:-0.21,0,1.2) {$\tilde{e}_1$};
\node[inner sep=0.5pt] at (axis cs:0.2,0,1.3) {$\tilde{e}_1$};
\node[inner sep=0.5pt] at (axis cs:0,0,0.08) {$e_2$};
\node[inner sep=0.5pt] at (axis cs:-0.22,0,0.5) {$f_1$};
\node[inner sep=0.5pt] at (axis cs:-1,0,0.5) {$f_1$};
\node[inner sep=0.5pt] at (axis cs:-0.65,0,0.82) {$s_1$};
\node[inner sep=0.5pt] at (axis cs:-0.65,0,0.18) {$s_1$};
\node[inner sep=0.5pt] at (axis cs:0.65,0,0.82) {$s_2$};
\node[inner sep=0.5pt] at (axis cs:0.65,0,0.18) {$s_2$};
\node[inner sep=0.5pt] at (axis cs:0.22,0,0.5) {$f_2$};
\node[inner sep=0.5pt] at (axis cs:1,0,0.5) {$f_2$};
\node[inner sep=0.5pt] at (axis cs:-0.65,0,0.5) {$\mathcal{E}_1$};
\node[inner sep=0.5pt] at (axis cs:0.65,0,0.5) {$\mathcal{E}_2$};
\end{axis}
\end{tikzpicture}\vspace{-0.2cm}
\caption{The divisors and curves on $F_2$ respectively.} \label{F1F2PTP2}
\end{figure}
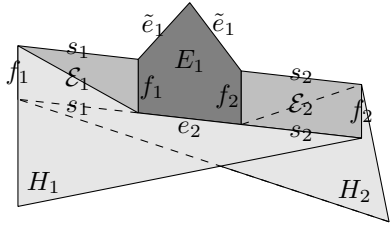

For $\{i,j\}=\{1,2\}$, we write $s_i=\mathcal{E}_i\cap H_j$, which is a section of $\mathcal{E}_i\to \tilde{\ell}_i$ of self-intersection $0$ and $f_i\subset \mathcal{E}_i$ a fibre; in particular $f_i\cdot \mathcal{E}_i=-1$. As $\delta_i=H_i\cap \mathcal{E}_i$ has bidegree $(1,1)$ in $\mathcal{E}_i\simeq \p^1\times \p^1$, we find $H_i\cdot f_i=H_i\cdot s_i=1$. Moreover, $H_j\cdot \delta_i=H_j\cdot f_i=1$, which implies that $H_j\cdot s_i=0$, since $\delta_i\equiv f_i+s_i$. 

We now denote by $e_2\subset F_2$ the strict transform of the unique line of $E_1\simeq \p^2$ that intersects $\tilde{\ell}_1$ and $\tilde{\ell}_2$. It then satisfies $e_2\cdot \mathcal{E}_1=e_2\cdot \mathcal{E}_2=1$ and since $(\eta_2)_*(e_2)=e_1$.

To compute $s_i\cdot \mathcal{E}_i$ and $e_2\cdot H_i$, we use $\eta_2^*(H_i)=H_i+\mathcal{E}_1+\mathcal{E}_2$. We have $s_i\cdot \eta_2^*(H_i)=(\eta_2)_*(s_i)\cdot H_i=\tilde{\ell}_i\cdot H_i=0$, which gives $s_i\cdot \mathcal{E}_i=-1$ since $s_i\cdot \mathcal{E}_j=0$ and $s_i\cdot H_i=1.$ Similarly, $e_2\cdot \eta_2^*(H_i)=(\eta_2)_*(e_2)\cdot H_i=e_1\cdot H_i=1$, which gives $e_2\cdot H_i=-1$.
We also compute $e_2\cdot E_1=e_2\cdot (\eta_2)^*(E_1)=e_1\cdot E_1=-1$.

For $\{i,j\}=\{1,2\}$, one finds the following intersection numbers: 
\[\begin{array}{|c|c|c|c|c|c|}
\hline
& s_i & s_j& f_i& f_j & e_2\\
\hline
H_i&1&0&1 & 1&-1  \\
E_1&1&1&0&0&-1  \\
\mathcal{E}_i& -1 &0&-1&0&1 \\
\hline
\end{array}\]
We now prove that the cone of effective curves of $F_2$ is polyhedral and generated by $s_1,s_2,f_1,f_2,e_{2}$, by proving that every irreducible curve $C$ is a non-negative linear combination of these.  If $C$ is contained in one of the surfaces $E_1$, $H_1$, $H_2$, $\mathcal{E}_1$ or $\mathcal{E}_2$ this is true as our curves include all extremal rays of these smooth toric surfaces. We may thus assume that $C$ is not contained in $E_1$, $\mathcal{E}_1$, $\mathcal{E}_2$ or $H_1$ or $H_2$. The curve $C$ is the strict transform of $\eta_2(C)\equiv a_1\tilde{\ell}_1+a_2\tilde{\ell}_2 +be_1$ for some $a_1,a_2,b\ge 0$. This gives $C\equiv a_1 s_1+a_2s_2+be_{2}+c_1 f_1+c_2f_2$ for some $c_1,c_2\in \Z$. For $\{i,j\}=\{1,2\}$ we have $0\le (\mathcal{E}_i+H_i)\cdot C=c_j$.

We have $E_2=\mathcal{E}_1+\mathcal{E}_2$ and $D=-\frac{3}{2}K_{F_2}-E_1-\frac{1}{2}E_2$. Writing $\eta=\eta_1\circ\eta_2$, we get $\eta^*(K_F)=-2H_1-2H_2-4E_1-4E_2$ and obtain
\[\begin{array}{ccccc}
K_{F_2}&=&\eta^*(K_F)+2E_1+E_2&=&-2H_1-2H_2-2E_1-3E_2,\\
D&=&-\frac{3}{2}\eta^*(K_F)-4E_1-2E_2&=& 3H_1+3H_2+2E_1+4E_2.\\
D+K_{F_2}&=&-\frac{1}{2}\eta^*(K_F)-2E_1-E_2&=&H_1+H_2+E_2\\
\end{array}\] This  implies that $D\cdot e_2=0,$ $D\cdot s_i=1$ and $D\cdot f_i=2$
 for all $i\in \{1,2\}$.

The divisor $D$ is thus nef. As $-\frac{1}{2}K_F$ is big (it is very ample), the divisor $-\frac{1}{2}\eta^*(K_F)=H_1+H_2+2E_1+2E_2$ is big. 
Since $D+\frac{1}{2}\eta^*(K_F)=2H_1+2H_2+2E_2$ is effective, this implies that $D$ is big. This achieves the proof of ~\ref{PTP2Dnefbig}.

 To prove~\ref{PTP2semiample}, we consider the linear system $\lvert D+K_{F_2}\rvert$ on $F_2$. 
Its elements are
 the strict transforms of hyperplane sections of $F\subset \p^2\times \p^2$ of bidegree $(1,1)$ having multiplicity $2$ at $p$ and passing through the two curves $\ell_1,\ell_2$. 
 This last condition is in fact implied by the first. The linear system then induces   the rational map $\tau\colon \p^2\times \p^2\dasharrow \p^3$ given by 
\[ \tau\colon ([x_0:x_1:x_2],[y_0:y_1:y_2])\mapsto [x_1y_0:x_2y_0:x_1y_2:x_2y_2],\]
whose image is contained in the smooth quadric $Q=\{[x_0:\cdots:x_3]\in \p^3\mid x_0x_3=x_1x_2\}\simeq \p^1\times \p^1$. 
The rational map to $\p^1\times \p^1$ is then given by \[ \tau'\colon ([x_0:x_1:x_2],[y_0:y_1:y_2])\mapsto ([x_1:x_2],[y_0:y_2])\]
and thus corresponds to the projections $\p^2\dasharrow \p^1$ away from $p_i$ on the two factors $i=1,2$ of $\p^2\times \p^2$. 

It remains to see that  $\tau_2=\tau'\circ \eta\colon F_2\to \p^1\times \p^1$  is a morphism, i.e.~that it is defined at every point $q\in F$. 
As $\tau$ is defined outside of $\ell_1\cup\ell_2$, we may assume that $\eta(q)\in \ell_1\cup\ell_2$. After composing with the automorphism $\sigma$, we may assume that $\eta(q)\in \ell_2$. If $\eta(q)\not=p$, then $\eta(q)$ belongs to the image of the open embedding $\A^3\hookrightarrow F, (r,s,t)\mapsto ([1:r:s],[-s-rt:t:1])$. The morphism $\eta$ corresponds on this chart to the blow-up of $r=s=0$, that is $\{((r,s,t),[u:v])\in \A^3\times \p^1\mid rv=su\}$. Hence, $\tau_2$ is locally given by $((r,s,t),[u:v])\mapsto ([u:v],[t:1])$ and is  then well-defined at every point. The remaining case is where $\eta(q)=p$, so $\eta_2(q)$ belongs to the surface $E_1\subset F_1$ isomorphic to $\p^2$. We study $\tau_1=\tau'\circ \eta_1\colon F_1\dasharrow \p^3$ in a neighbourhood of $E_1$. For this, we take the open embedding $\A^3\hookrightarrow F, (r,s,t)\mapsto ([1:r:s],[-r-st:1:t])$, and obtain that $\eta_1$ is the blow-up of the origin of $\A^3$ in this chart, corresponding to $\{((r,s,t),[u:v:w])\in \A^3\times \p^2\mid su=rv,sw=tv,rw=tu\}$. The rational map $\eta_1$ is locally given by $((r,s,t),[u:v:w])\mapsto ([u:v],[-u-tv:w])$. The divisor $E_1$ corresponds to $(0,0,0)\times \p^2$, so $\eta_1$ is defined at every point of $E_1$ except the two points $[0:0:1]$ and $[0:1:0]$. These are exactly the points where $(\eta_2)^{-1}$ is not an isomorphism. Using the symmetry, we may assume that $\eta_2(q)=E_1\cap \tilde{\ell}_2$, corresponding to $((0,0,0),[0:0:1])\in \A^3\times \p^2$. We choose the open embedding $\A^3\hookrightarrow \A^3\times \p^2$, $(a,b,c)\mapsto ((ac,bc,c),[a:b:1])$, and see $\eta_2$ as the blow-up of the line $a=b=0$. This latter is given by $\{((a,b,c),[\alpha:\beta])\in \A^3\times \p^1\mid a\beta=b\alpha\}$ and  $\tau_2$ is given by $((a,b,c),[\alpha:\beta])\mapsto ([\alpha:\beta],[-a-bc:1])$ and is thus defined at every point. This achieves the proof of \ref{PTP2semiample}.
\end{proof}

\section{Mori fibre spaces with general fibres isomorphic to $(\p^1)^3$ or $\P(T_{\p^2})$.}\label{Sec:MfsP13TP2}
In this section we prove Theorem~\ref{Thm:Dimension4}. Propositions \ref{prop:FibreP1P1P1goestoP3} and \ref{prop:FibrePTP2goesQP1} deal with Mori fibre spaces with general fibres isomorphic to $\p^1\times \p^1\times \p^1$ and  $\P(T_{\p^2})$ respectively.
In Lemma~\ref{Lemm:Balanced} we prove that the horizontal subvarieties of such Mori fibre spaces mark balanced subvarieties on the general fibres. 
\subsection{Balanced curves and divisors on a general fibre}
\begin{lemma}\label{Lemm:Balanced}
We write $F_1=\p^1\times \p^1\times \p^1$ and denote by $F_2\subset \p^2\times \p^2$ a smooth divisor of bidegree $(1,1)$. 
Let $i\in \{1,2\}$, let $\pi\colon X\to B$ be a Mori fibre space with the general fibre isomorphic to $F_i$ and let $Y\subsetneq X$ be a proper irreducible closed subset
with $\pi(Y)=B$. 
Then, the intersection of $Y$ with a general fibre of $\pi$ corresponds in $F_i$ to one of the following:
\begin{enumerate}
\item\label{DivBal}
A divisor linearly equivalent to $aK_{F_i}$ for some $a\in \Q$.
\item\label{CurvBal}
A curve $C\subset \p^1\times \p^1\times \p^1$ of tridegree $(a,a,a)$ $(i=1)$ or a curve $C\subset \p^2\times \p^2$ of bidegree $(a,a)$ $(i=2)$ for some integer $a\ge 1$. 
Moreover, the degree of $\pi_i\colon C\to \pi_i(C)$ is the same for all $i$, where $\pi_i$ is the projection onto each factor of $\p^1\times \p^1\times \p^1$ 
$($resp. of $\p^2\times \p^2)$.
\item
Finitely many points.
\end{enumerate}
\end{lemma}
\begin{proof} Assume that the intersection of $Y$ with the general fibre of $\pi$ has dimension at least $1$. Even if case \ref{DivBal} directly follows from the definition of Mori fibre space,
we will  
address case \ref{DivBal} and \ref{CurvBal} at the same time, as the proof is the same.

As general fibres are isomorphic to the smooth Fano threefold $F_i$, the generic  fibre of $\pi$ is a smooth Fano threefold $\mathcal{F}$ defined over the field $K=\C(B)$. 
We denote by $L=\overline{K}$ an algebraic closure of $K$,  and by $\mathcal{F}_L$ the geometric generic fibre. We now prove that $\mathcal{F}_L$ is isomorphic to $(F_i)_L$. 
By  \cite[Lemma~2.1]{Vial}, there is a field isomorphism $\C\iso L$ which induces an isomorphism  $\psi\colon F_i\iso\mathcal{F}_L$. 
In particular, the Picard rank of $\mathcal{F}_L$ is $4-i\in \{3,2\}$ and $K_{\mathcal{F}_L}^3=K_{F_i}^3=-48$. 
If $i=1$, this implies that $\mathcal{F}_L$ is either isomorphic to $(\p^1)^3$ or $\p^1\times \mathbb{F}_1$ \cite[Table~3]{MoriMukai1}, 
the second case being impossible because of the existence of the isomorphism $\psi$. If $i=2$, this implies that $\mathcal{F}_L$ is isomorphic to a hypersurface of bidegree $(1,1)$ in $\p^2\times \p^2$ \cite[Table~2]{MoriMukai1}. 
 
 Let $G=\mathrm{Gal}(L/K)$ be the Galois group. Then, we have an isomorphism $\mathrm{NS}(\mathcal{F})\simeq\mathrm{NS}
 (\mathcal{F}_L)^G$ \cite[Chapter~II, Proposition~4.3]{KolRC}. Since $\pi$ is a Mori fibration, the Picard rank of $\mathcal{F}$ is equal to $1$, so 
$\rk\mathrm{NS}(\mathcal{F})=\rk\mathrm{NS}(\mathcal{F}_L)^G=1$.
 As $\mathcal{F}_L$ is isomorphic to $(\p^1_L)^3$ or to a hypersurface of $(\p^2_L)^2$ of bidegree $(1,1)$, the Galois group has 
 to permute the factors in a transitive way. 
This implies that the cone of curves of $\mathcal{F}_L$ is of rank $1$, and that every curve on $\mathcal{F}_L$ corresponds in $(\p^1_L)^3$ or $(\p^2_L)^2$ to a curve of multidegree $(a,a,a)$ or $(a,a)$ for some integer $a\ge 1$; and the same holds for hypersurfaces.

We now consider the proper irreducible closed subset $Y\subsetneq X$. If the intersection of $Y$ with a general fibre $F$ is not finite, it is either a curve or a divisor. 
The generic fibre of $Y\to B$ is a curve or a divisor $\mathcal{Y}\subset \mathcal{F}$, corresponding in $\mathcal{F}_L$ to a balanced curve (of multidegree $(a,a,a)$ or $(a,a)$ for some integer $a\ge 1$, with projections on each factor of the same degree, as stated in \ref{CurvBal}) or a balanced divisor (equivalent to a multiple of the canonical divisor). Restricting to a general fibre, we obtain the cases~\ref{DivBal} and \ref{CurvBal}.
\end{proof}
\subsection{Mori fibre spaces with general fibres isomorphic to $(\p^1)^3$}
\begin{proposition}\label{prop:FibreP1P1P1goestoP3}
Let $\pi\colon X\to\p^1$ be a Mori fibre space whose general fibres are isomorphic to $(\p^1)^3$. Then, there is an $\Autz(X)$-equivariant commutative diagram
 \[\begin{tikzcd}
X\ar[dr,"\pi",swap]\ar[rr,"\varphi",dashed] &&Y\ar[dl,"\pi_Y"]   \\
& \p^1
\end{tikzcd}\]
where $\varphi$ is birational and $\pi_Y\colon Y\to \p^1$ is a Mori fibre space whose general fibres are isomorphic to $\p^3$.
\end{proposition}
\begin{proof}
By Theorem~\ref{Thm:ExistenceSection} there is a section $s\subset X$ of $\pi$ such that the following holds: the set
 \[\mathcal{S}=\Autz(X)\cdot s=\Autz(X)_{\p^1}\cdot s=(\Autz(X)_{\p^1})^{\circ}\cdot s\] is a proper closed subset of $X$ such that
 for each $b\in \p^1$, the fibre $\mathcal S_b=\pi^{-1}(b)\cap \mathcal{S}$ of $\pi|_{\mathcal{S}}\colon \mathcal{S}\to \p^1$ is equal to \[\pi^{-1}(b)\cap \mathcal{S}=(\Autz(X)_{\p^1})^{\circ}\cdot p,\] where $p\in s$ is the point such that $\pi(p)=b$.
 
 We first observe that $\dim(\mathcal{S})=3$ leads to a contradiction. For a general $b\in \p^1$, the fibre $\mathcal S_b$ is then a surface in 
 $\pi^{-1}(b)\simeq \p^1\times \p^1\times \p^1$ which has tridegree $(a,a,a)$ for some integer $a\ge 1$ (Lemma~\ref{Lemm:Balanced}). 
 As $\mathcal S_b$ is the orbit of a point by $(\Autz(X)_{\p^1})^{\circ}$, which acts on $\p^1\times \p^1\times \p^1$ via a subgroup 
 of $\PGL_2(\C)\times \PGL_2(\C)\times \PGL_2(\C)$, it is rational and therefore its canonical divisor is not pseudoeffective, whence $a=1$.
Hence the morphisms induced by projections $\mathcal S_t\to\p^1\times\p^1$ are birational, but not isomorphisms (this can be shown for instance by  computing  $K_{\mathcal S_t}^2$). The action of $(\Autz(X)_{\p^1})^{\circ}$ on $\p^1\times \p^1\times \p^1$ yields an action on $\p^1\times \p^1$, by Blanchard's Lemma~\ref{blanchard}, which cannot be transitive, contradicting the fact that $\mathcal S_t$ is an orbit.

We now study the case where $\dim(\mathcal{S})=2$. For a general $b\in \p^1$, the fibre $\mathcal S_b$ is then a curve in $\pi^{-1}(b)\simeq \p^1\times \p^1\times \p^1$ which has tridegree $(a,a,a)$ for some integer $a\ge 1$  (Lemma~\ref{Lemm:Balanced}). We then observe that $a=1$. Indeed, if $a\ge 2$, the projection onto each factor would be a finite ramified cover, so the action given by Blanchard's Lemma cannot be transitive, contradicting the fact that $\mathcal S_b$ is an orbit. 

We now denote by $\hat{X}\to X$ the blow-up of $\mathcal{S}$,
with exceptional divisor $E$.
Let $U\subset \p^1$ be the open set over which $\widehat X\to\p^1$ is smooth. We consider the divisor $-\frac{1}{2}K_{\hat{X}}-E$. By Lemma~\ref{Lem:CurveinP1P1P1link} we have $-\frac{1}{2}K_{\hat{X}}-E=K_{\hat{X}}+D$ where $D$ is relatively big and nef over $U$. 
By Theorem~\ref{fujino}, the divisor $-\frac{1}{2}K_{\hat{X}}-E$ induces a morphism over $U$ which, again by Lemma~\ref{Lem:CurveinP1P1P1link}, 
on each fibre contracts the strict transforms of 
the three  divisors of tridegree $(1,1,0),(1,0,1),(0,1,1)$ through $\mathcal{S}_b$.
This gives an $\Autz(X)$-equivariant birational morphism $\tilde X_U\to Y_U$, where $Y_U\to U$ has fibres isomorphic to $\p^3$ (again by Lemma~\ref{Lem:CurveinP1P1P1link}).
By Lemma~\ref{lem:cpt} we get a Mori fibre space $Y\to\p^1$ that is $\Autz(X)$-birational to $X$ over $\p^1$ and whose general fibres are isomorphic to $\p^3$.

It remains to study the case where $\dim(\mathcal{S})=1$, which implies that $\mathcal{S}=s$ is a section, pointwise fixed by $(\Autz(X)_{\p^1})^{\circ}$ and invariant by $\Autz(X)$.

We now denote by $\eta_1\colon X_1\to X$ the blow-up of $\mathcal{S}$ with exceptional divisor $E_1$. Let $H$ be an ample divisor of the form $-\eta_1^*K_X-E_1+\eta_1^*\pi^*\alpha$ where $\alpha$ is sufficiently ample on $\p^1$.
We consider the projective variety $\Chow_{1,1}(X_1)$
which parametrises the proper algebraic cycles of dimension $1$ and degree $1$ with respect to $H$. If $\hat{F}$ is a general fibre  of $\pi\circ \eta\colon X_1\to \p^1$,
by Lemma~\ref{Lem:PointpinP1P1P1linkII}\ref{P1P1P1p1Ample} the only $1$-cycles of degree~$1$ with respect to $H$ contained in $\hat{F}$ are the strict transforms of the three curves through $F\cap \mathcal{S}$ of tridegree $(1,0,0),(0,1,0),(0,0,1)$ and the lines in $E_1$. 
Moreover, if $\alpha$ is sufficiently ample, the only 1-cycles of degree $1$ with respect to $H$ are contained in fibres of $\pi\circ\eta_1$. We set $U=\{(x,[t])\in X_1\times \Chow_{1,1}(X_1)\vert\; x\in t\}$.

Therefore the image of the first projection $U\to X_1$ is a subvariety of $X_1$ 
of the form $Z\cup Z'\cup E $ where $Z$ is horizontal, and $Z'$ is vertical.

The subvariety $Z$ is such that its intersection with the general fibre of $\pi\circ\eta_1$ is the union of the strict transforms of the three curves through $F\cap \mathcal{S}$ of tridegree $(1,0,0),(0,1,0),(0,0,1)$. We set $\eta_2\colon X_2\to X_1$ the blow-up of $X_1$ along $Z$. We set $E_2$ the exceptional divisor and write again $E_1$ the strict transform of the exceptional divisor $E_1$ of $\eta_1$. Let $U\subset \p^1$ be the open set over which $ X_1\to\p^1$ is smooth. The divisor $-\frac{1}{2}K_{X_2}-E_1-\frac{1}{2}E_2$ is relatively semiample over $U$ by Theorem~\ref{fujino} as it is the sum of $K_{X_2}$ and of $-\frac{3}{2}K_{X_2}-E_1-\frac{1}{2}E_2$ which by Lemma~\ref{Lem:PointpinP1P1P1linkII} is relatively big and nef over $U$.  It induces then a morphism $(X_2)_U\to Y_U\to U$ where the general fibre of $Y_U\to U$ is $\p^3$. By Lemma~\ref{lem:cpt} we get a Mori fibre space $Y\to\p^1$ that is $\Autz(X)$-birational to $X$ over $\p^1$ and whose general fibre is $\p^3$ (using again Lemma~\ref{Lem:PointpinP1P1P1linkII}). 
\end{proof}

\subsection{Mori fibre spaces with general fibres isomorphic to $\p(T_{\p^2})$}

\begin{proposition}\label{prop:FibrePTP2goesQP1}
Let $\pi\colon X\to\p^1$ be a Mori fibre space whose general fibres  are isomorphic to a smooth hypersurface of $\p^2\times \p^2$ of bidegree $(1,1)$. Then, there is an $\Autz(X)$-equivariant commutative diagram
 \[\begin{tikzcd}
X\ar[dd,"\pi",swap]\ar[r,"\varphi",dashed] &Y\ar[d,"\pi_Y"]   \\
&B\ar[dl,"\pi_B"]\\
\p^1
\end{tikzcd}\]
where $\varphi$ is birational and $\pi_Y\colon Y\to B$ is a Mori fibre space whose general fibres are either isomorphic to $\p^1$ or  a smooth quadric $Q\subset \p^4$ $($in this latter case, $B\to \p^1$ is an isomorphism$)$.
\end{proposition}

\begin{proof}Let $F=\{([x_0:x_1:x_2],[y_0:y_1:y_2])\in \p^2\times \p^2 \mid \sum x_i y_i=0\}.$ The general fibres of $\pi$ are isomorphic to $F$ (Lemma~\ref{T4Aut}).

We apply Theorem~\ref{Thm:ExistenceSection} and obtain a section $s\subset X$ of $\pi$ such that the following holds: the set
 \[\mathcal{S}=\Autz(X)\cdot s=\Autz(X)_{\p^1}\cdot s=(\Autz(X)_{\p^1})^{\circ}\cdot s\] is a proper closed subset of $X$, and such that
 for each $b\in \p^1$, the fibre $\mathcal S_b=\pi^{-1}(b)\cap \mathcal{S}$ of $\pi|_{\mathcal{S}}\colon \mathcal{S}\to \p^1$ is equal to \[\pi^{-1}(b)\cap \mathcal{S}=(\Autz(X)_{\p^1})^{\circ}\cdot p,\] where $p\in s$ is the point such that $\pi(p)=b$.
 
 We first observe that $\dim(\mathcal{S})=3$ leads to a contradiction. For a general $b\in \p^1$, the fibre $\mathcal S_b$ is then a surface in $\pi^{-1}(b)\simeq F$ which is the intersection of $F$ with a hypersurface of  $\p^2\times \p^2$ of bidegree $(a,a)$ for some integer $a\ge 1$ (Lemma~\ref{Lemm:Balanced}). As $\mathcal S_b$ is the orbit of a point by $(\Autz(X)_{\p^1})^{\circ}$, its canonical divisor is not pseudoeffective, so $a=1$.
Hence the morphism induced by any projection $\mathcal S_t\to\p^2$ is birational, but not an isomorphism (this can be shown for instance by  computing  $K_{\mathcal S_t}^2$). The action of $(\Autz(X)_{\p^1})^{\circ}$ on $\p^2$ yields an action on $\p^2$, by Blanchard's Lemma~\ref{blanchard}, which cannot be transitive, contradicting the fact that $\mathcal S_t$ is an orbit.

We now study the case where $\dim(\mathcal{S})=2$. For a general $b\in \p^1$, the fibre $\mathcal S_b$ is then a curve $C_b$ in $\pi^{-1}(b)\simeq F\subset \p^2\times \p^2$ which has bidegree $(a,a)$ for some integer $a\ge 1$  (Lemma~\ref{Lemm:Balanced}). As $C_b$ is an orbit, it is smooth, and it is rational
 since $(\Autz(X)_{\p^1})^{\circ}$ acts on $\p^2\times \p^2$ via a subgroup of $\PGL_2(\C)\times \PGL_2(\C)$, by Blanchard's Lemma~\ref{blanchard}.
 The degree of $\pi_i\colon C_b\to \pi_i(C_b)$ is the same for $i=1,2$, where $\pi_i\colon \p^2\times \p^2$ is the projection on each factor (Lemma~\ref{Lemm:Balanced}). 
 So $\pi_i(C_b)$ is a curve for each $i$. Moreover, $C_b\to\pi_i(C_b)$ is an isomorphism, as otherwise the ramification points would be fixed 
 (using again Blanchard's Lemma, we have an action on each $\p^2$). There are thus two possibilities: either $a=1$, or $a=2$ and the projections to each factor are embeddings.

We now denote by $\hat{X}\to X$ the blow-up of $\mathcal{S}$,
with exceptional divisor $E$.
Let $U\subset \p^1$ be the open set over which $\widehat X\to\p^1$ is smooth. We consider the divisor $-\frac{1}{2}K_{\hat{X}}-E$. 
By Lemma~\ref{Lem:LineinPTP2link} if $a=1$ or Lemma~\ref{Lemm:Iso26} if $a=2$ we have $-\frac{1}{2}K_{\hat{X}}-E=K_{\hat{X}}+D$ where $D$ is relatively big and nef 
(in fact relatively ample) over $U$. By Theorem~\ref{fujino}, the divisor $-\frac{1}{2}K_{\hat{X}}-E$ induces a morphism over $U$.

If $a=1$, the morphism is birational, and contracts on each fibre the strict transforms of 
the two  divisors of bidegree $(0,1),(1,0)$ through $\mathcal{S}_b$ (Lemma~\ref{Lem:LineinPTP2link}). 
This gives an $\Autz(X)$-equivariant birational morphism $\tilde X_U\to Y_U$, where $Y_U\to U$ is a morphism with fibres isomorphic to a smooth quadric $Q\subset \p^4$ (again by Lemma~\ref{Lem:LineinPTP2link}).
By Lemma~\ref{lem:cpt}, we get a Mori fibre space $Y\to\p^1$ that is $\Autz(X)$-birational to $X$ over $\p^1$ and whose general fibres are isomorphic to $Q$. 
This concludes the proof in this case, and the variety $B$ in the statement is isomorphic to $\p^1$.

If $a=2$, the morphism induced by $-\frac{1}{2}K_{\hat{X}}-E$ is not birational. On each fibre, it gives a morphism to $\p^2$ with general fibres  isomorphic to $\p^1$ (Lemma~\ref{Lemm:Iso26}).
We then apply Lemmas~\ref{relBPF} and \ref{relMfs} to get an $\Autz(X)$-equivariant birational map  $\varphi\colon X\dasharrow Y$, where $Y\to B$ is a Mori fibre space with general fibres isomorphic to $\p^1$. We moreover obtain a morphism $B\to \p^1$ which makes the diagram commutative, as in the statement.

It remains to study the case where $\dim(\mathcal{S})=1$, which implies that $\mathcal{S}=s$ is a section, pointwise fixed by $(\Autz(X)_{\p^1})^{\circ}$ and invariant by $\Autz(X)$. The proof follows the same lines as the proof of Proposition~\ref{prop:FibreP1P1P1goestoP3}.

 We denote by $\eta_1\colon X_1\to X$ the blow-up of $\mathcal{S}$ with exceptional divisor $E_1$. Let $H$ be an ample divisor of the form $-\eta_1^*K_X-E_1+\eta_1^*\pi^*\alpha$ where $\alpha$ is sufficiently ample on $\p^1$.
We consider the projective variety $\Chow_{1,1}(X_1)$
which parametrises the proper algebraic cycles of dimension $1$ and degree $1$ with respect to $H$. If $\hat{F}$ is a general fibre of $\pi\circ \eta\colon X_1\to \p^1$,
by Lemma~\ref{Lem:PointpinPTP2linkII}\ref{PTP2p1Ample} the only $1$-cycles of degree $1$ with respect to $H$ contained in $\hat{F}$ are the strict transforms of the two curves through $F\cap \mathcal{S}$ of bidegree $(1,0),(0,1)$ and the lines in $E_1$. 
Moreover, if $\alpha$ is sufficiently ample, the only $1$-cycles of degree $1$ with respect to $H$ are contained in fibres of $\pi\circ\eta_1$. We set $U=\{(x,[t])\in X_1\times \Chow_{1,1}(X_1)\vert\; x\in t\}$.

Therefore the image of the first projection $U\to X_1$ is a subvariety of $X_1$ 
of the form $Z\cup Z'\cup E_1$ where $Z$ is horizontal and $Z'$ is vertical.  

The subvariety $Z$ is such that its intersection with the general fibre of $\pi\circ\eta_1$ is the union of the strict transforms of the two curves through $F\cap \mathcal{S}$ 
of bidegree $(1,0),(0,1)$. We set $\eta_2\colon X_2\to X_1$ the blow-up of $X_1$ along $Z$. We set $E_2$ the exceptional divisor and denote again by $E_1$ the strict transform of 
the exceptional divisor of $\eta_1$. Let $U\subset \p^1$ be the open set over which $X_2\to\p^1$ is smooth. 
The divisor $-\frac{1}{2}K_{X_2}-E_1-\frac{1}{2}E_2$ is relatively semiample over $U$ by Theorem~\ref{fujino} 
as it is the sum of $K_{X_2}$ and of $-\frac{3}{2}K_{X_2}-E_1-\frac{1}{2}E_2$ which by Lemma~\ref{Lem:PointpinPTP2linkII} is relatively big and nef over $U$. 
Therefore, the divisor $-\frac{1}{2}K_{X_2}-E_1-\frac{1}{2}E_2$  induces a morphism, which, on each fibre of $(X_2)_U\to U$, is a morphism to $\p^1\times\p^1$ with general fibre isomorphic to $\p^1$ (again by Lemma~\ref{Lem:PointpinPTP2linkII}). 
We then apply Lemmas~\ref{relBPF} and \ref{relMfs} to get an $\Autz(X)$-equivariant birational map  $\varphi\colon X\dasharrow Y$, where $Y\to B$ is a Mori fibre space with general fibres isomorphic to $\p^1$. We moreover obtain a morphism $B\to \p^1$ which makes the diagram commutative, as in the statement.
\end{proof}
We can now achieve this text by proving Theorem~\ref{Thm:Dimension4}.
\begin{proof}[Proof of Theorem~\ref{Thm:Dimension4}]
Let $\pi\colon X\to\p^1$ be a $\Q$-factorial terminal Mori fibre space such that a general fibre $F$ is a smooth threefold of Picard rank $\ge 2$, 
and such that $\Autz(X)$ is not trivial. The general fibres of $\pi$ belong to one of the families listed in Table~$\ref{3folds.MFS}$ (follows from \cite[Theorem~1.4]{CFST16}).

Suppose first that  \[\Autz(X)_{\p^1}=\{g\in \Autz(X)\mid \pi g=\pi\}\] is finite. In this case, Proposition~\ref{Prop:Torus} implies that $\Autz(X)$ is a torus of dimension~$1$ and provides an $\Autz(X)$-equivariant birational map $X\dasharrow \p^1\times Z$ where $Z$ is a terminal threefold. 
We are then in Case~\ref{Dim41} of Theorem~\ref{Thm:Dimension4}, with $Y=\p^1\times Z$ and $B=Z$.

We may now assume that $\Autz(X)_{\p^1}$ is of positive dimension. 
Let us write $k=\max\{\dim((\Autz(X)_{\p^1})^\circ\cdot x)\mid x\in X\}>0$ for the maximal dimension of an orbit of $(\Autz(X)_{\p^1})^\circ$ 
(equivalently of $\Autz(X)_{\p^1}$)  on $X$.

 If $k=1$, then  Proposition~\ref{Prop:Smallorbit} gives an $\Autz(X)$-equivariant birational map $X\dasharrow Y$, 
where $Y\to B$ is a Mori fibre space, with $\dim B=3$. In this case, the general fibres are isomorphic to $\p^1$, so we are in Case~\ref{Dim41} of Theorem~\ref{Thm:Dimension4}.

We now assume that $k\ge 2$. By Lemma~\ref{lemm:DimMaxopen}, a general fibre $F$ of $\pi\colon X\to \p^1$ then satisfies  
$\dim \Autz(F)\ge 2$. By Proposition~\ref{prop:AutFiniteTable}, the general fibres belong to  the families \hyperlink{T3}{$3$}, \hyperlink{T4}{$4$}, \hyperlink{T6}{$6$} or \hyperlink{T7}{$7$} of Table~$\ref{3folds.MFS}$. 
 We do a case-by-case analysis.

If the general fibres belong to the family  \hyperlink{T3}{$3$}, they are all such that $\Autz(F)\simeq \PGL_2(\C)$ and are isomorphic to the 
blow-up of the quadric $Q\subset \p^4$ of equation $x_0x_4 - 4x_1x_3 + 3x_2^2=0$ along the image of the Veronese embedding of degree $4$ of $\p^1$ (Lemma~\ref{T3Aut}). 
We are then in Case~\ref{Dim43} of Theorem~\ref{Thm:Dimension4}.

If the general fibres belong to the family  \hyperlink{T6}{$6$}, they are all such that $\Autz(F)\simeq \PGL_2(\C)$ and are isomorphic to
\[\left\{(x,y,z)\in (\p^2)^3 \left| 
 \sum_{i=0}^2 x_i y_i= \sum_{i=0}^2 x_i z_i=\sum_{i=0}^2 y_i z_i=0\right\}\right..\]
 (Lemma~\ref{T6Aut}). We are then in Case~\ref{Dim44} of Theorem~\ref{Thm:Dimension4}.

If the general fibres belong to the families \hyperlink{T4}{$4$} and  \hyperlink{T7}{$7$}, 
Proposition~\ref{prop:FibrePTP2goesQP1} and Proposition~\ref{prop:FibreP1P1P1goestoP3} provide an $\Autz(X)$-equivariant birational map $X\dasharrow Y$ where $Y\to B$ is 
a Mori fibre space whose general fibres are either isomorphic to $\p^1$, $\p^3$ or  a smooth quadric $Q\subset \p^4$. We are in Case~\ref{Dim41} of Theorem~\ref{Thm:Dimension4}.
\end{proof}

\bibliographystyle{alpha}
\bibliography{biblio}

\end{document}